\documentclass[abstract=true,a4paper]{scrartcl}
\usepackage[utf8]{inputenc}
\usepackage{tikz,amsmath,amsfonts,amsthm,amsopn,mathrsfs,amsthm}
\usepackage[noadjust]{cite}
\usepackage{epsfig,verbatim}
\usepackage{bm}
\usepackage{dsfont}
\providecommand {\norm}[1] {\lVert#1\rVert}

\providecommand {\set}[1]{\lbrace #1 \rbrace}

\newcommand {\one} e
\newcommand {\bn} {\ensuremath{\mathbb{N}}}

\newcommand {\br} {\ensuremath{\mathbb{R}}}

{

\DeclareMathOperator{\supp}{supp}

\newtheorem{prop}{Proposition}[section]
\newtheorem{cor}[prop]{Corollary}
\newtheorem{thm}[prop]{Theorem}

\newtheorem{lem}[prop]{Lemma}

\theoremstyle{definition}
\newtheorem{defn}[prop]{Definition}

\theoremstyle{remark}
\newtheorem*{rem*}{Remark}
\newtheorem*{rems*}{Remarks}
\newtheorem{rem}[prop]{Remark}

\newtheorem{ex}[prop]{Example}

\usepackage{enumitem}

\usepackage[top=2cm,bottom=3.5cm]{geometry}

\let\oldbibliography\thebibliography
\renewcommand{\thebibliography}[1]{%
  \oldbibliography{#1}%
  \setlength{\itemsep}{1.3pt}%
}

\numberwithin{equation}{section}

\usepackage[ddmmyyyy,hhmmss]{datetime}
\newdateformat{daymonthyeardate}{%
  \THEDAY.\twodigit{\THEMONTH}.\THEYEAR}
\newdateformat{TitleDate}{%
  \THEDAY.~\monthname[\THEMONTH]~\THEYEAR}
\usepackage{fancyhdr}
\usepackage{lastpage}
\usepackage{etoolbox}
\pagestyle{fancy}
\setlength{\footskip}{1.4cm}
\fancyhf{} 

\fancyfoot{}
\fancyfoot[R]{\footnotesize{Page \thepage~of \pageref{LastPage}}}

\usepackage{mathtools}
\usepackage{amscd}

\providecommand{\pdiffs}[2] {\frac 1 {#1} - \frac 1 {#2}}
\newcommand{\ball}{\mathcal{B}}
\DeclareMathOperator{\argmin}{argmin}
\DeclareMathOperator{\enc}{Enc}
\DeclareMathOperator{\codec}{Codecs}

\usepackage{dsfont}
\usepackage{xparse}
\usepackage{todonotes}
\usepackage{amssymb}
\usepackage{graphicx}

\let\emptyset\varnothing

\newcommand*{\bigs}[1]{\vcenter{\hbox{\scalebox{1.07}{\ensuremath#1}}}}

\newcommand{\signalClass}{\mathcal{S}}
\newcommand{\hilbert}{\mathcal{H}}
\newcommand{\banach}{\mathbf{X}}
\newcommand{\banachOne}{\banach}
\newcommand{\otherBanach}{\mathbf{Y}}
\newcommand{\banachTwo}{\otherBanach}
\newcommand{\banachThree}{\mathbf{Z}}
\newcommand{\FirstN}[1]{[#1]}

\newcommand{\powerset}[1]{2^{#1}}
\newcommand{\waveletT}{T}

\makeatletter
\newcommand{\specialcell}[1]{\ifmeasuring@#1\else\omit$\displaystyle#1$\ignorespaces\fi}
\makeatother

\DeclareMathOperator{\compressible}{Comp}
\DeclareMathOperator{\dist}{dist}

\NewDocumentCommand\endec{O{R}D<>{\signalClass}G{\hilbert}}{\enc^{#1}_{#2,#3}}
\NewDocumentCommand\codecs{D<>{\signalClass}G{\hilbert}}{\codec_{#1,#2}}
\NewDocumentCommand\distortion{D<>{\signalClass}G{\hilbert}}{\delta_{#1,#2}}
\NewDocumentCommand\comp{mG{\hilbert}}{\compressible_{#2}^{#1}}
\NewDocumentCommand\optRate{D<>{\signalClass}G{\hilbert}}{s^\ast_{#2} \big( #1 \big)}
\NewDocumentCommand\optRateSmall{D<>{\signalClass}G{\hilbert}}{s^\ast_{#2} \bigs( #1 \bigs)}
\NewDocumentCommand\approxClass{mD<>{\signalClass}G{\hilbert}}{\mathcal{A}_{#2,#3}^{#1}}
\NewDocumentCommand\mixSpace{O{\partition}G{\infty}G{\alpha}D<>{p}}{\ell^{#4,#2}_{#1,#3}}
\NewDocumentCommand\generalSpace{G{\theta}G{\infty}G{\alpha}D<>{p}}{\ell^{#4,#2}_{\partition,#3,#1}}
\NewDocumentCommand\sequenceSpaceSignalClass{D<>{\partition}G{q}}{\signalClass_{#1,\alpha}^{p,#2}}
\NewDocumentCommand\generalSignalClass{G{\theta}G{q}G{\alpha}D<>{p}}{\signalClass_{\partition,#3,#1}^{#4,#2}}
\NewDocumentCommand\sequenceProductMeasure{D<>{\partition}G{q}}{\prob_{#1,\alpha}^{p,#2}}
\NewDocumentCommand\ContinuousBesov{G{\alpha}D<>{p}G{q}}{\mathbf{B}^{#1}_{#2,#3}}
\NewDocumentCommand\LeopoldUnweighted{G{p}G{q}}{\ell^{#2}(\ell_{\LeopoldM_j}^{#1})}
\NewDocumentCommand\LeopoldUnweightedReal{G{p}G{q}}{\ell^{#2}(\ell_{\LeopoldM_j}^{#1}; \R)}
\NewDocumentCommand\LeopoldSpace{G{p}G{q}}{\ell^{#2}(\beta_j \, \ell_{\LeopoldM_j}^{#1})}
\NewDocumentCommand\LeopoldReal{G{p}G{q}}{\ell^{#2}(\beta_j \, \ell_{\LeopoldM_j}^{#1};\R)}

\NewDocumentCommand\generalProductMeasure{G{\theta}G{\infty}D<>{p}}{\prob_{\partition,\alpha,#1}^{#3,#2}}

\newcommand{\N}{\mathbb{N}}
\newcommand{\Z}{\mathbb{Z}}
\newcommand{\CC}{\mathbb{C}}

\newcommand{\B}{\mathbb{B}}
\newcommand{\R}{\mathbb{R}}
\newcommand{\X}{\mathbf{X}}
\newcommand{\Y}{\mathbf{Y}}
\newcommand{\code}{\mathcal{C}}
\newcommand{\eps}{\varepsilon}

\newcommand{\identity}{\mathrm{id}}
\newcommand{\prob}{\mathbb{P}}
\newcommand{\xvec}{\mathbf{x}}
\newcommand{\yvec}{\mathbf{y}}

\newcommand{\bvec}{\mathbf{b}}
\newcommand{\cvec}{\mathbf{c}}
\newcommand{\uvec}{\mathbf{u}}

\newcommand{\evec}{\mathbf{e}}
\newcommand{\indexSet}{\mathcal{I}}
\newcommand{\calB}{\mathcal{B}}
\newcommand{\calD}{\mathcal{D}}
\newcommand{\calT}{\mathcal{T}}
\newcommand{\Schwartz}{\mathscr{S}}
\newcommand{\Fourier}{\mathcal{F}}
\newcommand{\interior}{\mathrm{int}}
\newcommand{\ext}{\mathrm{ext}}
\newcommand{\LeopoldM}{N}
\newcommand{\BesovSmoothness}{\tau}
\newcommand{\indicator}{\mathds{1}}
\renewcommand{\P}{\mathbb{P}}
\newcommand{\extension}{\mathscr{E}}
\newcommand{\partition}{\mathscr{P}}
\newcommand{\gc}{c} 

\usepackage{stackengine}
\newcommand{\myhat}[1]{\mathop{}\!\!\ensurestackMath{\stackon[-1.25ex]{#1}{\smash{\widehat{\hphantom{#1}}}}}}

\usepackage{authblk}

\theoremstyle{definition}

\begin{document}
\title{Phase Transitions in Rate Distortion Theory and Deep Learning}
\author[$\dagger$,$\ddagger$,$\flat$]{Philipp Grohs}
\author[$\dagger$]{Andreas Klotz\thanks{AK acknowledges funding from the FWF projects I 3403 and P 31887.}}
\author[$\dagger$,$\S$]{Felix Voigtlaender}

\affil[$\,$]{\footnotesize \texttt{$\{$philipp.grohs, andreas.klotz, felix.voigtlaender$\}$@univie.ac.at}
\vspace*{0.3cm}}

\affil[$\dagger$]{\footnotesize
Faculty of Mathematics,
University of Vienna, Vienna, Austria.
\vspace*{0.1cm}}

\affil[$\ddagger$]{\footnotesize
Research Platform Data Science,
University of Vienna, Vienna, Austria
\vspace*{0.1cm}}

\affil[$\flat$]{\footnotesize
Johann Radon Institute for Computational and Applied Mathematics,\newline
Austrian Academy of Sciences,
Linz, Austria.
\vspace*{0.1cm}}

\affil[$\S$]{\footnotesize
Catholic University of Eichstätt--Ingolstadt,
Eichstätt, Germany}

\date{\vspace*{-0.3cm}\TitleDate\today\vspace*{-0.8cm}}
\maketitle

\begin{abstract}
  Rate distortion theory is concerned with optimally encoding a given signal class $\signalClass$
  using a budget of $R$ bits, as $R \to \infty$.
  We say that $\signalClass$ \emph{can be compressed at rate} $s$ if we can achieve an error
  of at most $\mathcal{O}(R^{-s})$ for encoding the given signal class;
  the supremal compression rate is denoted by $s^\ast (\signalClass)$.
  Given a fixed coding scheme, there usually are \emph{some} elements of $\signalClass$
  that are compressed at a higher rate than $s^\ast (\signalClass)$ by the given coding scheme;
  in this paper, we study the size of this set of signals.
  We show that for certain ``nice'' signal classes $\signalClass$,
  a \emph{phase transition} occurs:
  We construct a probability measure $\prob$ on $\signalClass$ such that
  for  \emph{every} coding scheme $\code$ and any $s > s^\ast(\signalClass)$,
  the set of signals encoded with error $\mathcal{O}(R^{-s})$ by $\code$ forms a $\prob$-null-set.
  In particular our results apply to all unit balls in Besov and Sobolev spaces
  that embed compactly into $L^2 (\Omega)$ for a bounded Lipschitz domain $\Omega$.
  As an application, we show that several existing sharpness results
  concerning function approximation using deep neural networks are in fact \emph{generically sharp}.

  In addition we provide quantitative and non-asymptotic bounds on the probability
  that a random $f\in \signalClass$ can be encoded to within accuracy $\eps$ using $R$ bits.
  This result is subsequently applied to the problem of approximately representing $f\in \signalClass$
  to within accuracy $\varepsilon$ by a (quantized) neural network constrained
  to have at most $W$ nonzero nodes that can be produced by any numerical ``learning'' procedure.
  We show that for any $s > s^\ast(\signalClass)$ there are constants $c,C$ such that,
  no matter how we choose the ``learning'' procedure, the probability of success is bounded
  from above by $\min \big\{1, 2^{C\cdot W \lceil \log_2 (1+W) \rceil^2 - c\cdot \varepsilon^{-1/s}} \big\}$.
\end{abstract}

\noindent
\textbf{Keywords:}
Rate distortion theory,
Phase transition,
Approximation rates,
Besov spaces,
Sobolev spaces,
Neural network approximation.

\vspace{0.2cm}

\noindent
\textbf{MSC (2010) classification:}
41A46, 28C20, 68P30.


\section{Introduction}
\label{sec:intro}

Let $\signalClass$ be a \emph{signal class}, that is, a relatively compact subset
of a Banach space $(\banach, \| \cdot \|_{\banach})$.
Rate distortion theory is concerned with the question of how well the elements
of $\signalClass$ can be encoded using a prescribed number $R$ of bits.
In many cases of interest, the best achievable coding error scales like $R^{-s^\ast}$,
where $s^\ast$ is the \emph{optimal compression rate} of the signal class $\signalClass$.
We show that a phase transition occurs: the set of elements $\xvec \in \signalClass$
that can be encoded using a \emph{strictly larger} exponent than $s^\ast$ is thin;
precisely, it is a null-set with respect to a suitable probability measure $\P$.
Crucially, the measure $\P$ is \emph{independent} of the chosen coding scheme.

In order to make these results more rigorous, let us state the needed notions of rate-distortion
theory, see also \cite{berger2003rate,boelcskeiNeural,GrohsNNApproximationTheory,grohs2015optimally}.

\subsection{A crash course in rate distortion theory}%
\label{sub:RateDistortionCrashCourse}

To formalize the notion of encoding a signal class $\signalClass \subset \banach$,
we define the set $\endec{\banach}$ of encoding/decoding pairs
$(E, D)$ of \emph{code-length} $R \in \bn$ as
\[
  \endec{\banach}
  :=\left\{
      (E,D)
      \quad : \quad
      E: \signalClass \to \{0,1\}^R \quad \mbox{ and }\quad D : \{0,1\}^R \to \banach
    \right\}.
\]
We are interested in choosing $(E,D) \in \endec{\banach}$ such as to minimize
the \emph{(maximal) distortion}
\(
  \distortion{\banach} (E,D)
  := \sup_{\xvec \in \signalClass} \| \xvec - D (E (\xvec)) \|_{\banach}
  .
\)

The intuition behind these definitions is that the encoder $E$ converts any signal
$\xvec \in \signalClass$ into a bitstream of code-length $R$ (i.e., consisting of $R$ bits),
while the decoder $D$ produces from a given bitstream ${b \in \{ 0, 1 \}^R}$ a signal $D(b) \in \banach$.
The goal of \emph{rate distortion theory} is to determine the minimal
distortion that can be achieved by any encoder/decoder pair of code-length $R \in \N$.
Typical results concerning the relation between code-length and distortion
are formulated in an asymptotic sense:
One assumes that for every code-length $R \in \N$, one is given an encoding/decoding pair
${(E_R,D_R) \in \endec{\banach}}$, and then studies the asymptotic behaviour
of the corresponding distortion $\distortion{\banach}(E_R, D_R)$ as $R \to \infty$.

We refer to a sequence $\big( (E_R, D_R) \big)_{R \in \N}$ of encoding/decoding pairs
as a \emph{codec}, so that the set of all codecs is
\[
  \codecs{\banach} := \prod_{R \in \N}
                        \endec{\banach} .
\]
For a given signal class $\signalClass$ in a Banach space $\banach$,
it is of great interest to find an asymptotically optimal codec; that is,
a sequence $\big( (E_R,D_R) \big)_{R\in \bn} \in \codecs{\banach}$ such that the asymptotic decay
of $\big(\distortion{\banach} (E_R,D_R) \big)_{R\in \bn}$ is, in a sense, maximal.
To formalize this, for each $s \in [0,\infty)$ define the class of subsets of $\banach$ that
\emph{admit compression rate $s$} as
\[
  \comp{s}{\banach}
  := \! \left\{
          \signalClass \subset \banach
          \quad \!\!\!\! \colon \!\!\!\! \quad
          \exists \, \big( (E_R, D_R) \big)_{R \in \N} \in \codecs{\banach} : \,\,
            \sup_{R \in \N}
              \big( R^{s} \cdot \distortion{\banach} (E_R, D_R) \big)
            \!<\! \infty
        \right\} \! .
\]
For a given (bounded) signal class $\signalClass \subset \banach$ we aim to determine the
\emph{optimal compression rate} for $\signalClass$ in $\banach$, that is
\begin{equation}
  \optRateSmall{\banach} := \sup \big\{
                                   s \in [0,\infty)
                                   \colon
                                   \signalClass \in \comp{s}{\banach}
                                 \big\}
                         \in [0,\infty].
  \label{eq:OptimalCompressionRate}
\end{equation}
Although the calculation of the quantity $\optRateSmall{\banach}$ may appear daunting
for a given signal class $\signalClass$, there exists in fact a large body of literature
addressing this topic.
A landmark result in this area states that the JPEG2000 compression standard
represents an optimal codec for the compression of piecewise smooth signals \cite{mallat1999wavelet}.
This optimality is typically stated more generally for the signal class
$\signalClass = \ball \big( 0, 1; B_{p,q}^\alpha (\Omega) \big)$, the unit ball in the Besov space
$B_{p,q}^\alpha(\Omega)$, considered as a subset of $\banach = \hilbert = L^2(\Omega)$,
for ``sufficiently nice'' bounded domains $\Omega \subset \R^d$; see \cite{DeVoreConstructiveApproximation}.

For a codec $\code = \big( (E_R, D_R) \big)_{R \in \N} \in \codecs{\banach}$,
instead of considering the maximal distortion of $\code$ over the \emph{entire} signal class $\signalClass$,
one can also measure the approximation rate that the codec $\code$ achieves
for each \emph{individual} $\xvec \in \signalClass$.
Precisely, the \emph{class of elements with compression rate $s$ under $\code$} is
\begin{equation}
  \approxClass{s}{\banach} (\code)
  := \Big\{
       \xvec \in \signalClass
       \quad \colon \quad
       \sup_{R \in \N}
         \big[ R^s \cdot \big\| \xvec - D_R(E_R (\xvec)) \big\|_{\banach} \big]
       < \infty
     \Big\} .
  \label{eq:ApproximationClass}
\end{equation}
If the signal class $\signalClass$ is ``sufficiently regular''---for instance if $\signalClass$
is compact and convex---then one can prove (see Proposition~\ref{prop:SingleHardToEncodeElement})
that the following \emph{dichotomy} is valid:
\begin{equation}
  \begin{aligned}
    s < \optRateSmall{\banach} & \Longrightarrow \exists \, \code \in \codecs{\banach}
                                                   \forall \, \xvec \in \signalClass \mkern-18mu && : \quad
                                                     \xvec \in \approxClass{s}{\banach} (\code) , \\
    s > \optRateSmall{\banach} & \Longrightarrow \forall \, \code \in \codecs{\banach}
                                                   \exists \, \xvec^\ast \in \signalClass \mkern-18mu && : \quad
                                                     \xvec^\ast \notin \approxClass{s}{\banach} (\code) .
  \end{aligned}
  \label{eq:dichotomy}
\end{equation}
Thus, all signals in $\signalClass$ can be approximated at any compression rate
lower than the optimal rate for $\signalClass$ using a \emph{common} codec.
Furthermore, for any approximation rate $s$ larger than the optimal rate for $\signalClass$,
and for any codec $\code$, there exists some ${\xvec^\ast = \xvec^\ast (s, \code) \in \signalClass}$
that is \emph{not} compressed at rate $s$ by $\code$.

\begin{rem*}[Encoding/decoding schemes vs.~discretization maps]
  As the above considerations suggest, the crucial quantity
  for our investigations are not the encoding/decoding pairs
  ${(E,D) \in \endec{\banach}}$, but the \emph{distortion} they cause for each
  $\xvec \in \signalClass$.
  Therefore, we could equally well restrict our attention to the \emph{discretization map}
  $D \circ E : \signalClass \to \banach$, which has the crucial property
  ${|\mathrm{range} (D \circ E)| \leq 2^R}$.
  Conversely, given any (discretization) map $\Delta : \signalClass \to \banach$ with
  $|\mathrm{range}(\Delta)| \leq 2^R$, one can construct an encoding/decoding pair ${(E,D) \in \endec{\banach}}$,
  by choosing a surjection ${D : \{ 0,1 \}^R \to \mathrm{range}(\Delta)}$, and then setting
  \[
    E : \signalClass \to \{ 0,1 \}^R,
        \xvec \mapsto \argmin_{c \in \{ 0,1 \}^R} \| \xvec - D(c) \|_{\banach}
    \quad ,
  \]
  which ensures that $\| \xvec - D(E(\xvec)) \|_{\banach} \leq \| \xvec - \Delta(\xvec) \|_{\banach}$
  for all $\xvec \in \signalClass$.
  Thus, all our results could equally well be rephrased in terms of such discretization maps
  rather than in terms of encoding/decoding pairs.
  For more details on this connection, see also Lemma~\ref{lem:EntropyAndDistortion}.
\end{rem*}

\subsection{Our contributions}%
\label{sub:contributions}
\subsubsection{Phase Transition}

We improve on the dichotomy (\ref{eq:dichotomy}) by measuring the size of the class
$\approxClass{s}{\banach} (\code)$ of elements with compression rate $s$ under the codec $\code$.
Then a phase transition occurs: the class of elements that can not be encoded
at a ``larger than optimal'' rate is \emph{generic}.
We prove this when the signal class is a ball in a Besov- or Sobolev space,
as long as this ball forms a compact subset of $\hilbert = L^2(\Omega)$
for a bounded Lipschitz domain $\Omega \subset \R^d$.

More precisely, for each such signal class $\signalClass$, we construct a probability measure $\prob$
on $\signalClass$ such that the compressibility exhibits a \emph{phase transition}
as in the following definition.

\begin{defn}\label{defn:IntroPhaseTransitionMeasure}
  A Borel probability measure $\P$ on a subset $\signalClass$ of a Hilbert space $\hilbert$
  \emph{exhibits a compressibility phase transition} if it satisfies the following:
  \begin{equation}
    \begin {aligned}
      \text{if } s < \optRateSmall & \text{ then } \exists \, \code \in \codecs :
                                           \prob \big( \approxClass{s} (\code) \big) = 1; \\
      \text{if } s > \optRateSmall & \text{ then } \forall \, \code \in \codecs:
                                           \prob^\ast \big( \approxClass{s} (\code) \big) = 0.
    \end{aligned}
    \label{eq:PhaseTransition}
  \end{equation}
  Here $\prob^\ast$ is the \emph{outer measure} corresponding to $\prob$,
  defined in Equation~\eqref{eq:OuterMeasureDefinition} below.
\end{defn}

The first implication in \eqref{eq:PhaseTransition} is always satisfied,
as a consequence of \eqref{eq:dichotomy}.
The second part of \eqref{eq:PhaseTransition} states that
for any $s > \optRateSmall$ and any codec $\code$,
almost every $\xvec \in \signalClass$ \emph{cannot} be compressed by $\code$ at rate $s$.
In other words, whenever $\P$ exhibits a compressibility phase transition on $\signalClass$,
\emph{the property of not being compressible at a ``larger than optimal'' rate is a generic property}.

\begin{rem}[Universality in Definition~\ref{defn:IntroPhaseTransitionMeasure}]\label{rem:IntroUniversality}
  Note that the measure $\P$ in Definition~\ref{defn:IntroPhaseTransitionMeasure}
  is required to satisfy the second property in \eqref{eq:PhaseTransition}
  universally for \emph{any} choice of codec $\code$.

  In fact, if $\P$ would be allowed to depend on $\code$, one could simply choose
  $\P = \delta_{\xvec}$, where $\xvec = \xvec(\code, s) \in \signalClass$ is a single element
  that is not approximated at rate $s$ by $\code$; for $s > \optRateSmall$ such an element exists
  under mild assumptions on $\signalClass$.
  In contrast, the measure $\P$ in Definition~\ref{defn:IntroPhaseTransitionMeasure} satisfies $\P (\{ \xvec \}) = 0$ for each $\xvec \in \signalClass$,
  as can be seen by taking ${\code = \bigl((E_R,D_R)\bigr)_{R \in \N}}$ with
  $D_R : \{ 0,1 \}^R \to \signalClass, c \mapsto \xvec$,
  so that ${\approxClass{s}(\code) = \{ \xvec \}}$ for all ${s > 0}$.
  This shows, in particular, that any probability measure $\P$ exhibiting a compressibility phase transition
  is \emph{atom free}, so that $\P(M)=0$ for any countable set $M$.
\end{rem}

Our first main result establishes the existence of critical measures for all Sobolev-
and Besov balls (denoted $\ball ( 0 , 1; W^{k,p}(\Omega ; \R))$,
resp.~$\ball ( 0 , 1; B_{p,q}^{\BesovSmoothness}(\Omega;\R)) )$; see Appendix~\ref{sec:BesovReview})
that are compact subsets of $L^2 (\Omega)$:

\begin{thm}\label{thm:IntroductionSobolevBesovTransition}
  Let $\emptyset \neq \Omega \subset \R^d$ be a bounded Lipschitz domain.
  Consider either of the following two settings:
  \setlength{\leftmargini}{0.4cm}
  \begin{itemize}
    \item $\signalClass \!:=\! \ball \big( 0, 1; B_{p,q}^{\BesovSmoothness} (\Omega; \R) \big)$
          and $s^\ast := \! \frac{\BesovSmoothness}{d}$, where $p,q \in \!(0,\infty]$
          and $\BesovSmoothness \in \R$
          with $\strut {\BesovSmoothness \!>\! d \cdot (\frac{1}{p} \!-\! \frac{1}{2})_{+}}$, or

    \item $\signalClass := \ball\bigl(0, 1; W^{k,p}(\Omega)\bigr)$
          and $s^\ast := \frac{k}{d}$, where $p \in [1,\infty]$
          and $k \in \N$ with ${k > d \cdot (\frac{1}{p} - \frac{1}{2})_+}$.
  \end{itemize}

  In either case, $\optRateSmall{L^2(\Omega)} = s^\ast$, and there is a Borel probability measure
  $\P$ on $\signalClass$ that exhibits a compressibility phase transition as in
  Definition~\ref{defn:IntroPhaseTransitionMeasure}.
\end{thm}

\begin{proof}
  This follows from Theorems~\ref{thm:mainbesovresult}, \ref{thm:SobolevPhaseTransition},
  and \ref{thm:DichotomyBanach}.
\end{proof}

Since Remark~\ref{rem:IntroUniversality} shows that the measure $\P$ from the preceding theorem
satisfies $\P(M) = 0$ for each countable set $M \subset \signalClass$,
we get the following strengthening of the dichotomy \eqref{eq:dichotomy}.

\begin{cor}\label{cor:IntroductionUncountablyManyBadSignals}
  Under the assumptions of Theorem~\ref{thm:IntroductionSobolevBesovTransition}, for each codec
  $\code \in \codecs{L^2(\Omega)}$ the set
  $\signalClass\setminus \bigcup_{s > s^\ast} \approxClass{s}{L^2(\Omega)} (\code)$,
  which consists of all signals that can \emph{not} be encoded by $\code$
  at compression rate $s$ for some $s > s^\ast$, is uncountable.
\end{cor}

In words, Corollary~\ref{cor:IntroductionUncountablyManyBadSignals} states that for every codec
the set of signals in $\signalClass$  that can not be approximated at any compression rate
larger than the optimal rate for $\signalClass$ is uncountable.
In contrast, previous results (such as Proposition~\ref{prop:SingleHardToEncodeElement})
only state the existence of a single such ``badly approximable'' signal.

\subsubsection{Quantitative lower bounds}

As a quantitative version of Theorem~\ref{thm:IntroductionSobolevBesovTransition}, we show that
if one randomly chooses a function $f \sim \P$ according to the probability measure $\P$
constructed in (the proof of) Theorem~\ref{thm:IntroductionSobolevBesovTransition},
one can precisely bound the probability that a given encoding/decoding pair $(E_R,D_R)$
of code-length $R$ achieves a given error $\eps$ for $f$.
To underline a probabilistic interpretation, we define, for any property $\tau$ of elements $f \in \signalClass$
\begin{equation}
  \mathrm{Pr}(f \text{ satisfies } \tau)
  := \P^\ast (\{ f \in \signalClass \colon f \text{ satisfies } \tau \}),
  \label{eq:SpecialProbabilityNotation}
\end{equation}
where $\P^\ast$ denotes the outer measure induced by $\P$.

\begin{thm}\label{thm:IntroductionQuantitativeTheorem}
  Let $\signalClass$ and $s^\ast$ as in Theorem~\ref{thm:IntroductionSobolevBesovTransition}.
  Then for any $s > s^\ast$ there exist $c, \eps_0 > 0$
  such that for arbitrary $R \in \N$ and $(E_R, D_R) \in \endec{L^2(\Omega)}$ it holds that
  \[
   \mathrm{Pr}
    \big(
      \| f - D_R(E_R(f)) \|_{L^2(\Omega)} \leq \eps
    \big)
    \leq 2^{R - c \cdot \eps^{-1/s}}
    \qquad \forall \, \eps \in (0, \eps_0).
  \]
\end{thm}

\begin{proof}
  This follows from Theorems~\ref{thm:mainbesovresult}, \ref{thm:SobolevPhaseTransition},
  and \ref{thm:DichotomyBanach}.
\end{proof}

Theorem~\ref{thm:IntroductionQuantitativeTheorem} is interesting due to its nonasymptotic nature.
Indeed, given a fixed budget of $R$ bits and a desired accuracy $\eps$, it provides a partial answer
to the question:
\begin{quote}
  How likely is one to succeed in describing a random $f \in \signalClass$
  to within accuracy $\eps$ using $R$ bits?
\end{quote}
Figure~\ref{fig:phaseplot} provides an illustration of the phase transition behaviour
in dependence of $\eps$ and $R$; it graphically shows that the transition is quite sharp.

\begin{figure}[h]
\begin{center}
\includegraphics[width=.74\textwidth]{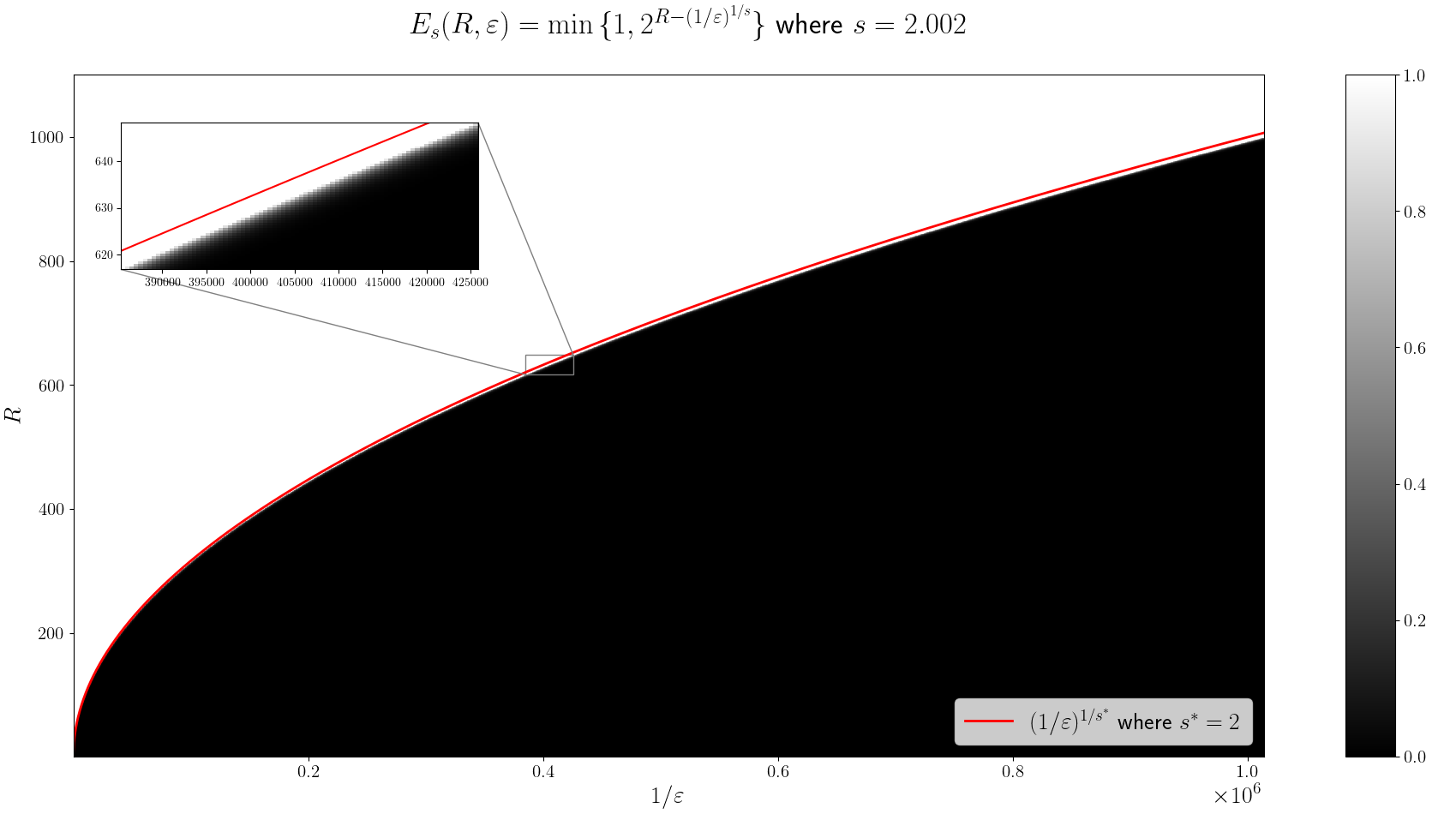}
\includegraphics[width=.74\textwidth]{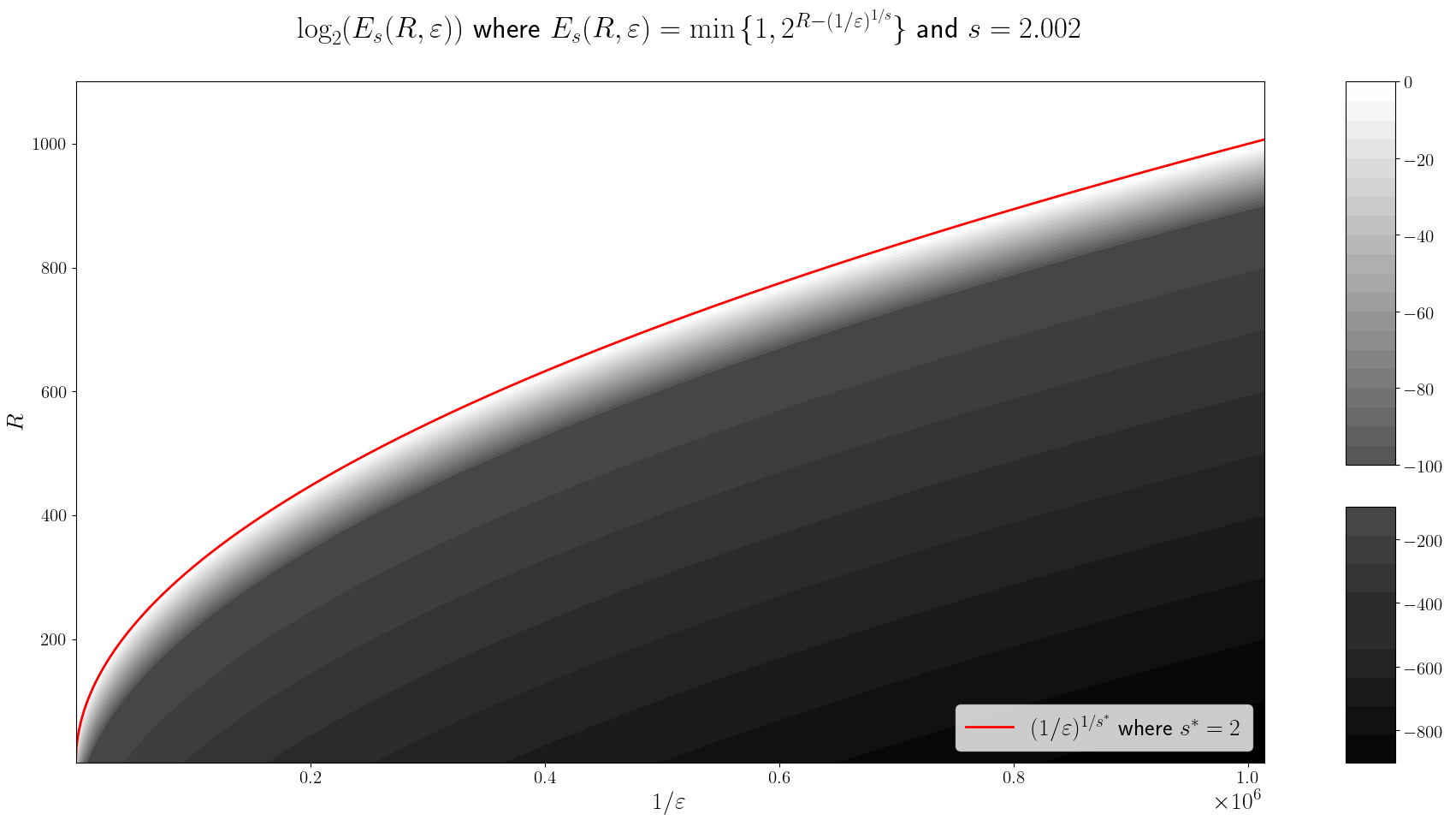}
\caption{\footnotesize For $\signalClass$ a Sobolev or Besov ball, Theorem~\ref{thm:IntroductionQuantitativeTheorem}
provides bounds on the probability of being able to describe a random function $f\in \signalClass$
to within accuracy $\eps$ using $R$ bits.
This probability is, for every $s > s^\ast$ and $\eps \in (0,\eps_0)$
($s^\ast$ denoting the optimal compression rate of $\signalClass$),
upper bounded by $E_s(R,\eps) := \min\bigl\{1,2^{R-c\cdot \eps^{-1/s}}\bigr\}$.
In this figure we show two plots of the function $E_s$ over the $(R,1/\eps)$-plane.
Both grayscale plots show $E_s$ for $s = 2.002 > s^\ast = 2$ and $c = 1$,
while the red curve indicates the critical region where $R = (1/\eps)^{1/s}$.
We see that a sharp phase transition occurs in the sense that above and slightly below
the critical curve $R = \eps^{-1/2}$ (white area) the upper bound $E_s$
does not rule out the possibility that it is always possible to describe $f\in \signalClass$
to within accuracy $\eps$ using $R$ bits;
but even slightly below the critical curve (dark area) the bound $E_s$ shows that
such a compression is almost impossible.
The sharpness of the phase transition is more clearly shown in the zoomed part of the figure.
The bottom plot further illustrates the quantitative behaviour by using a logarithmic colormap.
Note that in the bottom plot two different colormaps are used for the range $[-100,0]$
and the remaining range $[-1000, -100)$.}
\label{fig:phaseplot}
\end{center}
\end{figure}

\subsubsection{Lower Bounds for Neural Network Approximation}
\label{sub:NNApproximationLowerBounds}

As an application we draw a connection between the previously described results and
function approximation using neural networks.
We will use the following mathematical formalization of
(fully connected, feed forward) neural networks \cite{PetersenVoigtlaenderOptimalReLUApproximation}.

\begin{defn}\label{def:NNDefinition}
  Let $d,L \in \N$ and $\mathbf{N} = (N_0,\dots,N_L) \subset \N$ with $N_0 = d$.
  A \emph{neural network} (NN) with architecture $\mathbf{N}$ is a tuple
  $\Phi = \big( (A_1, b_1),\dots, (A_L,b_L) \big)$ of matrices $A_\ell \in \R^{N_\ell \times N_{\ell - 1}}$
  and bias vectors $b_\ell \in \R^{N_\ell}$.
  Given a function $\varrho : \R \to \R$, called the \emph{activation function},
  the mapping computed by the network $\Phi$ is defined as
  \[
    R_\varrho \Phi : \R^d \to \R^{N_L}, x \mapsto x^{(L)},
  \]
  where $x^{(L)}$ results from setting $x^{(0)} := x$ and furthermore
  \[
    x^{(\ell + 1)} := \varrho \big( A_{\ell+1} \, x^{(\ell)} + b_{\ell+1} \big)
    \,\,\text{ for }\,\,
    0 \leq \ell \leq L - 2,
    \quad \text{and} \quad
    x^{(L)} := A_L \, x^{(L-1)} + b_L.
  \]
  Here, $\varrho$ acts componentwise on vectors, meaning
  $\varrho ( (x_1,\dots,x_m) ) = (\varrho(x_1),\dots,\varrho(x_m))$.

  The complexity of the network $\Phi$ is described by the number $L(\Phi) := L$ of layers,
  the number $N(\Phi) := \sum_{\ell=0}^L N_\ell$ of neurons and the number
  $W(\Phi) := \sum_{\ell=1}^L \big( \| A_\ell \|_{\ell^0} + \| b_\ell \|_{\ell^0} \big)$
  of weights (or connections) of $\Phi$.
  Here, for a matrix or vector $A$, we denote by $\| A \|_{\ell^0}$ the number of nonzero
  entries of $A$.
  Furthermore, we set $d_{\mathrm{in}}(\Phi) := N_0$ and $d_{\mathrm{out}}(\Phi) := N_L$.

  We will also be interested in the complexity of the \emph{individual} weights and biases
  of the network.
  Precisely, for $\sigma , W \in \N$, we say that $\Phi$ is \emph{$(\sigma,W)$-quantized}
  if all entries of the matrices $A_\ell$ and the vectors $b_\ell$ belong to
  \({
    \bigl[ - W^{\sigma \lceil \log_2 W \rceil}, W^{\sigma \lceil \log_2 W \rceil} \bigr]
    \cap 2^{-\sigma \lceil \log_2 W \rceil^2} \Z
    \subset \R
  }\).
\end{defn}

Note that in applications one necessarily deals with quantized NNs due to the necessity
to store and process the weights on a digital computer.
Regarding function approximation by such quantized neural networks, we have the following result:

\begin{thm}\label{thm:IntroductionNNResult}
  Let $\varrho : \R \to \R$ be measurable with $\varrho(0) = 0$ and let $d, \sigma \in \N$.
  For $W \in \N$, define
  \[
    \mathcal{NN}_{d,W}^{\sigma,\varrho}
    := \big\{
         R_\varrho \Phi
         \colon
         \Phi \text{ is a } (\sigma,W)\text{-quantized NN and }
         W(\Phi) \leq W,
         d_{\mathrm{in}}(\Phi) \!=\! d,
         d_{\mathrm{out}}(\Phi) \!=\! 1
       \big\}.
  \]
  Let $\signalClass$, $s^\ast$, and $\P$ as in Theorem~\ref{thm:IntroductionSobolevBesovTransition}.
  Then the following hold:
  \setlength{\leftmargini}{0.8cm}
  \begin{enumerate}
    \item There is $C = C(d,\sigma) \in \N$ such that for each $s > s^\ast$ there are $c, \eps_0 > 0$
          satisfying
          \[
            \mathrm{Pr}
            \Bigl(\,
              \min_{g \in \mathcal{NN}_{d,W}^{\sigma,\varrho}}
                \| f - g \|_{L^2(\Omega)}
            \leq \eps
            \Bigr)
            \leq 2^{C \cdot W \, \lceil \log_2 (1+W) \rceil^2 - c \cdot \eps^{-1/s}}
            \quad \forall \, \eps \in (0,\eps_0)
            .
            \vspace*{-0.2cm}
          \]

    \item If we define\vspace*{-0.1cm}
          \[
            W_\eps^{\sigma,\varrho} (f)
            := \inf \Bigl\{
                      W \in \N
                      \,\,\, \colon \,\,\,
                      \exists \, g \in \mathcal{NN}_{d,W}^{\sigma,\varrho}
                      \text{ such that }
                      \| f - g \|_{L^2(\Omega)} \leq \eps
                    \Bigr\} \in \N \cup \{ \infty \}
          \]
          and\vspace*{-0.1cm}
          \[
            \mathcal{A}_{\mathcal{NN},\varrho}^\ast
            := \Big\{
                 f \in \signalClass
                 \colon
                 \exists \, \tau \in (0, \tfrac{1}{s^\ast}), \sigma \in \N, C > 0 \quad
                   \forall \, \eps \in (0, 1):
                     W_\eps^{\sigma,\varrho}(f) \leq C \cdot \eps^{-\tau}
               \Big\} ,
          \]
          then $\P^\ast ( \mathcal{A}_{\mathcal{NN},\varrho}^\ast ) = 0$.
  \end{enumerate}
\end{thm}

\begin{proof}
  The proof of this theorem is deferred to Appendix~\ref{sec:NNLowerBoundProofs}.
\end{proof}

Theorem~\ref{thm:IntroductionNNResult} can be interpreted as follows:
Suppose we would like to approximate a function $f \in \signalClass$  to within accuracy $\eps$
using (quantized) neural networks of size $\leq W$.
Theorem~\ref{thm:IntroductionNNResult} provides an upper bound on the probability of success.
In particular it shows that the network size has to scale at least of order $\eps^{-1/s^\ast}$
to succeed with high probability if $\signalClass$ is a Sobolev- or Besov ball;
see Figure~\ref{fig:NNfig}.

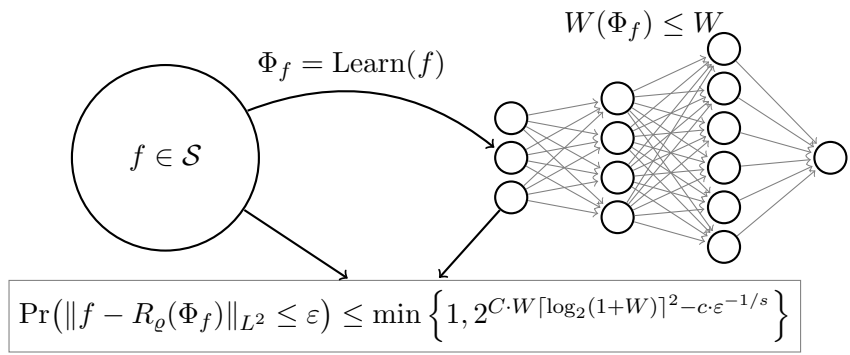
\begin{figure}
\begin{center}
\begin{tikzpicture}[shorten >=1pt,->,draw=black!50, node distance=5cm, scale=0.35]
\tikzstyle{every pin edge}=[<-,shorten <=1pt]
\tikzstyle{neuron}=[circle,draw=black,thick,minimum size=12pt,inner sep=1pt]
\tikzstyle{bigcirc}=[circle,draw=black,thick,minimum size=70pt,inner sep=1pt]
\tikzstyle{box}=[rectangle,rounded corners,fill=black!5,draw=black!50,thick,minimum size=25pt,inner sep=1pt]
\tikzstyle{input neuron}=[neuron];
\tikzstyle{output neuron}=[neuron];
\tikzstyle{hidden neuron}=[neuron];
\tikzstyle{signal class}=[bigcirc];

{
  \foreach \name / \y in {1,2,3}
  \node[input neuron] (I-\name) at (9+4,-1.5*\y+6) {};

  \foreach \name / \y in {1,...,4}
  \path
  node[hidden neuron] (H-\name) at (13+4,-1.5*\y+6.75) {};

  \foreach \source in {1,2,3}
  \foreach \dest in {1,2,3,4}
  \path (I-\source) edge (H-\dest);

  \foreach \name / \y in {1,...,6}
  \path
  node[hidden neuron] (L-\name) at (17+4,-1.5*\y+8.625) {};


  \foreach \source in {1,...,4}
  \foreach \dest in {1,...,6}
  \path (H-\source) edge (L-\dest);

  \node[output neuron] (out) at (21+4,3){};

  \foreach \source in {1,...,6}
  \path (L-\source) edge (out);


}

{\node[signal class] (sigs) at (0,3) {$f\in \mathcal{S}$};}

{
  \draw[->] (sigs)  edge[bend left, thick, black] (I-2);

  \node[] at (7,6.5) {$\Phi_f=\mbox{Learn}(f)$};
}

{\node[] at (18,8) {$W(\Phi_f)\le W$};}

{
  \node[draw] (err) at (9,-3) {\(
                                 \mathrm{Pr}\big(\|f-R_{\varrho}( \Phi_f)\|_{L^2} \leq \varepsilon\big)
                                 \leq \min\left\{
                                             1,
                                             2^{C\cdot W\lceil\log_2 (1+W) \rceil^2  - c\cdot \varepsilon^{-1/s}}
                                          \right\}
                               \)};

  \draw[->] (sigs) edge[thick,black] (err);
  \draw[->] (I-3) edge[thick,black] (err);
}

\end{tikzpicture}
\caption{\footnotesize Suppose we want to approximately represent a signal $f$ to within accuracy $\varepsilon$
by a (quantized) neural network $\Phi_f$ constrained to be of size $W(\Phi_f)\le W$
(for example due to limited memory).
Such a network shall be produced by any numerical ``learning'' procedure $\Phi_f=\mbox{Learn}(f)$.
Suppose further that the only available prior information is that $f \in \signalClass$,
where $\signalClass$ has optimal compression rate $s^\ast$ as in Theorem~\ref{thm:IntroductionSobolevBesovTransition}
(such prior information is, for instance, available if $f$ is the solution of a linear elliptic PDE
with known right hand side).
Then, no matter how we choose the ``learning'' algorithm $\mbox{Learn}(f)$,
Theorem~\ref{thm:IntroductionNNResult} states that for any $s>s^\ast$ there are constants $c,C$
such that  the probability of success is bounded from above by
$\min \big\{1, 2^{C \cdot W \lceil\log_2 (1+W) \rceil^2  - c\cdot \varepsilon^{-1/s}} \big\}$.}
\label{fig:NNfig}
\end{center}
\end{figure}

\begin{rem}[Sharpness of Theorem~\ref{thm:IntroductionNNResult}]\label{rem:NNResultSharpness}
  For the ReLU activation function given by ${\varrho (x) = \max\{ 0, x \}}$
  and $\Omega = [0,1]^d$, Theorem~\ref{thm:IntroductionNNResult} is sharp; in other words,
  there exist $C = C(\signalClass) > 0$ and $\sigma = \sigma(\signalClass) \in \N$ such that
  \[
    \forall \, f \in \signalClass \quad
      \forall \, \eps \in (0, \tfrac{1}{2}): \qquad
        W_\eps^{\sigma,\varrho} (f)
        \leq C \cdot \eps^{-1/s^\ast} \cdot \log_2(1/\eps)
        \lesssim \eps^{-\tau} ,
  \]
  where $\tau \in (0, \frac{1}{s^\ast})$ is arbitrary.
  This follows from results in \cite{SuzukiBesovNNApproximation,GrohsNNApproximationTheory}.
  Since the details are mainly technical, the proof is deferred to Appendix~\ref{sec:NNLowerBoundProofs}.
  We remark that by similar arguments as in \cite{SuzukiBesovNNApproximation,GrohsNNApproximationTheory},
  one can also prove the sharpness for other activation
  functions than the ReLU and other domains than $[0,1]^d$.
\end{rem}

\subsection{Related literature}%
\label{sub:Literature}

Many (optimality) results in approximation theory are formulated in a \emph{minimax sense},
meaning that one precisely characterizes the asymptotic decay of
\[
  d_\banach (\signalClass, M_n)
  = \sup_{f \in \signalClass} \,\,
      \inf_{g \in M_n} \,\,
        \| f - g \|_{\banach} ,
\]
where $\signalClass \subset \banach$ is the signal class to be approximated, and
$M_n \subset \banach$ contains all functions ``of complexity $n$'',
for example polynomials of degree $n$ or shallow neural networks with $n$ neurons, etc.
As recent examples of such results related to neural networks, we mention
\cite{boelcskeiNeural,YarotskyPhaseDiagram,PetersenVoigtlaenderOptimalReLUApproximation}.

A minimax lower bound of the form $d_{\banach} (\signalClass, M_n) \gtrsim n^{-s^\ast}$,
however, only makes a claim about the possible \emph{worst case} of approximating elements $f \in \signalClass$.
In other words, such an estimate in general only guarantees that there is \emph{at least one}
``hard to approximate'' function $f^\ast \in \signalClass$ that satisfies
$\inf_{g \in M_n} \| f^\ast - g \|_{\banach} \gtrsim n^{-s}$ for each $s > s^\ast$,
but \emph{nothing is known about how ``massive'' this set of ``hard to approximate'' functions is},
or about the ``average case''.

The first paper to address this question---and one of the main sources of inspiration for the
present paper---is \cite{MMR99}.
In that paper, Maiorov, Meir, and Ratsaby consider essentially the
``$L^2$-Besov-space type'' signal class $\signalClass = \signalClass_r$
of functions $f \in L^2(\B_d)$ (with ${\B_d = \{ x \in \R^d \colon \| x \|_2 \leq 1 \}}$)
that satisfy
\[
  d_{L^2}(f, \mathscr{P}_{2^N}) \leq 2^{-r N},
\]
where
\(
  \mathscr{P}_K
  = \mathrm{span} \bigl
  \{
    x^\alpha
    \colon
    \alpha \in \N_0^d \text{ with } |\alpha| \leq K \bigr
  \}
\)
denotes the space of $d$-variate polynomials of degree at most $K$.
On this signal class, they construct a probability measure $\P$ such that given the subset of functions
\[
  M_n
  =
  \Bigl\{
    \sum_{i =1}^n
      g_i (\langle a_i, x \rangle)
    \colon
    a_i \in \mathbb{S}^{d-1} \text{ and } g_i \in L^2([-1,1])
  \Bigr\} ,
\]
one obtains the minimax asymptotic
\(
  d_{L^2}(\signalClass_r, M_n)
  \asymp n^{-r/(d-1)},
\)
but furthermore there is $c > 0$ such that
\[
  \P \Big(
       \Big\{
         f \in \signalClass_r
         \colon
         d_{L^2}(f, M_n) \geq c \cdot n^{-r/(d-1)}
       \Big\}
     \Big)
  \geq 1 - e^{- c \cdot n^{d/(d-1)}} .
\]
In other words, the measure of the set of functions for which the minimax asymptotic is sharp
tends to $1$ for $n \to \infty$.
In this context, we would also like to mention the recent article \cite{LinLimitationsOfShallowNets},
in which the results of \cite{MMR99} are extended to cover more general signal classes
and approximation in stronger norms than the $L^2$ norm.

While we draw heavily on the ideas from \cite{MMR99} for the construction of the measure $\P$
in Theorem~\ref{thm:IntroductionSobolevBesovTransition}, it should be noted that we are interested
in phase transitions for general encoding/decoding schemes, while \cite{MMR99,LinLimitationsOfShallowNets}
exclusively focus on approximation using the ridge function classes $M_n$.

Finally, we would like to point out that our lower bounds for neural network approximation
consider networks with \emph{quantized weights},
as in \cite{boelcskeiNeural,PetersenVoigtlaenderOptimalReLUApproximation}.
The main reason is that without such an assumption, even two-layer networks
\emph{with a fixed number of neurons} can approximate any function arbitrarily well
if the activation function is chosen suitably; see \cite[Theorem~4]{MaiorovPinkusLowerBoundsMLP}.
Moreover, even if one considers the popular ReLU activation function, it was recently observed
that the optimal approximation rates for networks with quantized weights can in fact be
\emph{doubled} using arbitrarily deep ReLU networks with highly complex weights
\cite{YarotskyPhaseDiagram}.

\subsection{Outline}
\label{sub:Outline}

In Section~\ref{sec:GrowthRateOfMeasures}, we introduce and study a class of probability measures
with a certain growth behaviour.
More precisely, we say that $\P$ is of \emph{logarithmic growth order} $s_0$
on $\signalClass \subset \banach$ if for each $s > s_0$, we have
\[
  \P \bigl(\ball(\xvec,\eps; \banach)\bigr)
  \leq 2^{-c \cdot \eps^{-1/s}}
  \qquad \forall \, \xvec \in \banach \text{ and } \eps \in (0, \eps_0),
\]
for suitable $c,\eps_0 > 0$ depending on $s_0$.
Here, as in the rest of the paper, $\ball(\xvec,\eps;\banach)$ is the open ball
around $\xvec$ of radius $\eps$ with respect to $\| \cdot \|_{\banach}$.
A measure has \emph{critical growth} if its logarithmic growth order equals the optimal
compression rate $\optRateSmall{\banach}$.
We show in particular that every critical measure exhibits a compressibility phase transition
as in Definition~\ref{defn:IntroPhaseTransitionMeasure}, and we show how critical measures can
be transported from one set to another.

In Section~\ref{sec:MainProof}, we study certain sequence spaces $\mixSpace{q}$;
these are essentially the coefficient spaces associated to Besov spaces.
By modifying the construction given in \cite{MMR99}, we construct probability measures of
critical growth on the unit balls $\sequenceSpaceSignalClass$ of the spaces $\mixSpace{q}$,
for the range of parameters for which the embedding $\mixSpace{q} \hookrightarrow \ell^2$ is compact.

The construction of critical measures on the unit balls of Besov and Sobolev spaces is then
accomplished in Section~\ref{sec:Examples}, essentially by using wavelet systems to transfer
the critical measure from the sequence spaces to the function spaces.
This makes heavy use of the transfer results established in Section~\ref{sec:GrowthRateOfMeasures}.

A host of more technical proofs are deferred to the appendices.

\subsection{Notation}
\label{sub:Notation}

We write $\N := \{1,2,3,\dots\}$ for the set of natural numbers, and $\N_0 := \{0\} \cup \N$
for the natural numbers including zero.
For $n \in \N_0$, we define $\FirstN{n} := \{ k \in \N \colon k \leq n \}$;
in particular, $\FirstN{0} = \emptyset$.

For $x \in \R$, we write $x_+ := \max \{ 0, x \}$ and $x_{-} := (-x)_+ = \max \{ 0, -x \}$.

We assume all vector spaces to be over $\R$, unless explicitly stated otherwise.

For a given (quasi)-normed vector space $(\banach, \|\cdot\|)$,
we denote the \emph{closed ball} of radius $r \geq 0$ around $\xvec \in \banach$ by
$\ball (\xvec, r; \banach) := \{ \yvec \in \banach \colon \| \yvec - \xvec \| \leq r \}$.
If we want to emphasize the quasi-norm (for example, if multiple quasi-norms are considered on
the same space $\banach$), we write $\ball ( \xvec, r; \| \cdot \| )$ instead.

For an index set $\mathcal{I}$ and an integrability exponent $p \in (0,\infty]$,
 the sequence space $\ell^p (\mathcal{I}) \subset \R^{\mathcal{I}}$ is
\[
  \ell^p (\mathcal{I})
  = \big\{
      \xvec = (x_i)_{i \in \mathcal{I}} \in \R^{\mathcal{I}}
      \quad \colon \quad
      \| \xvec \|_{\ell^p} < \infty
    \big\} ,
\]
where $\| \xvec \|_{\ell^p} := \bigl(\sum_{i \in \mathcal{I}} |x_i|^p\bigr)^{1/p}$ if $p < \infty$,
while $\| \xvec \|_{\ell^\infty} := \sup_{i \in \mathcal{I}} |x_i|$.

For a measure $\mu$ on a measurable space $(\signalClass, \mathscr{A})$ the
\emph{outer measure} $\mu^\ast : \powerset{\signalClass} \to [0,\infty]$ induced by $\mu$ is given by
\begin{equation}
  \mu^\ast (M)
  := \inf \Big\{
            \sum_{n=1}^\infty
              \mu(M_n)
            \, \colon
            (M_n)_{n \in \N} \subset \mathscr{A} \text{ with } M \subset \bigcup_{n=1}^\infty M_n
          \Big\} \, .
  \label{eq:OuterMeasureDefinition}
\end{equation}
It is well-known (see \cite[Proposition~1.10]{FollandRA}) that $\mu^\ast$ is $\sigma$-subadditive,
meaning that ${\mu^\ast (\bigcup_{n=1}^\infty M_n) \leq \sum_{n=1}^\infty \mu^\ast (M_n)}$
for arbitrary $M_n \subset \signalClass$.
We will be interested in \emph{$\mu^\ast$-null-sets};
that is, subsets $N \subset \signalClass$ satisfying $\mu^\ast (N) = 0$.
This holds if and only if there is $N' \in \mathscr{A}$ satisfying
$N \subset N'$ and $\mu(N') = 0$.
Furthermore, directly from the $\sigma$-subadditivity of $\mu^\ast$, it follows that
a countable union of $\mu^\ast$-null-sets is again a $\mu^\ast$-null-set.

\textbf{A comment on measurability:}
Given a (not necessarily measurable) subset $M \subset \banachOne$ of a Banach space $\banachOne$,
we will always equip $M$ with the trace $\sigma$-algebra
\[
  M \Cap \mathcal{B}_{\banachOne}
  = \{ M \cap B \colon B \in \mathcal{B}_{\banachOne} \}
\]
of the Borel $\sigma$-algebra $\mathcal{B}_{\banachOne}$.
A \emph{Borel measure} on $M$ is then a measure defined on $M \Cap \mathcal{B}_{\banachOne}$.

Note that if $(\Omega, \mathscr{A})$ is an arbitrary measurable space,
then ${\Phi : \Omega \to M}$ is measurable if and only if it is measurable
when considered as a map ${\Phi : \Omega \to (\banachOne, \mathcal{B}_{\banachOne})}$.

\section{General results on phase transitions in Banach spaces}
\label{sec:GrowthRateOfMeasures}

In this section we establish an abstract version of the phase transition considered in
\eqref{eq:PhaseTransition} for signal classes in general Banach spaces and a class of measures that
satisfy a uniform growth property that we term ``critical'' (see Definition~\ref{def:GrowthRate}).
We will show in Section~\ref{sec:Crit-Fine} that such critical measures automatically
induce a phase transition behavior.
We furthermore show in Section~\ref{sec:TransferResults} that criticality is preserved
under pushforward by ``nice'' mappings.
The \emph{existence} of critical measures is by no means trivial;
quite the opposite, their construction for a class of sequence spaces in Section~\ref{sec:PhaseTrans_ss}---%
and for Besov and Sobolev spaces on domains in Section~\ref{sec:Examples}---%
constitutes an essential part of the present article.

\subsection{Measures of logarithmic growth}
\label{sec:Crit-Fine}

\begin{defn}\label{def:GrowthRate}
 Let $\signalClass$ be a subset of a Banach space $\banach$, and let $s_0 \in [0,\infty)$.

  A Borel probability measure $\P$ on $\signalClass$ has
  \emph{(logarithmic) growth order $s_0$ (with respect to $\banach$)}
  if for every $s > s_0$, there are constants $\eps_0, \gc > 0$
  (depending on $s,s_0,\P,\signalClass,\banach$) such that
  \begin{equation}
    \P \big(
         \signalClass \cap \ball (\xvec, \eps; \banach)
       \big)
    \leq 2^{-\gc \cdot \eps^{-1/s}}
    \qquad \forall \, \xvec \in \banach \text{ and } \eps \in (0,\eps_0) .
    \label{eq:CriticalMeasureDefinition}
  \end{equation}

  We say that $\P$ is \emph{critical for $\signalClass$ (with respect to $\banach$)}
  if $\P$ has logarithmic growth order $\optRateSmall{\banach}$,
  with the optimal compression rate $\optRateSmall{\banach}$
  as defined in Equation~\eqref{eq:OptimalCompressionRate}.
\end{defn}

\begin{rem*}
  If $\P$ has growth order $s_0$, then $\P$ also has growth order $\sigma$,
  for arbitrary $\sigma > s_0$.
\end{rem*}

The motivation for considering the growth order of a measure is that it leads to
bounds regarding the measure of elements $\xvec \in \signalClass$ that are well-approximated
by a given codec; see Equation~\eqref{eq:NewQuantitativeStatement} below.
Furthermore, as we will see in Corollary~\ref{cor:GrowthOrderAtLeastOptimalRate},
if $\P$ is a probability measure of growth order $s_0$, then necessarily $s_0 \geq \optRateSmall{\banach}$,
so \emph{critical measures have the minimal possible growth order}.

The following theorem summarizes our main structural results, showing that critical measures
always exhibit a compressibility phase transition.

\begin{thm}\label{thm:DichotomyBanach}
  Let the signal class $\signalClass$ be a subset of the Banach space $\banach$,
  let $\P$ be a Borel probability measure on $\signalClass$ that is critical for $\signalClass$
  with respect to $\banach$, and set $s^\ast := \optRateSmall{\banach}$.
  Then the following hold:
  \begin{enumerate}[label={(\roman*)}]
    \item \label{enu:DichotomyQuantitative}
          Let $s > s^\ast$ and let $\gc = \gc(s) > 0$ and $\eps_0 = \eps_0(s)$
          as in Equation~\eqref{eq:CriticalMeasureDefinition}.
          Then, for any $R \in \N$ and $(E_R, D_R) \in \endec{\banach}$, we have
          \begin{equation}
            \mathrm{Pr}\big( \| \xvec - D_R(E_R(\xvec)) \|_{\banach} \leq \eps \big)
            \leq 2^{R - c \cdot \eps^{-1/s}}
            \qquad \forall \, \eps \in (0,\eps_0) ,
            \label{eq:NewQuantitativeStatement}
          \end{equation}
          where we use the notation from Equation~\eqref{eq:SpecialProbabilityNotation}.

    \item \label{enu:DichotomySupcrit}
          For every $s > s^*$ and every codec $\code \in \codecs<\signalClass>{\banach}$, the set
          $\approxClass{s}<\signalClass>{\banach}(\code)$ is a $\P^*$-null-set:
          \begin{equation*}
            \mathrm{Pr}
            \left(
              \approxClass{s}<\signalClass>{\banach} (\code)
            \right) = 0.
          \end{equation*}

    \item \label{enu:DichotomySubcrit}
          For every $0 \leq s < {s^*}$, there is a codec
          \(
            \code = \big( (E_R,D_R) \big)_{R \in \N}
            \in \codecs<\signalClass>{\banach}
          \)
          with distortion
          \[
            \distortion<\signalClass>{\banach} (E_R, D_R) \leq C \cdot R^{-s}\,
            \qquad \forall \, R \in \N ,
          \]
          for a constant $C = C(s, \code) > 0$.
          In particular, the set of $s$-compressible signals $\approxClass{s}<\signalClass>{\banach} (\code)$
          defined in Eq.~\eqref{eq:ApproximationClass} satisfies
          $\approxClass{s}<\signalClass>{\banach} (\code) = \signalClass$
          and hence ${\P (\approxClass{s}<\signalClass>{\banach} (\code)) = 1}$.
  \end{enumerate}
\end{thm}

\begin{rem*}
  1) Note that the theorem does not make any statement about the case $s = s^\ast$.
  In this case, the behavior depends on the specific choices of $\signalClass$ and $\P$.

  \smallskip{}

  2) As noted above, the question of the existence of a critical probability measure $\P$
  is nontrivial.
\end{rem*}

The proof of Theorem~\ref{thm:DichotomyBanach} is divided into several auxiliary results.
Part~(i) is contained in the following lemma.

\begin{lem}\label{lem:CriticalMeasuresAreSmall}
  Let $\signalClass$ be a subset of a Banach space $\banach$,
  and let $\P$ be a Borel probability measure on $\signalClass$ that is
  of logarithmic growth order $s_0 \geq 0$ with respect to $\banach$.

  Let $s > s_0$ and let $\gc = \gc(s) > 0$ and $\eps_0 = \eps_0(s)$
  as in Equation~\eqref{eq:CriticalMeasureDefinition}.
  Then, for any $R \in \N$ and $(E_R, D_R) \in \endec{\banach}$, we have
  \[
    \P^\ast \big(\{\xvec\in \signalClass:\  \| \xvec - D_R(E_R(\xvec)) \|_{\banach} \leq \eps\} \big)
    \leq 2^{R - c \cdot \eps^{-1/s}}
    \qquad \forall \, \eps \in (0,\eps_0) .
  \]

  Furthermore, for any given $s > s_0$ and $K > 0$ there exists a minimal code-length
  ${R_0 = \! R_0(s,s_0,K,\P,\signalClass,\banach) \! \in \N}$
  such that for every codec
  ${\code = \big( (E_R, D_R) \big)_{R \in \N} \!\in \codecs{\banach}}$, we have
  \begin{equation}
    \label{eq:measapperr}
    \P^*\big(
           \set{
             \xvec \in \signalClass
             \colon
             \norm{\xvec - D_R(E_R(\xvec))}_\banach \leq K \cdot R^{-s}
           }
         \big)
    \leq 2^{-R}\,
    \qquad \forall \, R \geq R_0 .
  \end{equation}
\end{lem}

\begin{rem*}
  The lemma states that the measure of the subset of points $\xvec \in \signalClass$
  with approximation error ${\cal E}_R(\xvec) = \norm{\xvec - D_R(E_R(\xvec))}_\banach$ satisfying
  ${\cal E}_R(\xvec) \leq K \cdot R^{-s}$ for some $s > s_0$ decreases exponentially with $R$.
  In fact, the proof shows that the approximation error is decreasing asymptotically
  \emph{superexponentially}.
\end{rem*}

\begin{proof}
Let $s > s_0$ and let $\gc, \eps_0$ as in Equation~\eqref{eq:CriticalMeasureDefinition}.
For $R \in \N$ and $\eps \in (0, \eps_0)$, define
\(
A (R, \eps)
  := \set{
    \xvec \in \signalClass
    \colon
    \norm{\xvec - D_R (E_R(\xvec)) }_{\banach} \leq \eps
  }
  .
\)
By definition,
\[
  A(R, \eps)
  \subset \bigcup_{\yvec \in \mathrm{range}(D_R)}
          \big[
            \signalClass \cap \ball (\yvec, \eps ; \banach)
          \big]
  \,.
\]
Since $\P$ is of growth order $s_0$ and because of $|\mathrm{range}(D_R)| \leq 2^R$,
we can apply~\eqref{eq:CriticalMeasureDefinition} and the subadditivity of the outer measure
$\P^\ast$ to deduce
\[
  \P^\ast \bigl( A(R,\eps) \bigr)
     \leq \sum_{\yvec \in \mathrm{range} (D_R)}
            \P
            \big(
              \signalClass \cap \ball (\yvec, \eps; \banach)
            \big)
     \leq 2^R \cdot 2^{-\gc \, \eps^{-1/s}}.
\]
This proves the first part of the lemma.

To prove the second part, let $s > s_0$, and choose $\sigma = \frac{s + s_0}{2}$,
noting that  $\sigma \in (s_0, s)$.
Therefore, the first part of the lemma, applied with $\sigma$ instead of $s$,
yields $c, \eps_0 > 0$ such that
\(
  \P^\ast (\{ \xvec \in \signalClass \colon \| \xvec - D_R(E_R(\xvec)) \|_{\banach} \leq \eps\})
  \leq 2^{R - c \cdot \eps^{-1/\sigma}}
\)
for all $R \in \N$ and $\eps \in (0, \eps_0)$.

Note that $\eps := K \cdot R^{-s} \leq \eps_0/2 < \eps_0$ holds as soon as
$R \geq \big\lceil (2 K / \eps_0)^{1/s} \, \big\rceil =: R_1$.
Finally, since $s / \sigma > 1$ we can find a code-length $R_2 \in \N$ such that
\[
  R - \gc \, \eps^{-1/\sigma}
  = R - \gc \, K^{-1/\sigma} \cdot R^{s/\sigma} \leq -R \quad \text{for } R \geq R_2 \,.
\]
Overall, we thus see that \eqref{eq:measapperr} holds, with $R_0 = \max \{ R_1, R_2 \}$.
\end{proof}

\begin{prop}\label{lem:CriticalMeasuresAreFine}
  Let $\signalClass$ be a subset of the Banach space $\banach$.
   If $\P$ is a Borel probability measure on $\signalClass$ that is
   of growth order $s_0 \in [0,\infty)$, then, for every $s > s_0$ and every codec
   $\code = \big( (E_R, D_R) \big)_{R \in \N} \in \codecs{\banach}$, we have
   \[
     \P^*\bigl(\approxClass{s}<\signalClass>{\banach} (\code)\bigr) = 0 \,.
   \]
\end{prop}

\begin{proof}
  First, note that
  \[
    \approxClass{s}<\signalClass>{\banach} (\code)
    = \bigcup_{N \in \N}
         \big\{
           \xvec \in \signalClass
           \colon
           \forall \, R \in \N:
             \| \xvec - D_R (E_R(\xvec)) \|_{\banach} \leq N \cdot R^{-s}
         \big\}
    = \bigcup_{N \in \N}
        \bigcap_{R \in \N}
          A^{(s)}_{N,R} ,
  \]
  where
  \(
    A^{(s)}_{N,R} = \set{
                      \xvec \in \signalClass
                      \colon
                      \norm{\xvec - D_R (E_R(\xvec)) }_{\banach} \leq N \cdot R^{-s}
                    }
    .
  \)

  By $\sigma$-subadditivity of $\P^\ast$, it is thus enough to show that
  $\P^\ast (\bigcap_{R \in \N} A_{N,R}^{(s)}) = 0$ for each $n \in \N$.
  To see that this holds, note that Lemma~\ref{lem:CriticalMeasuresAreSmall} shows
  \[
    0
    \leq \P^*\bigg( \bigcap_{R \in \N} A_{N,R}^{(s)}\bigg)
    \leq \P^*\bigl(A^{(s)}_{N,R}\bigr)
    \leq 2^{-R}
    \qquad \forall \, R \geq R_0 (s,s_0,N,\P,\signalClass,\banach) .
  \]
  This easily implies $\P^*\bigl( \bigcap_{R \in \N} A_{N,R}^{(s)}\bigr) = 0$.
\end{proof}

The proof of Theorem~\ref{thm:DichotomyBanach} merely consists of combining the preceding lemmas.

\begin{proof}[Proof of Theorem~\ref{thm:DichotomyBanach}]
  \emph{\textbf{Proof of \ref{enu:DichotomyQuantitative}:}}
  This is contained in the statement of Lemma~\ref{lem:CriticalMeasuresAreSmall}.

  \medskip{}

  \emph{\textbf{Proof of \ref{enu:DichotomySupcrit}:}}
  This follows from Proposition~\ref{lem:CriticalMeasuresAreFine}.

  \medskip{}

  \emph{\textbf{Proof of \ref{enu:DichotomySubcrit}:}}
  This follows from the definition of the optimal compression rate:
  for $s< s^*$ there exists a codec $\big((E_R,D_R)\big)_{R \in \N} \in \codecs{\banach}$ such that
  \[
    R^s \cdot \norm{\xvec -D_R(E_R(\xvec))}_\banach \leq C
    \qquad \forall \, R \in \N ,
  \]
  for a constant $C > 0$ and all $\xvec \in \signalClass$.
  In particular, this implies $ \approxClass{s}<\signalClass>{\banach} (\code) = \signalClass$,
  and therefore $\P(\approxClass{s}<\signalClass>{\banach} (\code) )=1$.
\end{proof}

We close this subsection by showing that if $\P$ is a probability measure with logarithmic
growth order $s_0$, then this growth order is at least as large as the optimal compression rate
of the set on which $\P$ is defined.
This justifies the nomenclature of ``critical measures'' as introduced in Definition~\ref{def:GrowthRate}.

\begin{cor}\label{cor:GrowthOrderAtLeastOptimalRate}
  Let $\signalClass$ be a subset of $\banach$,
  and $\P$ be a Borel probability measure on $\signalClass$ of growth order $s_0$.
  Then $s_0 \geq \optRateSmall{\banach}$, with $\optRateSmall{\banach}$ as defined in
  Equation~\eqref{eq:OptimalCompressionRate}.
\end{cor}

\begin{proof}
  Suppose for a contradiction that $0 \leq s_0 < \optRateSmall{\banach}$,
  and choose $s \in (s_0, \optRateSmall{\banach})$.
  By definition of $\optRateSmall{\banach}$, there is a codec
  $\big( (E_R,D_R) \big)_{R \in \N} \in \codecs{\banach}$ such that
  ${\approxClass{s}<\signalClass>{\banach} (\code) = \signalClass}$.
  By Proposition~\ref{lem:CriticalMeasuresAreFine}, we thus obtain the contradiction
  ${1 = \P (\signalClass) = \P (\approxClass{s}<\signalClass>{\banach} (\code)) = 0}$.
\end{proof}

\subsection{Transferring critical measures}
\label{sec:TransferResults}

Our main goal in this paper is to prove a phase transition as in \eqref{eq:PhaseTransition}
for Besov- and Sobolev spaces.
To do so, we will first prove (in Section~\ref{sec:MainProof}) that such a phase-transition
occurs for a certain class of sequence spaces, and then transfer this result to the
Besov- and Sobolev spaces, essentially by discretizing these function spaces using suitable
wavelet systems.
In the present subsection, we formulate general results that allow such a transfer
from a phase transition as in \eqref{eq:PhaseTransition} from one space to another.

In general, it would be most convenient if we had access to an orthonormal wavelet basis
(or at least to a Riesz basis) of wavelets that is ``compatible'' with Besov- and Sobolev spaces.
For the setting of very general domains $\Omega \subset \R^d$ and the full range of parameters $p,q$,
however, it seems to be unknown whether such orthonormal wavelet bases exist.
Therefore, our transfer results will allow to use two distinct maps:
Essentially, one can use a frame to transfer the optimal compression rate,
and a (possibly different) Riesz sequence to transfer the critical measure.
In the abstract formulation of this section, this will be formulated using a Lipschitz continuous
surjection $\Phi$ (the synthesis operator of the frame) and an expansive injection $\Psi$
(the synthesis operator of the Riesz sequence).

The precise transference result reads as follows:

\begin{thm}\label{thm:LipschitzTransferResult}
  Let $\banachOne, \banachTwo, \banachThree$ be Banach spaces, and let
  $\signalClass_\banachOne \subset \banachOne$, $\signalClass_\banachTwo \subset \banachTwo$,
  and $\signalClass \subset \banachThree$.
  Assume that
  \begin{enumerate}
    \item \(
            \optRateSmall<\signalClass_\banachOne>{\banachOne}
            = \optRateSmall<\signalClass_{\banachTwo}>{\banachTwo}
            ;
          \)

    \item there exists a Lipschitz continuous map
          $\Phi : \signalClass_\banachOne \subset \banachOne \to \banachThree$
          satisfying $\Phi(\signalClass_{\banachOne}) \supset \signalClass$;

    \item there exists a Borel probability measure $\P$ on $\signalClass_\banachTwo$
          that is critical for $\signalClass_{\banachTwo}$ with respect to $\banachTwo$;

    \item there exists an \emph{expansive} measurable map $\Psi : \signalClass_{\banachTwo} \to \signalClass$,
          meaning that there is $\kappa > 0$ satisfying
          \[
            \| \Psi (\xvec) - \Psi(\xvec') \|_{\banachThree}
            \geq \kappa \cdot \| \xvec - \xvec' \|_{\banachTwo}
            \qquad \forall \, \xvec, \xvec' \in \signalClass_{\banachTwo}.
          \]
  \end{enumerate}

  Then $\optRateSmall{\banachThree} = \optRateSmall<\signalClass_\banachOne>{\banachOne}$,
  and the push-forward measure $\P \circ \Psi^{-1}$ is a Borel probability measure
  on $\signalClass$ that is critical for $\signalClass$ with respect to $\banachThree$.
\end{thm}

\begin{rem*}
  1) In many cases, it is natural to take $\signalClass_\banachOne = \signalClass_\banachTwo$
  and $\Phi = \Psi$.
  As we will see in Section~\ref{sec:Examples}, however, the added flexibility of the formulation
  above is necessary to transfer critical measures from the sequence spaces $\generalSpace{\theta}{q}$
  considered in Section~\ref{sec:MainProof} to Besov and Sobolev spaces.

  2) As mentioned in Section~\ref{sub:Notation}, regarding the measurability of $\Psi$,
  $\signalClass_{\banachTwo}$ is equipped with the trace $\sigma$-algebra of the Borel $\sigma$-algebra
  on $\banachTwo$, and analogously for $\signalClass$.
\end{rem*}

\begin{proof}
  The proof is given in Appendix~\ref{sec:TransferResultsProofs}.
\end{proof}

\section{Proof of the phase transition in \texorpdfstring{$\ell^2(\indexSet)$}{ℓ²(𝓘)}}
\label{sec:PhaseTrans_ss}
\label{sec:MainProof}

In this section, we provide the proof of the phase transition for a class of sequence spaces
associated to Sobolev and Besov spaces; these sequences spaces are defined in Section~\ref{sub:SequenceSpaceMainResult},
where we also formulate the main result (Theorem~\ref{thm:MainSequenceSpaceResult})
concerning the compressibility phase transition for these spaces.
Section~\ref{sec:ssembed} establishes elementary embedding results for these spaces
and provides a lower bound for their optimal compression rate;
the latter essentially follows by adapting results by Leopold \cite{LeopoldEntropyNumbersOfWeightedSequenceSpaces}
to our setting.
The construction of the critical probability measure for the sequence spaces is presented
in Section~\ref{sec:constrmeas}, while the proof of Theorem~\ref{thm:MainSequenceSpaceResult}
is given in Section~\ref{sub:QuantitativeInverseTheorem}.

\subsection{Main Result}%
\label{sub:SequenceSpaceMainResult}

\begin{defn}[$d$-regular partitions] \label{def:RegPart}
  Let $\indexSet$ be a countably infinite index set, and
  $\partition = (\indexSet_m)_{m \in \N}$ be a partition of $\indexSet$;
  that is, $\indexSet = \biguplus_{m=1}^\infty \indexSet_m$, where the union is disjoint.
  For $d \in \N$ we call $\partition$ a \emph{$d$-regular partition},
  if there are $0 < a < A < \infty$ satisfying
  \begin{equation}\label{eq:weightequival}
    a \, 2^{d m} \leq |\indexSet_m| \leq A \, 2^{d m}
    \quad \mbox{for all $m\in \bn$.}
  \end{equation}
\end{defn}

\noindent
\textbf{Convention:} We will always assume that $\indexSet$, $\partition$ and $d$ have this meaning.

Associated with a $d$-regular partition we now define the following family of weighted sequence spaces.
\begin{defn}[Sequence Spaces]\label{def:SequenceSpaces}
  Let $p,q \in (0,\infty]$ and $\alpha, \theta \in \R$.
  For ${\xvec = (x_i)_{i \in \indexSet} \in \R^{\indexSet}}$, we define
  \begin{equation}
    \xvec_m := \xvec|_{\indexSet_m} = (x_i)_{i \in \indexSet_m}
    \qquad\!\! \text{and} \!\!\qquad
    \| \xvec \|_{\generalSpace{\theta}{q}}
    := \Big\|
         \Big(
           2^{\alpha m} \cdot m^\theta \cdot \big\| \xvec_m \big\|_{\ell^p (\indexSet_m)}
         \Big)_{m \in \N}
       \Big\|_{\ell^q(\N)} .
    \label{eq:mixnorm}
  \end{equation}

  The \emph{mixed-norm sequence space $\generalSpace{\theta}{q}$}
  is
  \[
    \generalSpace{\theta}{q}
    := \left\{
         \xvec \in \R^{\indexSet}
         \quad \colon \quad
         \| \xvec \|_{\generalSpace{\theta}{q}} < \infty
       \right\} .
  \]
  For brevity, we also define $\mixSpace{q} := \generalSpace{0}{q}$ and
  \[
    \generalSignalClass := \ball \big( 0,1;\generalSpace{\theta}{q} \big) ,
    \qquad \text{as well as} \qquad
    \sequenceSpaceSignalClass := \generalSignalClass{0}.
  \]
\end{defn}

In the remainder of this section, we will prove the existence of a critical measure
on each of the sets $\sequenceSpaceSignalClass$,
provided that $\alpha > d \cdot (\frac{1}{2} - \frac{1}{p})_+$.
In the proof, the (otherwise not really important) spaces $\generalSpace{\theta}{q}$ will play
an essential role.
Our main result is thus the following theorem, the proof of which is given in
Section~\ref{sub:QuantitativeInverseTheorem} below.

\begin{thm}\label{thm:MainSequenceSpaceResult}
  Let $p,q \in (0,\infty]$ and $\alpha \in \R$,
  and assume that $\alpha > d \cdot \big( \frac{1}{2} - \frac{1}{p} \big)_+$.

  Then $\sequenceSpaceSignalClass \subset \ell^2(\indexSet)$ is compact and hence Borel measurable,
  its optimal compression rate is given by
  ${\optRate<\sequenceSpaceSignalClass>{\ell^2(\indexSet)} = \frac{\alpha}{d} - (\frac{1}{2} - \frac{1}{p})}$,
  and there exists a Borel probability measure $\P_{\partition,\alpha}^{p,q}$ on $\sequenceSpaceSignalClass$
  that is critical for $\sequenceSpaceSignalClass$ with respect to $\ell^2(\indexSet)$.
  In particular, the phase transition described in Theorem~\ref{thm:DichotomyBanach} holds.
\end{thm}

\subsection{Embedding results and a lower bound for the compression rate}
\label{sec:ssembed}

Having introduced the signal classes $\sequenceSpaceSignalClass$, we now collect two technical
ingredients needed to construct the measures $\sequenceProductMeasure$ on these sets:
A lower bound for the optimal compression rate of $\sequenceSpaceSignalClass$
(Proposition~\ref{prop:ApproximationRateLowerBound}) and certain elementary embeddings
between the spaces $\generalSpace{\theta}{q}$ for different choices of the parameters
(Lemma~\ref{lem:ElementaryEmbeddings}).

\begin{lem}\label{lem:ElementaryEmbeddings}
  Let $p,q,r \in (0,\infty]$ and $\alpha,\beta,\theta,\vartheta \in \R$.
  If $q > r$ and $\vartheta > \frac{1}{r} - \frac{1}{q}$,
  then $\generalSpace{\theta+\vartheta}{q} \hookrightarrow \generalSpace{\theta}{r}$.
  More precisely, there is a constant $\kappa = \kappa(r,q,\vartheta) > 0$ such that
  \(
    \| \xvec \|_{\generalSpace{\theta}{r}}
    \leq \kappa \cdot \| \xvec \|_{\generalSpace{\theta + \vartheta}{q}}
  \)
  for all $\xvec \in \R^{\indexSet}$.
\end{lem}

\begin{proof}
  The claim follows by an elementary application of Hölder's inequality;
  the details can be found in Appendix~\ref{sec:TechnicalSequenceSpaceProofs}.
\end{proof}

We continue by lower bounding the optimal compression rate of the classes $\sequenceSpaceSignalClass$.
As we will see in Theorem~\ref{thm:MainSequenceSpaceResult}, we actually have an equality.

\begin{prop}\label{prop:ApproximationRateLowerBound}
  Let $p,q \in (0,\infty]$ and $\alpha \in (0,\infty)$, and assume that
  ${\alpha > d \cdot (\tfrac{1}{2} - \tfrac{1}{p})_+}$.
  Then $\mixSpace{q} \hookrightarrow \ell^2(\indexSet)$ and
  $\sequenceSpaceSignalClass \subset \ell^2 (\indexSet)$ is compact with
  \({
    \optRate<\sequenceSpaceSignalClass>{\ell^2(\indexSet)}
    \geq \tfrac{\alpha}{d} - (\tfrac{1}{2} - \tfrac{1}{p})
  }\).
  Furthermore, there exists a codec
  ${\code = \big( (E_R, D_R) \big)_{R \in \N} \in \codecs<\sequenceSpaceSignalClass>{\ell^2(\indexSet)}}$
  satisfying
  \[
    \distortion<\sequenceSpaceSignalClass>{\ell^2(\indexSet)} (E_R, D_R)
    \lesssim R^{- \bigs( \frac{\alpha}{d} - (\frac{1}{2} - \frac{1}{p}) \bigs)}
    \qquad \forall \, R \in \N.
  \]
\end{prop}

\begin{proof}
In essence, this an entropy estimate for sequence spaces; see \cite{EdmundsTriebelFunctionSpacesEntropyNumbers}.
Since the precise proof is mainly technical, it is deferred to Appendix~\ref{sec:SequenceSpaceCompressionRateLowerBound}.
\end{proof}

\subsection{Construction of the measure}
\label{sec:constrmeas}

We now come to the technical heart of this section---the construction of the measures $\sequenceProductMeasure$.
We will provide different constructions for $q = \infty$ and for $q < \infty$:
Since for $q = \infty$ the class $\generalSignalClass{\theta}{\infty}$ has a natural
product structure (Lemma~\ref{lem:ballprod}), we define the
measure as a product measure (Definition~\ref{def:Measureinfty}).
We then use the embedding result of Lemma~\ref{lem:ElementaryEmbeddings} to transfer the measure
on $\generalSignalClass{\theta}{\infty}$ to
the general signal classes $\sequenceSpaceSignalClass$; see Definition~\ref{def:Measurenotinfty}.

We start with the elementary observation that the balls $\generalSignalClass{\theta}{\infty}$
can be written as infinite products of finite dimensional balls.

\begin{lem}\label{lem:ballprod}
  The balls of the mixed-norm sequence spaces satisfy (up to canonical identifications)
  the factorization
  \[
    \generalSignalClass{\theta}{\infty}
    = \ball \big( 0, 1 ; \generalSpace \big)
    = \prod_{m \in \N}
        \ball \big( 0, 2^{-\alpha m} \; m^{-\theta} ; \ell^p(\indexSet_m) \big).
  \]
\end{lem}

\begin{proof}
  We identify $\xvec \in \R^\indexSet$ with
  $(\xvec_m)_{m \in \N} \! \in \! \prod_{m \in \N} \R^{\indexSet_m}$,
  as defined in Equation~\eqref{eq:mixnorm}.
  Set $w_m := m^\theta \cdot 2^{\alpha m}$ for $m \in \N$.
  The statement of the lemma then follows by recalling that
  \[
    \| \xvec \|_{\generalSpace}
    = \sup_{m \in \N}
        \Big( w_m \cdot \|\xvec_m\|_{\ell^p(\indexSet_m)} \Big).
    \qedhere
  \]
\end{proof}

With Lemma~\ref{lem:ballprod} in hand we can readily define $\generalProductMeasure$
as a product measure.

\begin{defn}[Measures for $q=\infty$]\label{def:Measureinfty}
  Let $\partition = (\indexSet_m)_{m \in \N}$ be a $d$-regular partition of $\indexSet$.
  Let $\calB_{m}$ be the Borel $\sigma$-algebra on $\R^{\indexSet_m}$,
  and denote the Lebesgue measure on $(\R^{\indexSet_m}, \calB_{m})$ by $\lambda_{m}$.

  For $p \in (0,\infty]$ and $w_m > 0$ define the probability measure
  $\prob_{m}^{p,w_m}$ on $(\R^{\indexSet_m}, \calB_{m})$ by
  \begin{equation}
    \mathbb{P}_{m}^{p,w_m} :
        \mathcal{B}_{m} \to [0,1], \quad
        A               \mapsto \frac{\lambda_{m} \bigl( \strut
                                                    \ball \big( 0, w_m^{-1}; \ell^p (\indexSet_m) \big)
                                                    \cap A
                                                  \bigr)}
                                     {\lambda_{m} \bigl( \strut
                                                    \ball \big( 0, w_m^{-1}; \ell^p (\indexSet_m)
                                                  \bigr) }\,.
    \label{eq:MeasureOnFactors}
  \end{equation}

  Given $p \in (0,\infty]$ and $\alpha,\theta \in \R$ define $w_m := m^\theta \cdot 2^{\alpha m}$,
  let $\calB_{\indexSet}$ denote the product $\sigma$-algebra on $\R^{\indexSet}$,
  and define $\generalProductMeasure$ as the product measure of the family
  $\bigl(\prob_{m}^{p, w_m}\bigr)_{m \in \N}$ (see e.g.~\mbox{\cite[Section~8.2]{DudleyRealAnalysis}}):
  \begin{equation}\label{eq:MeasureForInfiniteQ}
    \generalProductMeasure
    := \bigotimes_{m \in \bn}
         \prob_{m}^{p, w_m}
    : \calB_{\indexSet} \to [0,1] .
  \end{equation}
\end{defn}

With the help of the preceding results, we can now describe the construction of the measure
$\sequenceProductMeasure$ on $\sequenceSpaceSignalClass$, also for $q < \infty$.
A crucial tool will be the embedding result from Lemma~\ref{lem:ElementaryEmbeddings}.

\begin{defn}[Measures for $q<\infty$]\label{def:Measurenotinfty}
  Let the notation be as in Definition~\ref{def:Measureinfty}.

  For given $q \in (0,\infty]$, choose (according to Lemma~\ref{lem:ElementaryEmbeddings}) a constant
  $\kappa = \kappa(q) > 0$ (with $\kappa = 1$ if $q = \infty$) such that
  $\| \xvec \|_{\generalSpace{0}{q}} \leq \kappa \cdot \| \xvec \|_{\generalSpace{2/q}{\infty}}$
  for all $\xvec \in \R^{\indexSet}$, and define
  \[
    \sequenceProductMeasure : \calB_{\indexSet} \to [0,1],
    A \mapsto \generalProductMeasure{2/q} (\kappa \cdot A) .
  \]
\end{defn}

In the following, we verify that the measures defined according to Definitions~\ref{def:Measureinfty}
and \ref{def:Measurenotinfty} are indeed (Borel) probability measures on the signal classes
$\generalSignalClass{\theta}{\infty}$ and $\sequenceSpaceSignalClass$, respectively.
To do so, we first show that the signal classes are measurable with respect to the product
$\sigma$-algebra $\calB_{\indexSet}$, and we compare this $\sigma$-algebra to the Borel $\sigma$-algebra
on $\ell^2(\indexSet)$.

\begin{lem}\label{lem:ProductSigmaAlgebraLarge}
  Let $\calB_{\indexSet}$ denote the product $\sigma$-algebra on $\R^{\indexSet}$
  and let $p,q \in (0,\infty]$ and ${\alpha, \theta \in \R}$.
  Then the (quasi)-norm
  ${\| \cdot \|_{\generalSpace{\theta}{q}}} : \R^{\indexSet} \to [0,\infty]$
  is measurable with respect to $\calB_{\indexSet}$.
  In particular, $\generalSignalClass \in \calB_{\indexSet}$.

  Further, the Borel $\sigma$-algebra $\calB_{\ell^2}$ on $\ell^2(\indexSet)$ coincides with the
  trace $\sigma$-algebra ${\ell^2(\indexSet) \Cap \calB_{\indexSet}}$.
\end{lem}

\begin{proof}
  The (mainly technical) proof is deferred to Appendix~\ref{sec:TechnicalSequenceSpaceProofs}.
\end{proof}

\begin{lem}\label{lem:MeasuresAreProbabilityMeasures}
  (a) The measure $\generalProductMeasure$ is a probability measure on
      \({
        \big(
          \generalSignalClass{\theta}{\infty},
           \generalSignalClass{\theta}{\infty} \Cap \calB_{\indexSet}
        \big)
      }\).

  \smallskip{}

  (b) If $\alpha > d \cdot (\frac{1}{2} - \frac{1}{p})_+$,
      then $\sequenceSpaceSignalClass \subset \ell^2(\indexSet)$,
      and the measure $\sequenceProductMeasure$ is a probability measure on
      $\big( \sequenceSpaceSignalClass, \sequenceSpaceSignalClass \Cap \calB_{\ell^2} \big)$,
      where $\calB_{\ell^2}$ denotes the Borel $\sigma$-algebra on $\ell^2(\indexSet)$.
\end{lem}

\begin{proof}
  For the first part, Lemma~\ref{lem:ProductSigmaAlgebraLarge} implies that
  $\generalSignalClass{\theta}{\infty} \in \calB_{\indexSet}$, so that
  $\generalProductMeasure$ is a measure on $\generalSignalClass{\theta}{\infty} \Cap \calB_{\indexSet}$.
  Furthermore, Lemma~\ref{lem:ballprod} and Definition~\ref{def:Measureinfty} show
  $\generalProductMeasure(\generalSignalClass{\theta}{\infty}) = 1$.

  For the second part, recall from Proposition~\ref{prop:ApproximationRateLowerBound}
  that $\sequenceSpaceSignalClass \subset \ell^2(\indexSet)$,
  so that Lemma~\ref{lem:ProductSigmaAlgebraLarge} implies
  $\sequenceSpaceSignalClass \Cap \calB_{\ell^2} = \sequenceSpaceSignalClass \Cap \calB_{\indexSet}$,
  which easily implies that $\sequenceProductMeasure$ is a measure on
  $\sequenceSpaceSignalClass \Cap \calB_{\ell^2}$.
  Finally, observe that, by choice of $\kappa$, we have
  $\generalSignalClass{2/q}{\infty} \subset \kappa \cdot \sequenceSpaceSignalClass$,
  and hence
  \[
    1 \geq
    \sequenceProductMeasure (\sequenceSpaceSignalClass)
    \geq \sequenceProductMeasure (\kappa^{-1} \, \generalSignalClass{2/q}{\infty})
    =  \generalProductMeasure{2/q} (\generalSignalClass{2/q}{\infty})
    =  1.
    \qedhere
  \]
\end{proof}

\subsection{Proof of Theorem~\ref{thm:MainSequenceSpaceResult}}
\label{sub:QuantitativeInverseTheorem}

In this subsection, we prove that the measures $\sequenceProductMeasure$
constructed in Definition~\ref{def:Measurenotinfty} are critical, provided that
$\alpha > d \cdot (\frac{1}{2} - \frac{1}{p})_+$.
An essential ingredient for the proof is the following estimate for the volumes of balls
in $\ell^p ([m])$.

\begin{lem}\label{lem:ellPBallVolume}
  Let $m \in \N$ and $p \in (0,\infty]$.
  The $m$-dimensional Lebesgue measure of $\ball(0,1;\ell^p ([m]))$ is
  \begin{equation}\label{eq:lpballvol}
    \lambda_m \big( \ball(0,1;\ell^p([m])) \big)
    = \frac{2^m \cdot \big( \Gamma(1+ \frac 1 p) \big)^m}{\Gamma(1 + \frac{m}{p})}.
  \end{equation}

  For every $p \in (0,\infty]$ there exist constants $c_p, C_p \in (0,\infty)$,
  such that for all $m \in \bn$
  \begin{equation}
    c_p^m \cdot m^{-m (\pdiffs 2 p)}
    \leq \frac{\lambda_m \big( \ball(0,1;\ell^2([m])) \big)}
              {\lambda_m \big( \ball(0,1;\ell^p([m])) \big)}
    \leq C_p^m \cdot m^{-m (\pdiffs{2}{p})}\,.
    \label{eq:VolumeEstimate}
 \end{equation}
\end{lem}

\begin{proof}
  A proof of \eqref{eq:lpballvol} can be found e.g.~in \cite[Theorem~5]{Vybiral18}.

  For proving \eqref{eq:VolumeEstimate}, it is shown in \cite[Lemma~4]{Vybiral18}
  that for each $p \in (0,\infty)$ there are constants $\lambda_p, \Lambda_p > 0$ satisfying
  \begin{equation}
    \lambda_p \cdot x^{1/p}
    \leq \Big[ \Gamma\big(1 + \tfrac{x}{p}\big) \Big]^{1/x}
    \leq \Lambda_p \cdot x^{1/p}
    \quad \forall \, x \in [1,\infty).
    \label{eq:StirlingAlternative}
  \end{equation}
  It is clear that this remains true also for $p = \infty$; in fact, since $\Gamma(1) = 1$,
  one can simply choose $\lambda_\infty = \Lambda_\infty = 1$ in this case.

  By \eqref{eq:lpballvol}, we see that
  \[
    \frac{\lambda_m \big( \ball(0,1;\ell^2([m])) \big)}
         {\lambda_m \big( \ball(0,1;\ell^p([m])) \big)}
    = \bigg(
        \frac{\Gamma(1 + \frac{1}{2})}{\Gamma(1 + \frac{1}{p})}
      \bigg)^m
      \cdot \frac{\Gamma(1 + \frac{m}{p})}{\Gamma(1 + \frac{m}{2})} ,
  \]
  and the estimate~\eqref{eq:StirlingAlternative} implies
  \[
    \frac{\lambda_p^m \cdot m^{m/p}}{\Lambda_2^m \cdot m^{m/2}}
    \leq \frac{\Gamma(1 + \frac{m}{p})}
              {\Gamma(1 + \frac{m}{2})}
    \leq \frac{\Lambda_p^m \cdot m^{m/p}}
              {\lambda_2^m \cdot m^{m/2}}.
  \]
  Hence, we can choose
  \(
    C_p
    := \frac{\Gamma(1 + \frac{1}{2}) \strut}
            {\strut \Gamma(1 + \frac{1}{p})}
       \cdot \frac{\strut \Lambda_p}{\strut \lambda_2}
  \)
  and
  \(
    c_p
    := \frac{\Gamma(1 + \frac{1}{2}) \strut}
            {\strut \Gamma(1 + \frac{1}{p})}
       \cdot \frac{\strut \lambda_p}{\strut \Lambda_2}
    .
  \)
\end{proof}

We are finally equipped to prove Theorem~\ref{thm:MainSequenceSpaceResult}.

\begin{proof}[Proof of Theorem~\ref{thm:MainSequenceSpaceResult}]
  \textbf{Step 1:} We show for $s^\ast := \frac{\alpha}{d} - (\frac{1}{2} - \frac{1}{p})$
  and arbitrary $\theta \in \R$ that the measure $\generalProductMeasure$ has
  growth order $s^\ast$ with respect to $\ell^2(\indexSet)$.

  To this end, let $s > s^\ast$ be arbitrary, and let
  $\eps \in (0, \eps_0)$ (for a suitable $\eps_0 > 0$ to be chosen below),
  and $\xvec \in \ell^2(\indexSet)$.
  We estimate the measure $\generalProductMeasure ( \ball(\xvec, \eps; \ell^2(\indexSet)))$ by
  estimating the measure of certain finite-dimensional projections of the ball, exploiting the
  product structure of the measure:
  Recall the identification $\xvec = (\xvec_m)_{m \in \N}$, where $\xvec_m = \xvec|_{\indexSet_m}$.
  Set $w_m := m^\theta \cdot 2^{\alpha m}$ for $m \in \N$, as in Definition~\ref{def:Measureinfty}.
  For arbitrary $m \in \N$, we have
  \begin{align*}
     \ball(\xvec, \eps; \ell^2(\indexSet))
    \subset \prod_{t=1}^{m-1} \R^{\indexSet_t}
            \times \ball(\xvec_m, \eps ; \ell^2(\indexSet_m))
            \times \prod_{t=m+1}^\infty \R^{\indexSet_t}.
  \end{align*}
  Using the product structure of $\generalProductMeasure$
  (cf.~Equation~\eqref{eq:MeasureForInfiniteQ}), we thus see for each $m \in \N$ that
  \begin{align*}
    \generalProductMeasure \bigl( \ball(\xvec, \eps; \ell^2(\indexSet))\bigr)
    & \leq \prob_{m}^{p,w_m}
           \left(
             \ball\left(\xvec_m, \eps; \ell^2(\indexSet_m)\right)
           \right) &\qquad\\
    & \leq \frac{
                 \lambda_{m}
                 \left(
                   \ball\left(\xvec_m, \eps; \ell^2(\indexSet_m) \right)
                 \right)
                }
                {
                 \lambda_{m}
                 \left(
                   \ball\left( 0, w_m^{-1}; \ell^p(\indexSet_m) \right)
                 \right)
                }
    & \qquad \text{by Equation  \eqref{eq:MeasureOnFactors}}\,,\\
    & = \eps^{n_m} \, w_m^{n_m} \cdot
        \frac{
              \lambda_{m}
              \left(
                \ball ( 0, 1; \ell^2(\indexSet_m) )
              \right)
             }
             {
              \lambda_{m}
              \left(
                \ball( 0, 1; \ell^p(\indexSet_m) )
              \right)
             }
     & \qquad \text{for } n_m := |\indexSet_m|,\\
     & \leq \Big(
              C_{p} \cdot \eps \, w_m \cdot n_m^{- (\pdiffs 2 p)}
            \Big)^{n_m}
     &\qquad \text{by Lemma } \ref{lem:ellPBallVolume} .
  \end{align*}
  From~\eqref{eq:weightequival} we see that $n_m = 2^{d m} \, \eta_m$ with $\eta_m \in [a,A]$.
  Therefore, we conclude
  \begin{align*}
    w_m \, n_m^{-(\pdiffs{2}{p})}
    & = m^\theta \,
        2^{\alpha m} \,
        2^{-md(\pdiffs{2}{p})} \,
        \eta_m^{-(\pdiffs{2}{p})} \\
    & = m^\theta \,
        2^{md s^*} \,
        \eta_m^{-(\pdiffs{2}{p})}
      \leq K_1 \cdot 2^{md s}
  \end{align*}
  for a suitable constant $K_1 = K_1(s,\theta,\alpha,d,p,a,A) > 0$, since $s > s^\ast$.
  Therefore,
  \begin{equation}\label{eq:fundmeasin}
    \generalProductMeasure \bigl( \ball(\xvec, \eps; \ell^2(\indexSet)) \bigr)
    \leq \big(
           C_{p} K_1
           \cdot \eps \, 2^{m d s}
         \big)^{2^{m d} \eta_m}
    \leq \big(
           K_2 \cdot \eps \cdot 2^{m d s}
         \big)^{2^{m d} \eta_m}
  \end{equation}
  for a suitable constant $K_2 = K_2 (s,\theta,\alpha,d,p,a,A) > 0$ and for arbitrary $m \in \N$.
  A candidate for an upper bound for $\generalProductMeasure ( \ball(\xvec, \eps; \ell^2(\indexSet)))$
  is a positive integer close to
  \[
    \widetilde{m} (\eps)
    := \mathop{\mathrm{argmin}}_{m\in \br}
         \big(
           K_2 \, \eps \, 2^{m s d}
         \big)^{2^{m d}}
    = -\frac{\log_2(K_2 \cdot \eps)}{d s}
       - \frac{\log_2e}{d} \,.
  \]
  Choose a positive $\eps_0 = \eps_0 (s,\theta,\alpha,d,p,a,A)$ so small that
  $\widetilde m(\eps) > 1$ for all $\eps \in (0, \eps_0)$.
  Set $m_0 := \lfloor \widetilde{m}(\eps) \rfloor \in \N$.
  By construction, $2^{d s \cdot \widetilde{m}(\eps)} = \frac{e^{- s}}{K_2 \cdot \eps}$,
  and hence $K_2 \, \eps \, 2^{d s \cdot m_0} \leq e^{- s } < 1$.

  For the exponent in \eqref{eq:fundmeasin}, observe that
  \[
    2^{d \, m_0} \, \eta_{m_0}
    \geq a \cdot 2^{d \cdot (\widetilde m(\eps) - 1)}
    =    \frac{a}{2^d} \cdot \big( 2^{d s \widetilde m(\eps)} \big)^{1/ s}
    =    \frac{a}{2^d \cdot e \cdot K_2^{1/s}} \cdot \eps^{-1/ s}
    =    K_3 \cdot \eps^{-1/s}
  \]
  for a constant $K_3 = K_3(s,\theta,\alpha,d,p,a,A)$.
  Now~\eqref{eq:fundmeasin} can be estimated further, yielding
  \[
    \generalProductMeasure \bigl( \ball(\xvec, \eps; \ell^2(\indexSet)) \bigr)
    \leq \big( K_2 \cdot \eps \cdot 2^{m_0 \, d s} \big)^{K_3 \cdot \eps^{-1/ s}}
    \leq e^{-{K_3 s} \cdot \eps^{-1/s}}
    =    2^{- K_4 \cdot \eps^{-1/s}},
  \]
  for a suitable constant $K_4 = K_4 (s,\theta,\alpha,d,p,a,A) > 0$.
  Since $s > s^\ast$ was arbitrary, this shows that $\generalProductMeasure$
  is of logarithmic growth order $s^\ast$; see Definition~\ref{def:GrowthRate}.

  \medskip{}

  \noindent
  \textbf{Step 2:}
  We show that $\sequenceProductMeasure$ is of growth order $s^\ast$ with respect to $\ell^2(\indexSet)$
  on $\sequenceSpaceSignalClass$.

  To see this, let $s > s^\ast$ be arbitrary, and choose (by virtue of Step~1) $\eps_0, c > 0$ such that
  \(
    \generalProductMeasure{2/q} \bigl( \ball(\xvec,\eps;\ell^2(\indexSet)) \bigr)
    \leq 2^{-c \,\cdot \eps^{-1/s}}
  \)
  for all $\xvec \in \ell^2(\indexSet)$ and $\eps \in (0, \eps_0)$.
  Recall from Definition~\ref{def:Measurenotinfty} that
  $\sequenceProductMeasure(M) = \generalProductMeasure{2/q}(\kappa M)$
  for a suitable $\kappa = \kappa(q) > 0$.
  Define $\eps_0' := \eps_0 / \kappa$ and $c' := c \, \kappa^{-1/s}$.

  Now, if $\eps \in (0,\eps_0')$, then $\kappa \eps \in (0,\eps_0)$ and hence
  \begin{align*}
    \sequenceProductMeasure\bigl( \ball(\xvec, \eps; \ell^2(\indexSet)) \bigr)
    & = \generalProductMeasure{2/q} \bigl( \kappa \ball(\xvec, \eps; \ell^2(\indexSet)) \bigr) \\
    & = \generalProductMeasure{2/q} \bigl( \ball(\kappa\xvec, \kappa\eps; \ell^2(\indexSet)) \bigr) \\
    & \leq 2^{-c \cdot (\kappa \eps)^{-1/s}}
      = 2^{-c' \cdot \eps^{-1/s}} ,
  \end{align*}
  proving that $\sequenceProductMeasure$ is of growth order $s^\ast$ with respect to $\ell^2(\indexSet)$.

  \medskip{}

  \noindent
  \textbf{Step 3:} \emph{(Completing the proof):}
  By Proposition~\ref{prop:ApproximationRateLowerBound}, $\sequenceSpaceSignalClass \subset \ell^2(\indexSet)$
  is compact with
  \(
    \optRate<\sequenceSpaceSignalClass>{\ell^2(\indexSet)}
    \geq s^\ast
    .
  \)
  By Step~2 and Lemma~\ref{lem:MeasuresAreProbabilityMeasures},
  $\sequenceProductMeasure$ is a Borel probability measure on $\sequenceSpaceSignalClass$
  of growth order $s^\ast$ with respect to $\ell^2(\indexSet)$.
  Thus, Lemma~\ref{lem:OneSidedBoundsSufficeForExactRate} shows that
  \(
    \optRate<\sequenceSpaceSignalClass>{\ell^2(\indexSet)}
    = s^\ast
  \)
  and that $\sequenceProductMeasure$ is critical for $\sequenceSpaceSignalClass$
  with respect to $\ell^2(\indexSet)$.
\end{proof}

\begin{rem*}
  The proof borrows its main idea (using the product measure structure of $\generalProductMeasure$
  to work on finite dimensional projections) from \cite{MMR99}.
\end{rem*}

\section{Examples}
\label{sec:Examples}

\subsection{Besov spaces on bounded open sets \texorpdfstring{$\Omega \subset \R^d$}{Ω ⊂ ℝᵈ}}

For Besov spaces on bounded domains,
we obtain the following consequence of Theorem~\ref{thm:MainSequenceSpaceResult},
by using suitable wavelet bases to ``transport'' the measure
$\sequenceProductMeasure$ to the Besov spaces.

For a review of the definition of Besov spaces (on $\R^d$ and on domains),
and the characterization of these spaces by wavelets, we refer to
Appendices~\ref{sub:BesovFourierDefinition} and \ref{sub:BesovWaveletCharacterization}.

\begin{thm}\label{thm:mainbesovresult}
  Let $\emptyset \neq \Omega \subset \R^d$ be open and bounded,
  let $p,q \in (0,\infty]$, and $\BesovSmoothness \in \R$ with
  $\BesovSmoothness > d \cdot (p^{-1} - 2^{-1})_{+}$.

  \setlength{\leftmargini}{0.9cm}
  Then
  \begin{enumerate}[label={(\roman*)}]
    \item $\signalClass := \ball \big( 0, 1; B_{p,q}^{\BesovSmoothness} (\Omega; \R) \big)$
          is a compact subset of $L^2(\Omega)$, and $\optRateSmall{L^2(\Omega)} = \frac{\tau}{d}$;

    \item there is a Borel probability measure $\P$ on $\signalClass$ that is critical for $\signalClass$
          with respect to $L^2(\Omega)$;

    \item \label{enu:BesovCriticalCodecExists}
          there is a codec $\code = \big( (E_R, D_R) \big)_{R \in \N} \in \codecs{L^2(\Omega)}$
          with $\distortion{L^2(\Omega)} (E_R, D_R) \lesssim R^{-\frac{\tau}{d}}$.
  \end{enumerate}
\end{thm}

\begin{rem*}
  In the discussion following Theorem~\ref{thm:DichotomyBanach}, we observed that the existence
  of a critical measure in general leaves open what happens for $s = s^\ast$.
  In the case of Besov spaces, the above theorem shows that the compression rate $s = s^\ast$
  is actually achieved by a suitable codec.
\end{rem*}

\begin{proof}
  Define $\alpha := \BesovSmoothness + d \cdot (2^{-1} - p^{-1})$, noting that
  \[
    \alpha
    > d \cdot \bigl[ (p^{-1} - 2^{-1})_{+} + (2^{-1} - p^{-1}) \bigr]
    = d \cdot (2^{-1} - p^{-1})_{+}
    \,\, ,
  \]
  so that $\alpha$ satisfies the assumptions of Theorem~\ref{thm:MainSequenceSpaceResult}.

  Using the wavelet characterization of Besov spaces, it is shown in
  Appendix~\ref{sub:WaveletsForBesovSpacesOnDomains}%
  \footnote{Precisely, this follows by combining Lemmas~\ref{lem:BesovSequenceSpaceConnection}
  and \ref{lem:WaveletSequenceSpacesAreNice} and by taking
  $Q_{\interior} = T_{\interior} \circ \iota_{\interior}$ and ${Q_{\ext} = T_{\ext} \circ \iota_{\ext}}$.}
  that there are countably infinite index sets $J^{\ext}, J^{\interior}$
  with associated $d$-regular partitions $\partition^{\ext} = \big( \indexSet_m^{\ext} \big)_{m \in \N}$
  and $\partition^{\interior} = \big( \indexSet_m^{\interior} \big)_{m \in \N}$
  and such that there are linear maps
  \[
    Q_{\interior} :
    \mixSpace[\partition^{\interior}]{q} \to B_{p,q}^{\BesovSmoothness} (\Omega ; \R)
    \quad \text{and} \quad
    Q_{\ext} :
    \mixSpace[\partition^{\ext}]{q} \to B_{p,q}^{\BesovSmoothness} (\Omega ; \R)
  \]
  with the following properties:
  \begin{enumerate}
    \item $\mixSpace[\partition^{\interior}]{q} \hookrightarrow \ell^2(J^{\interior})$
          and $\mixSpace[\partition^{\ext}]{q} \hookrightarrow \ell^2(J^{\ext})$;
          this follows from Proposition~\ref{prop:ApproximationRateLowerBound}.

    \item There is some $\gamma > 0$ such that
          $\| Q_\interior \, \cvec \|_{L^2(\Omega)} = \gamma \cdot \| \cvec \|_{\ell^2} < \infty$
          and furthermore
          \(
            \| Q_\interior \, \cvec \|_{B_{p,q}^\tau (\Omega)}
            \leq \| \cvec \|_{\mixSpace[\partition^{\interior}]{q}}
          \)
          for all $\cvec \in \mixSpace[\partition^{\interior}]{q}$.

    \item There is $\varrho > 0$ such that
          $\| Q_{\ext} \, \cvec \|_{L^2(\Omega)} \leq \varrho \cdot \| \cvec \|_{\ell^2} < \infty$
          for all $\cvec \in \mixSpace[\partition^{\ext}]{q}$, and
          \begin{equation}
            \ball \big( 0, 1; B_{p,q}^{\BesovSmoothness} (\Omega; \R) \big)
            \subset Q_{\ext} \big(
                               \ball \big( 0, 1; \mixSpace[\partition^{\ext}]{q} \big)
                             \big)
            \subset L^2 (\Omega) .
            \label{eq:BesovProofTExtProperty}
          \end{equation}
  \end{enumerate}
  Furthermore, Theorem~\ref{thm:MainSequenceSpaceResult} shows that\vspace*{-0.1cm}
  \[
    \optRateSmall<\sequenceSpaceSignalClass<\partition^{\interior}>>{\ell^2(J^{\interior})}
    = \optRateSmall<\sequenceSpaceSignalClass<\partition^{\ext}>>{\ell^2(J^{\ext})}
    = \frac{\alpha}{d} - \Bigl( \frac{1}{2} - \frac{1}{p} \Bigr)
    = \frac{\tau}{d}
  \]
  and that there exists a Borel probability measure $\P_0$ on
  $\ball (0,1;\mixSpace[\partition^{\interior}]{q})$ that is critical
  for $\ball (0,1;\mixSpace[\partition^{\interior}]{q})$ with respect to $\ell^2(J^{\interior})$.
  Therefore, we can apply Theorem~\ref{thm:LipschitzTransferResult} with the choices
  ${\banachOne = \ell^2(J^{\ext})}$, $\banachTwo = \ell^2(J^{\interior})$ and $\banachThree = L^2(\Omega)$
  as well as
  \[
    \signalClass_{\banachOne}
    = \ball\bigl(0,1;\mixSpace[\partition^{\ext}]{q}\bigr),
    \quad
    \signalClass_{\banachTwo} = \ball\bigl(0,1;\mixSpace[\partition^{\interior}]{q}\bigr),
    \quad \text{and} \quad
    \signalClass = \ball\bigl(0,1; B_{p,q}^\tau (\Omega;\R)\bigr) ,
  \]
  and finally $\Phi = Q_{\ext}$, $\Psi = Q_{\interior}$, and $\kappa = \gamma$.
  This theorem then shows ${\optRateSmall{L^2(\Omega)} = \frac{\tau}{d} > 0}$
  (in particular, $\signalClass \subset L^2(\Omega)$ is totally bounded and hence compact,
  since $\signalClass \subset L^2(\Omega)$ is closed by Lemma~\ref{lem:BesovBallsMeasurable})
  and that $\P := \P_0 \circ Q_{\interior}^{-1}$ is a Borel probability measure on $\signalClass$
  that is critical for $\signalClass$ with respect to $L^2(\Omega)$.

  Finally, Proposition~\ref{prop:ApproximationRateLowerBound} yields a codec
  \({
    \code^\ast
    = \big( (E_R^\ast, D_R^\ast) \big)_{R \in \N}
    \in \codecs<\signalClass_{\banachOne}>{\ell^2(J^{\ext})}
  }\)
  satisfying $\distortion<\signalClass^\ast>{\ell^2(J^\ext)} (E_R^\ast, D_R^\ast) \lesssim R^{-\frac{\tau}{d}}$.
  Furthermore, $Q_{\ext}$ is Lipschitz (with respect to ${\| \cdot \|_{\ell^2}}$ and $\| \cdot \|_{L^2}$)
  and satisfies \eqref{eq:BesovProofTExtProperty};
  thus, the remark after Lemma~\ref{lem:ApproximationRateTransference} shows that
  $\distortion<\signalClass>{L^2(\Omega)} (E_R, D_R) \lesssim R^{- \frac{\tau}{d}}$ for a suitable codec
  ${\code = \big( (E_R, D_R) \big)_{R \in \N} \in \codecs<\signalClass>{L^2(\Omega)}}$.
\end{proof}

\subsection{Sobolev spaces on Lipschitz domains \texorpdfstring{$\Omega \subset \R^d$}{Ω ⊂ ℝᵈ}}

Let $\emptyset \neq \Omega \subset \R^d$ be an open bounded Lipschitz domain
(precisely, we require $\Omega$ to satisfy the conditions
in \cite[Chapter~VI, Section~3.3]{SteinSingularIntegrals}).
We consider the usual \emph{Sobolev spaces} $W^{k,p}(\Omega)$
($k \in \N$ and $p \in [1,\infty]$), and prove that also for these spaces,
the phase transition phenomenon holds.
To be completely explicit, we endow the space $W^{k,p}(\Omega)$ with the following norm:
\begin{equation}
  \| f \|_{W^{k,p} (\Omega)}
  := \max_{|\alpha| \leq k}
       \| \partial^\alpha f \|_{L^p (\Omega)} .
  \label{eq:SobolevNormDefinition}
\end{equation}

Our phase-transition result reads as follows:

\begin{thm}\label{thm:SobolevPhaseTransition}
  Let $\emptyset \neq \Omega \subset \R^d$ be an open bounded Lipschitz domain.
  Let $k \in \N$ and $p \in [1,\infty]$, and define
  $\signalClass := \ball \big( 0, 1; W^{k,p}(\Omega) \big)$.
  \setlength{\leftmargini}{0.9cm}
  If $k > d \cdot (p^{-1} - 2^{-1})_{+}$, then
  \begin{enumerate}[label={(\roman*)}]
    \item $\signalClass \subset L^2(\Omega)$ is bounded
          and Borel measurable and satisfies $\optRateSmall{L^2(\Omega)} = \frac{k}{d}$;

    \item there is a Borel probability measure $\P$ on $\signalClass$ that is critical for
          $\signalClass$ with respect to $L^2(\Omega)$;

    \item \label{enu:SobolevCriticalCodecExists}
          there is a codec $\code = \big( (E_R, D_R) \big)_{R \in \N} \in \codecs{L^2(\Omega)}$
          with $\distortion{L^2(\Omega)} (E_R, D_R) \lesssim R^{-\frac{k}{d}}$.
  \end{enumerate}
\end{thm}

\begin{rem*}
  1) As for the case of Besov spaces, the theorem shows that the critical rate $s = s^\ast = \frac{k}{d}$
  is actually attained by a suitable codec.

  \medskip{}

  2) The condition $k > d \cdot (p^{-1} - 2^{-1})_{+}$ is \emph{equivalent}
  to $\signalClass \subset L^2(\Omega)$ being precompact.
  The sufficiency is a consequence of the \emph{Rellich-Kondrachov theorem};
  see \cite[Theorem~6.3]{AdamsSobolevSpaces}.
  For the converse implication, note that $k > d \cdot (p^{-1} - 2^{-1})_+$
  trivially holds for $p \geq 2$.
  In the remaining case $p < 2$, one can consider the sequence
  $\psi_n (x) := c \cdot n^{\frac{d}{p} - k} \cdot \psi( n \cdot (x - x_0))$,
  where $c > 0$, $x_0 \in \Omega$, and $\psi \in C_c^\infty (\R^d)$.
  It is easy to see that $\psi_n \in \signalClass$ for all $n \in \N$, for a suitable choice of $c > 0$,
  while $\psi_n \to 0$ almost everywhere, so that if $\signalClass \subset L^2(\Omega)$
  is precompact, then $\| \psi_n \|_{L^2} \to 0$, which easily implies
  $k > d \cdot (p^{-1} - 2^{-1})_+$.
\end{rem*}

\begin{proof}[Proof of Theorem~\ref{thm:SobolevPhaseTransition}]
  We present here the proof for the case $p \in (1,\infty)$, where we will see that
  the claim follows from that for the Besov spaces.
  For the case $p \in \{1,\infty\}$, the proof is more involved, and thus postponed
  to Appendix~\ref{sec:SobolevExceptionalCaseProof}.

  First, the Rellich-Kondrachov compactness theorem (see \cite[Theorem~6.3]{AdamsSobolevSpaces})
  shows that $W^{k,p}(\Omega)$ embeds compactly into $L^2(\Omega)$.
  In particular, $\signalClass = \ball (0,1; W^{k,p}(\Omega)) \subset L^2(\Omega)$
  is bounded; in fact, $\signalClass$ is also compact (hence Borel measurable)
  by reflexivity of $W^{k,p}(\Omega)$\footnote{
    Indeed, if $(f_n)_{n \in \N} \subset \signalClass$ is arbitrary, then since $W^{k,p}(\Omega)$
    is reflexive (see \cite[Example~8.11]{AltFA}), the closed unit ball in $W^{k,p}(\Omega)$ is
    weakly sequentially compact (see \cite[Theorem~8.10]{AltFA}), so that
    there is a subsequence $(f_{n_\ell})_{\ell \in \N}$ satisfying
    $f_{n_\ell} \,\raisebox{-0.05cm}{$\xrightharpoonup{W^{k,p}(\Omega)}$}\, f \in \signalClass$.
    Again by compactness of the embedding $W^{k,p}(\Omega) \hookrightarrow L^2(\Omega)$,
    this implies $f_{n_\ell} \,\raisebox{-0.05cm}{$\xrightarrow{L^2}$}\, f \in \signalClass$,
    showing that $\signalClass \subset L^2(\Omega)$ is compact.
  }.

  Define $\widetilde{p} := \min \{p, 2\}$ and $\widehat{p} := \max \{p, 2\}$,
  as well as $\signalClass_s := \ball \big( 0, 1; B_{p,\widetilde{p}}^k (\Omega) \big)$
  and $\signalClass_b := \ball \big( 0, 1; B_{p,\widehat{p}}^k (\Omega) \big)$.
  We will prove below that there are constants $C_1, C_2 > 0$ such that
  \begin{equation}
    C_1^{-1} \cdot \signalClass_s
    = \ball \big( 0, C_1^{-1}; B_{p,\widetilde{p}}^k (\Omega) \big)
    \subset \signalClass
    \subset \ball \big( 0, C_2; B_{p,\widehat{p}}^k (\Omega) \big)
    = C_2 \cdot \signalClass_b .
    \label{eq:BesovSobolevConnection}
  \end{equation}
  Assuming this for the moment, recall from Theorem~\ref{thm:mainbesovresult} that
  \(
    \optRateSmall<\signalClass_b>{L^2(\Omega)}
    = \optRateSmall<\signalClass_s>{L^2(\Omega)}
    = \frac{k}{d}
  \)
  and that there exists a Borel probability measure $\P_0$ on $\signalClass_s$ that is critical
  for $\signalClass_s$ with respect to $L^2(\Omega)$.
  Define $\banachOne := \banachTwo := \banachThree := L^2(\Omega)$ and
  $\signalClass_{\banachOne} := \signalClass_b$, $\signalClass_{\banachTwo} := \signalClass_s$,
  as well as
  \[
    \Phi : \signalClass_b \to L^2(\Omega), f \mapsto C_2 \cdot f
    \qquad \text{and} \qquad
    \Psi : \signalClass_s \to \signalClass, f \mapsto C_1^{-1} \cdot f.
  \]
  Using \eqref{eq:BesovSobolevConnection}, one easily checks that all assumptions of
  Theorem~\ref{thm:LipschitzTransferResult} are satisfied.
  An application of that theorem shows that $\optRateSmall{L^2(\Omega)} = \frac{k}{d}$
  and that $\P := \P_0 \circ \Psi^{-1}$ is a Borel probability measure on $\signalClass$
  that is critical for $\signalClass$ with respect to $L^2(\Omega)$.

  Finally, Part~\ref{enu:BesovCriticalCodecExists} of Theorem~\ref{thm:mainbesovresult}
  yields a codec
  $\code^\ast = \big( (E_R^\ast, D_R^\ast) \big)_{R \in \N} \in \codecs<\signalClass_b>{L^2(\Omega)}$
  satisfying $\distortion<\signalClass_b>{L^2(\Omega)} (E_R^\ast, D_R^\ast) \lesssim R^{-\frac{k}{d}}$.
  Since $\Phi$ is Lipschitz continuous (with respect to the $L^2$-norm) with
  $\signalClass \subset \Phi(\signalClass_b)$, the remark
  after Lemma~\ref{lem:ApproximationRateTransference} provides a codec
  $\code = \big( (E_R, D_R) \big)_{R \in \N} \in \codecs{L^2(\Omega)}$ satisfying
  ${\distortion{L^2(\Omega)}(E_R, D_R) \lesssim R^{-\frac{k}{d}}}$ as well.
  This establishes Property~\ref{enu:SobolevCriticalCodecExists} of the current theorem.

  It remains to prove \eqref{eq:BesovSobolevConnection}.
  First, a combination of \cite[Theorem in Section 2.5.6]{TriebelTheoryOfFunctionSpaces1}
  and \cite[Proposition 2 in Section 2.3.2]{TriebelTheoryOfFunctionSpaces1} shows for
  the so-called \emph{Triebel-Lizorkin spaces}%
  \footnote{The precise definition of these spaces is immaterial for us.
  We merely remark that the identity $F_{p,2}^k (\R^d) = W^{k,p}(\R^d)$ is only valid
  for $p \in (1,\infty)$.}
  $F^k_{p,2}(\R^d)$ that
  \[
    B_{p,\widetilde{p}}^k (\R^d)
    \hookrightarrow F_{p,2}^k (\R^d) = W^{k,p}(\R^d)
    \hookrightarrow B_{p, \widehat{p}}^k (\R^d) .
  \]
  Hence, there are $C_3, C_4 > 0$ satisfying
  $\| f \|_{W^{k,p}(\R^d)} \leq C_3 \cdot \| f \|_{B_{p,\widetilde{p}}^k (\R^d)}$
  for all $f \in B_{p,\widetilde{p}}^k (\R^d)$, and
  $\| f \|_{B_{p,\widehat{p}}^k (\R^d)} \leq C_4 \cdot \| f \|_{W^{k,p}(\R^d)}$
  for all $f \in W^{k,p}(\R^d)$.
  Furthermore, since $\Omega$ is a Lipschitz domain,
  \cite[Chapter VI, Theorem~5]{SteinSingularIntegrals}
  shows that there is a bounded linear ``extension operator''
  $\extension : W^{k,p} (\Omega) \to W^{k,p}(\R^d)$
  satisfying $(\extension f)|_{\Omega} = f$ for all $f \in W^{k,p} (\Omega)$.

  It is now easy to prove the inclusion~\eqref{eq:BesovSobolevConnection},
  with $C_1 := C_3$ and $C_2 := C_4 \cdot \| \extension \|$.
  First, if $f \in \signalClass_s$ and $\eps > 0$, then there is
  $g \in B^k_{p,\widetilde{p}}(\R^d)$ satisfying $f = g|_\Omega$ and
  $\| g \|_{B^k_{p,\widetilde{p}}(\R^d)} \leq 1 + \eps$,
  and hence
  \(
    \| f \|_{W^{k,p}(\Omega)}
    = \| g|_\Omega \|_{W^{k,p}(\Omega)}
    \leq \| g \|_{W^{k,p}(\R^d)}
    \leq C_3 \cdot (1 + \eps) .
  \)
  Since this holds for all $\eps > 0$, we see that
  $\| C_1^{-1} \, f \|_{W^{k,p} (\Omega)} \leq 1$; that is, $C_1^{-1} f \in \signalClass$.

  Conversely, if $f \in \signalClass$,
  then $g := \extension f \in W^{k,p}(\R^d) \subset B^k_{p,\widehat{p}} (\R^d)$
  and $f = g|_{\Omega}$, which implies
  \(
    \strut
    \| f \|_{B^k_{p,\widehat{p}} (\Omega)}
    \leq \| g \|_{B^k_{p,\widehat{p}} (\R^d)}
    \leq C_4 \, \| g \|_{W^{k,p} (\R^d)}
    \leq C_4 \, \| \extension \| \cdot \| f \|_{W^{k,p}(\Omega)}
    \leq C_2,
  \)
  and hence $f \in C_2 \cdot \signalClass_b$.
\end{proof}

\appendix

\section{Transferring approximation rates and measures}%
\label{sec:TransferResultsProofs}

In this appendix, we provide the proof of Theorem~\ref{thm:LipschitzTransferResult}.
Along the way we will show that expansive maps can be used to transfer measures with
a certain growth order from one set to another, while Lipschitz maps can be used
to transfer estimates for the optimal compression rate from one set to another.

\begin{lem}\label{lem:CriticalMeasuresPushForward}
  Let $\banach, \otherBanach$ be Banach spaces and let $\signalClass \subset \banach$
  and $\signalClass' \subset \otherBanach$.
  Let $\Phi : \signalClass \to \signalClass'$ be measurable
  (with respect to the trace $\sigma$-algebra of the Borel $\sigma$-algebras)
  and \emph{expansive}, in the sense that there is $\kappa > 0$ such that
  \[
    \| \Phi(\xvec) - \Phi(\xvec') \|_{\otherBanach}
    \geq \kappa \cdot \| \xvec - \xvec' \|_{\banach}
    \qquad \forall \, \xvec, \xvec' \in \signalClass .
  \]
  If $s_0 \geq 0$ and if $\P$ is a Borel probability measure on $\signalClass$ of growth order $s_0$,
  then the push-forward measure $\P \circ \Phi^{-1}$ is a Borel probability measure on $\signalClass'$
  of growth order $s_0$ as well.
\end{lem}

\begin{proof}
  Since $\Phi : \signalClass \to \signalClass '$ is measurable, $\nu := \P \circ \Phi^{-1}$
  is a Borel probability measure on $\signalClass'$.

  To prove that $\nu$ has growth order $s_0$, let $s > s_0$ be arbitrary.
  Since $\P$ is of growth order $s_0$, there are $\eps_0, \gc > 0$
  such that Equation~\eqref{eq:CriticalMeasureDefinition} is satisfied.
  Define $\eps_0 ' := \frac{\kappa}{2} \cdot \eps_0$ and
  $\gc' := \gc \cdot (2 \, \kappa^{-1})^{-1/s} = 2^{-1/s} \gc \cdot \kappa^{1/s}$.
  We claim that
  \(
    \nu \big( \signalClass' \cap \ball(\yvec, \eps; \otherBanach) \big)
    \leq 2^{- \gc' \cdot \eps^{-1/s}}
  \)
  for all $\yvec \in \otherBanach$ and all $\eps \in (0, \eps_0')$;
  this will show that $\nu$ has growth order $s_0$.

  \smallskip{}

  The estimate is trivial if $\Phi(\signalClass) \cap \ball(\yvec, \eps; \otherBanach) = \emptyset$,
  since then
  \(
    \Phi^{-1} \big( \ball (\yvec, \eps ; \otherBanach) \big) = \emptyset ,
  \)
  and hence
  \(
    \nu \big( \signalClass' \cap \ball(\yvec, \eps; \otherBanach) \big)
    = \P \big(
            \Phi^{-1}
            ( \signalClass' \cap \ball(\yvec, \eps; \otherBanach) )
          \big)
    = \P (
            \emptyset
          )
    = 0 .
  \)
  Therefore, let us assume that
  $\emptyset \neq \Phi(\signalClass) \cap \ball(\yvec, \eps; \otherBanach) \ni \yvec'$;
  say $\yvec' = \Phi(\xvec')$ for some $\xvec ' \in \signalClass$.
  Now, for arbitrary
  \(
    \xvec
    \in \Phi^{-1} \big(
                    \signalClass' \cap \ball(\yvec, \eps; \otherBanach)
                  \big)
    \subset \signalClass ,
  \)
  we have
  \[
    \| \xvec - \xvec' \|_{\banach}
    \leq \kappa^{-1} \,
         \| \Phi(\xvec) - \Phi(\xvec') \|_{\otherBanach}
    \leq \kappa^{-1} \cdot
         \big(
           \| \Phi(\xvec) - \yvec \|_{\otherBanach}
           + \| \yvec - \Phi(\xvec') \|_{\otherBanach}
         \big)
    \leq 2 \cdot \kappa^{-1} \cdot \eps .
  \]
  We have thus shown
  \(
    \Phi^{-1} (\signalClass' \cap \ball(\yvec, \eps; \otherBanach))
    \subset \signalClass \cap \ball(\xvec', \frac{2}{\kappa} \, \eps; \banach) .
  \)
  Since $\frac{2}{\kappa} \eps < \frac{2}{\kappa} \eps_0 ' = \eps_0$,
  we see by Property~\eqref{eq:CriticalMeasureDefinition} as claimed that
  \begin{align*}
    \nu \big( \signalClass ' \cap \ball(\yvec, \eps, \otherBanach) \big)
    & = \P \big(
              \Phi^{-1}
              ( \signalClass ' \cap \ball (\yvec, \eps; \otherBanach) )
            \big) \\
    & \leq \P \Big(
                 \signalClass
                 \cap \ball \big( \xvec', \tfrac{2}{\kappa} \eps; \banach \big)
               \Big) \\
    & \leq 2^{-\gc \cdot (2 \kappa^{-1} \eps)^{-1/s}}
      =    2^{- \gc' \cdot \eps^{-1/s}} .
    \qedhere
  \end{align*}
\end{proof}

As a kind of converse of the previous result, we now show that Lipschitz maps can be used
to obtain bounds for the optimal compression rate $\optRateSmall{\banach}$ of a signal class
$\signalClass \subset \banach$.

\begin{lem}\label{lem:ApproximationRateTransference}
  Let $\banach, \otherBanach$ be Banach spaces, and let $\signalClass \subset \banach$
  and $\signalClass' \subset \otherBanach$.
  Assume that $\Phi : \signalClass \to \otherBanach$ is Lipschitz continuous,
  and that $\Phi(\signalClass) \supset \signalClass'$.
  Then $\optRateSmall<\signalClass'>{\otherBanach} \geq \optRateSmall<\signalClass>{\banach}$.
\end{lem}

\begin{rem*}
  The proof shows that if there exists a codec
  $\code = \big( (E_R, D_R) \big)_{R \in \N} \in \codecs{\banach}$
  satisfying $\distortion{\banach} (E_R, D_R) \lesssim R^{-s}$ for some $s \geq 0$,
  then one can construct a modified codec
  ${\code^\ast = \big( (E_R^\ast, D_R^\ast) \big)_{R \in \N} \in \codecs<\signalClass'>{\otherBanach}}$
  satisfying $\distortion<\signalClass'>{\otherBanach} (E_R^\ast, D_R^\ast) \lesssim R^{-s}$
  as well.
\end{rem*}

\begin{proof}
  The claim is clear if $\optRateSmall<\signalClass>{\banach} = 0$.
  Thus, let us assume $\optRateSmall<\signalClass>{\banach} > 0$,
  and let ${s \in [0, \optRateSmall<\signalClass>{\banach})}$ be arbitrary.
  Then there is a codec $\code = \big( (E_R, D_R) \big)_{R \in \N} \in \codecs<\signalClass>{\banach}$
  and a constant $C > 0$ such that $\distortion<\signalClass>{\banach} ( E_R, D_R ) \leq C \cdot R^{-s}$
  for all $R \in \N$.
  Let $L > 0$ denote a Lipschitz constant for $\Phi$.

  Now, for $\eps > 0$ and $\xvec \in \banach$, choose $\Psi_\eps (\xvec) \in \signalClass$
  such that $\| \xvec - \Psi_\eps (\xvec) \|_{\banach} \leq \eps + \dist (\xvec, \signalClass)$,
  and let
  \[
    D_R^\ast : \{0,1\}^R \to \otherBanach, c \mapsto \Phi \big( \Psi_{R^{-s}} (D_R(c)) \big)
    \qquad \text{for } R \in \N .
  \]
  Now, if $\yvec \in \signalClass' \subset \Phi(\signalClass)$ is arbitrary,
  then $\yvec = \Phi(\xvec)$ for some $\xvec \in \signalClass$, and hence
  \begin{align*}
    \| \yvec - D_R^\ast (E_R(\xvec)) \|_{\otherBanach}
    & = \Big\|
          \Phi(\xvec) - \Phi \Big( \Psi_{R^{-s}} \big( D_R(E_R(\xvec)) \big) \Big)
        \Big\|_{\otherBanach} \\
    & \leq L \cdot \big\| \xvec - \Psi_{R^{-s}} \big( D_R( E_R(\xvec)) \big) \big\|_{\banach} \\
    & \leq L \cdot \big[
                     \| \xvec - D_R(E_R(\xvec)) \|_{\banach}
                     + \big\| D_R(E_R(\xvec)) - \Psi_{R^{-s}} \big( D_R( E_R(\xvec)) \big) \big\|_{\banach}
                   \big] \\
    & \leq L \cdot \big[
                     C \cdot R^{-s}
                     + R^{-s}
                     + \dist \! \big( D_R(E_R(\xvec)), \signalClass \big)
                   \big]
      \leq L \cdot (1 + 2C) \cdot R^{-s},
  \end{align*}
  since
  \(
    \dist \! \big( D_R(E_R(\xvec) \big), \signalClass)
    \leq \| D_R(E_R(\xvec)) - \xvec \|_{\banach} \leq C \cdot R^{-s}
  \).
  Therefore, if for each $\yvec \in \signalClass'$ and $R \in \N$
  we choose $c_{\yvec, R} \in \{0,1\}^R$ with
  \(
    \| \yvec - D_R^\ast (c_{\yvec, R}) \|_{\otherBanach}
    = \min_{c \in \{0,1\}^R} \| \yvec - D_R^\ast (c) \|_{\otherBanach}
  \)
  and define $E_R^\ast : \signalClass' \to \{0,1\}^R, \yvec \mapsto c_{\yvec, R}$, then
  \(
    \| \yvec - D_R^\ast (E_R^\ast (\yvec)) \|_{\otherBanach}
    \leq L \cdot (1 + 2C) \cdot R^{-s}
  \)
  for all $\yvec \in \signalClass'$ and $R \in \N$,
  and hence $\optRateSmall<\signalClass'>{\otherBanach} \geq s$.
  Since $s \in [0, \optRateSmall<\signalClass>{\banach})$ was arbitrary, this completes the proof.
\end{proof}

The following lemma shows that if a signal class $\signalClass \subset \banach$
carries a Borel probability measure of growth order $s_0$ and satisfies $\optRateSmall{\banach} \geq s_0$,
then in fact $\optRateSmall{\banach} = s_0$.
This is elementary, but will be used quite frequently, so that we prefer to state it as a lemma.

\begin{lem}\label{lem:OneSidedBoundsSufficeForExactRate}
  Let $s_0 \in [0,\infty)$, let $\banach$ be a Banach space,
  and let $\signalClass \subset \banach$.
  Assume that there exists a Borel probability measure $\P$ of growth order $s_0$ on $\signalClass$
  and that $\optRateSmall{\banach} \geq s_0$.
  Then $\optRateSmall{\banach} = s_0$ and $\P$ is critical for $\signalClass$ with respect to $\banachOne$.
\end{lem}

\begin{proof}
  Corollary~\ref{cor:GrowthOrderAtLeastOptimalRate} shows that $s_0 \geq \optRateSmall{\banach}$.
  Since $s_0 \leq \optRateSmall{\banach}$ by assumption, the claim of the lemma follows.
\end{proof}

We finally provide the proof of Theorem~\ref{thm:LipschitzTransferResult}.

\begin{proof}[Proof of Theorem~\ref{thm:LipschitzTransferResult}]
  Since $\Phi : \signalClass_\banachOne \to \banachThree$ is Lipschitz continuous
  with $\Phi(\signalClass_\banachOne) \supset \signalClass$,
  Lemma~\ref{lem:ApproximationRateTransference} shows that
  $\optRateSmall{\banachThree} \geq \optRateSmall<\signalClass_\banachOne>{\banachOne} =: s^\ast$.
  Furthermore, since $\Psi : \signalClass_{\banachTwo} \to \signalClass$ is measurable and
  expansive and $\P$ has growth order $\optRateSmall<\signalClass_\banachTwo>{\banachTwo} = s^\ast$,
  Lemma~\ref{lem:CriticalMeasuresPushForward} shows that $\nu := \P \circ \Psi^{-1}$ is a Borel
  probability measure on $\signalClass$ of growth order $s^\ast$ as well.
  Now, Lemma~\ref{lem:OneSidedBoundsSufficeForExactRate} shows that $\optRate{\banachThree} = s^\ast$
  and that $\nu$ is critical for $\signalClass$ with respect to $\banachThree$.
\end{proof}

\section{A lower bound for the optimal compression rate \texorpdfstring{$\optRate<\sequenceSpaceSignalClass>{\ell^2(\indexSet)}$}{in sequence spaces}}
\label{sec:SequenceSpaceCompressionRateLowerBound}

Our goal in this subsection is to show that the optimal compression rate
for the class $\sequenceSpaceSignalClass$ satisfies
\(
  \optRate<\sequenceSpaceSignalClass>{\ell^2(\indexSet)}
  \geq \tfrac{\alpha}{d} - (\tfrac{1}{2} - \tfrac{1}{p})
\),
assuming that $\alpha > d \cdot (\tfrac{1}{2} - \tfrac{1}{p})_+$.
Our proof of this fact relies on an equivalence between the optimal distortion
for a set and the so-called \emph{entropy numbers} of that set.
By combining this equivalence with known estimates for the entropy numbers of certain
embeddings between sequence spaces (taken from \cite{LeopoldEntropyNumbersOfWeightedSequenceSpaces}),
we will obtain the claim.

First, let us describe the equivalence between the optimal achievable distortion and the entropy
numbers of a set.
Following \cite{CarlStephaniEntropy,EdmundsTriebelFunctionSpacesEntropyNumbers},
given a (quasi)-Banach space $\X$, a set $M \subset \X$, and $k \in \N$,
the \emph{$k$-th entropy number} $e_k(M) := e_k (M ; \X)$ of $M$ is defined as
\[
  e_k (M ; \X)
  := \inf \bigg\{
            \eps > 0
            \,\, \bigg| \,\,
            \exists \, \xvec_1, \dots, \xvec_{2^{k-1}} \in \X :
              M \subset \bigcup_{i=1}^{2^{k-1}} B_\eps (\xvec_i)
          \bigg\}
  \in [0,\infty],
\]
with the convention that $\inf \emptyset = \infty$.
Note that $e_k (M)$ is finite if and only if $M$ is bounded.
Furthermore, $e_k(M) \xrightarrow[k\to\infty]{} 0$ if and only if $M \subset \X$ is totally bounded.
Finally, if $\Y$ is a further (quasi)-Banach space, and $T : \Y \to \X$ is linear,
then the entropy numbers $e_k(T)$ are defined as $e_k(T) := e_k \big( T(\ball (0, 1; \Y)) ; \X \big)$.

For proving that
\(
  \optRate<\sequenceSpaceSignalClass>{\ell^2(\indexSet)}
  \geq \tfrac{\alpha}{d} - (\tfrac{1}{2} - \tfrac{1}{p})
\),
we will use the following folklore equivalence between entropy numbers and the optimal
achievable distortion for a given set:

\begin{lem}\label{lem:EntropyAndDistortion}
  Let $\X$ be a Banach space and $\signalClass \subset \X$.
  Then
  \[
    e_{R+1} (\signalClass ; \X)
    = \inf \big\{
             \distortion{\banach} (E_R , D_R)
             \colon
             (E_R, D_R) \in \endec{\banach}
           \big\}
    \text{ for all } R \in \N .
  \]
\end{lem}

\begin{proof}
  ``$\leq$'': Let $(E_R, D_R) \in \endec{\X}$ be arbitrary.
  Note that $\mathrm{range}(D_R) = D_R (\{ 0,1\}^R)$ is nonempty and has at most $2^R$
  elements, so that $\mathrm{range}(D_R) = \{\xvec_1, \dots, \xvec_{2^R} \}$,
  where we possibly repeat some elements.
  Define $\delta := \distortion{\X} (E_R, D_R)$.
  If $\delta = \infty$, then trivially $e_{R+1}(\signalClass; \banach) \leq \delta$;
  hence, assume that $\delta < \infty$.
  By definition of the distortion, this means $\| \uvec - D_R(E_R(\uvec)) \|_{\X} \leq \delta$
  for all $\uvec \in \signalClass$, and hence
  \[
    \uvec \in \ball \big( D_R(E_R(\uvec)), \delta; \X \big)
      \subset \bigcup_{i=1}^{2^R} \ball (\xvec_i, \delta; \X)
    \qquad \forall \, \uvec \in \signalClass ,
  \]
  since $D_R(E_R(\uvec)) \in \mathrm{range}(D_R) = \{\xvec_1, \dots, \xvec_{2^R} \}$.
  Therefore, $\signalClass \subset \bigcup_{i=1}^{2^{(R+1) - 1}} \ball (\xvec_i, \delta; \X)$,
  which shows that $e_{R+1} (\signalClass) \leq \delta = \distortion{\X} (E_R, D_R)$.

  \medskip{}

  ``$\geq$'': This is trivial if $e_{R+1}(\signalClass) = \infty$;
  hence, assume that $e_{R+1}(\signalClass) < \infty$.
  Choose a bijection $\iota : \{0,1\}^R \to \{1,\dots,2^R\}$,
  and let $\delta > e_{R+1} (\signalClass)$ be arbitrary.
  By definition of the entropy number, there are $\xvec_1, \dots, \xvec_{2^R} \in \X$
  such that $\signalClass \subset \bigcup_{i=1}^{2^R} \ball(\xvec_i, \delta; \X)$.
  Hence, for each $\uvec \in \signalClass$, there is $n_\uvec \in \{1, \dots, 2^R\}$ such that
  $\| \uvec - \xvec_{n_\uvec} \|_\X \leq \delta$.
  Now, define $(E_R, D_R) \in \endec{\X}$ by
  \[
    E_R : \signalClass \to \{0,1\}^R, \uvec \mapsto \iota^{-1} (n_\uvec)
    \qquad \text{and} \qquad
    D_R : \{0,1\}^R \to \X, c \mapsto \xvec_{\iota(c)} ,
  \]
  so that $D_R(E_R(\uvec)) = \xvec_{\iota(\iota^{-1}(n_\uvec))} = \xvec_{n_\uvec}$, and thus
  ${\| \uvec - D_R(E_R(\uvec)) \|_\X = \| \uvec - \xvec_{n_\uvec} \|_X \leq \delta}$
  for all $\uvec \in \signalClass$.
  Therefore, $\distortion{\X} (E_R, D_R) \leq \delta$.
  Since $\delta > e_{R+1} (\signalClass)$ was arbitrary, this completes the proof.
\end{proof}

In addition to this equivalence between entropy numbers and best achievable distortion,
we will use two results from \cite{LeopoldEntropyNumbersOfWeightedSequenceSpaces} about the
asymptotic behavior of the entropy numbers of certain sequence spaces.
The following definition introduces the terminology
used in \cite{LeopoldEntropyNumbersOfWeightedSequenceSpaces}.

\begin{defn}\label{def:LeopoldStuff}
  (see \cite[Equations (10), (11), and Definition 1]{LeopoldEntropyNumbersOfWeightedSequenceSpaces})

  A sequence $(\beta_j)_{j \in \N_0} \subset (0,\infty)$ is called
  \begin{itemize}
    \item an \emph{admissible sequence} if there are $d_0, d_1 \in (0,\infty)$
          such that $d_0 \, \beta_j \leq \beta_{j+1} \leq d_1 \, \beta_{j}$ for all $j \in \N_0$;

    \item \emph{almost strongly increasing} if there is $\kappa \in \N$ such that
          \(
            2 \beta_j \leq \beta_k
          \)
          for all $j,k \in \N_0$ with $k \geq j + \kappa$.
  \end{itemize}

  Given $p,q \in (0,\infty]$ and sequences
  $\boldsymbol{\beta} = (\beta_j)_{j \in \N_0} \subset (0,\infty)$
  and $\mathbf{\LeopoldM} = (\LeopoldM_j)_{j \in \N_0} \subset \N$, define
  \(
    J_{\mathbf{\LeopoldM}} := \{ (j,\ell) \in \N_0 \times \N \colon 1 \leq \ell \leq \LeopoldM_j \}
  \)
  and
  \[
    \| \xvec \|_{\ell^q(\beta_j \ell_{\LeopoldM_j}^p)}
    := \Big\|
         \Big(
         \beta_j \, \big\| (x_{j,\ell})_{\ell \in \{1,\dots,\LeopoldM_j\}} \big\|_{\ell^p}
         \Big)_{j \in \N_0}
       \Big\|_{\ell^q}
    \in [0,\infty]
    \quad \text{for} \! \quad
    \xvec = (x_{j,\ell})_{(j,\ell) \in J_{\mathbf{\LeopoldM}}} \in \CC^{J_{\mathbf{\LeopoldM}}},
  \]
  as well as
  \(
    \ell^q (\beta_j \, \ell_{\LeopoldM_j}^p)
    := \big\{
         \xvec \in \CC^{J_{\mathbf{\LeopoldM}}}
         \colon
         \| \xvec \|_{\ell^q (\beta_j \, \ell_{\LeopoldM_j}^p)} < \infty
       \big\}.
  \)
  For the case $\boldsymbol{\beta} = (1)_{j \in \N_0}$, we simply write
  $\LeopoldUnweighted$ instead of $\LeopoldSpace$.
\end{defn}

Using these notions, Leopold proved the following results:

\begin{thm}\label{thm:LeopoldEntropyResults}
  (see \cite[Theorems 3 and 4]{LeopoldEntropyNumbersOfWeightedSequenceSpaces})

  Let $p_1,p_2, q_1,q_2 \in (0,\infty]$, and let
  ${\mathbf{\LeopoldM} = (\LeopoldM_j)_{j \in \N_0} \subset \N}$
  and ${\boldsymbol{\beta} = (\beta_j)_{j \in \N_0} \subset (0,\infty)}$
  both be admissible, almost strongly increasing sequences.

  Assume that either
  \begin{enumerate}[label=(\roman*)]
    \item $p_1 \leq p_2$; or

    \item $p_2 < p_1$ and the sequence
          $\Big( \beta_j \cdot \LeopoldM_j^{p_1^{-1} - p_2^{-1}} \Big)_{j \in \N_0}$
          is almost strongly increasing.
  \end{enumerate}

  Then the embedding
  \(
    \LeopoldSpace{p_1}{q_1} \hookrightarrow \LeopoldUnweighted{p_2}{q_2}
  \)
  holds, and there are $C_1, C_2 > 0$ such that for all $L \in \N$, we have
  \[
    C_1 \cdot \beta_L^{-1} \, \LeopoldM_L^{-(p_1^{-1} - p_2^{-1})}
    \leq e_{2 \LeopoldM_L} \big(
                             \identity : \LeopoldSpace{p_1}{q_1} \to \LeopoldUnweighted{p_2}{q_2}
                           \big)
    \leq C_2 \cdot \beta_L^{-1} \, \LeopoldM_L^{-(p_1^{-1} - p_2^{-1})} .
  \]
\end{thm}

\begin{rem*}
  We note that the above results pertain to spaces of \emph{complex} sequences.
  At least concerning the \emph{upper} bound, however, this is no problem:
  To see this, note that if we denote by $\mathrm{Re} \, \xvec$ the (componentwise)
  real part of the sequence $\xvec$, then clearly
  \(
    \| \mathrm{Re} \, \xvec \|_{\LeopoldSpace} \leq \| \xvec \|_{\LeopoldSpace}
  \).
  Hence, defining the real-valued version of the space $\LeopoldSpace$ as
  \[
    \LeopoldReal := \big\{
                      \xvec \in \R^{J_{\mathbf{\LeopoldM}}}
                      \colon
                      \| \xvec \|_{\LeopoldSpace} < \infty
                    \big\},
  \]
  we see that if
  \(
    \ball \big( 0,1; \LeopoldSpace{p_1}{q_1} \big)
    \! \subset \! \bigcup_{i=1}^N
                    \ball \big( \xvec_i, \eps; \LeopoldUnweighted{p_2}{q_2} \big)
  \)
  for ${\xvec_1, \dots, \xvec_N \in \LeopoldUnweighted{p_2}{q_2}}$, then
  \(
    \ball \big( 0,1; \LeopoldReal{p_1}{q_1} \big)
    \subset \bigcup_{i=1}^N
              \ball \big( \mathrm{Re} \, \xvec_i, \eps; \LeopoldUnweightedReal{p_2}{q_2} \big)
  \),
  and hence
  \begin{equation}
    e_k \big( \identity : \LeopoldReal{p_1}{q_1} \to \LeopoldUnweightedReal{p_2}{q_2} \big)
    \leq e_k \big( \identity : \LeopoldSpace{p_1}{q_1} \to \LeopoldUnweighted{p_2}{q_2} \big)
    \quad \forall \, k \in \N .
    \label{eq:LeopoldResultForRealSpaces}
  \end{equation}
\end{rem*}

\begin{proof}[Proof of Proposition~\ref{prop:ApproximationRateLowerBound}]
  Let $n_m := |\indexSet_m|$ and $\LeopoldM_j := n_{j+1}$ for $m \in \N$ and $j \in \N_0$.
  Further, set ${\kappa := \big\lceil d^{-1} \big( 1 + \log_2(A / a) \big) \big\rceil}$,
  where we recall from Equation~\eqref{eq:weightequival} that ${a, A > 0}$
  satisfy $a \, 2^{d m} \leq n_m \leq A \, 2^{d m}$.
  Thus,
  \(
    \LeopoldM_{j+1}
    = n_{j+2} \sim 2^{d (j+2)}
    = 2^d \, 2^{d (j+1)}
    \sim n_{j+1}
    = \LeopoldM_j
    ,
  \)
  which shows that ${\mathbf{\LeopoldM} = (\LeopoldM_j)_{j \in \N_0}}$ is admissible.
  Furthermore, if $k \geq j + \kappa$, then
  \[
    \LeopoldM_{k}
    =    n_{k+1}
    \geq a \cdot 2^{d (k+1)}
    \geq a \cdot 2^{d (j+1) + 1 + \log_2 (A/a)}
    =    2 \cdot A \cdot 2^{d (j+1)}
    \geq 2 \, n_{j+1}
    =    2 \LeopoldM_j,
  \]
  which shows that $\mathbf{\LeopoldM}$ is almost strongly increasing.

  Next, define $\beta_j := 2^{\alpha (j+1)}$ for $j \in \N_0$, noting that
  $\beta_{j+1} = 2^{\alpha} \, \beta_j$, which implies that $(\beta_j)_{j \in \N_0}$ is admissible.
  Furthermore, if $k \geq j + \lceil \alpha^{-1} \rceil$, then
  \(
    \beta_k
    \geq 2 \cdot 2^{\alpha (j+1)}
    =    2 \, \beta_j,
  \)
  so that $(\beta_j)_{j \in \N_0}$ is also almost strongly increasing.
  Here, we used that $\alpha > 0$.

  Finally, for each $m \in \N$ pick a bijection $\iota_m : \FirstN{\LeopoldM_{m-1}} \to \indexSet_m$
  (which is possible since $\LeopoldM_{m-1} = n_m = |\indexSet_m|$), and define
  \[
    \Psi : \R^{\indexSet} \to \R^{J_{\mathbf{\LeopoldM}}},
           \xvec = (x_i)_{i \in \indexSet} \mapsto \big(
                                                     x_{\iota_{j+1}(\ell)}
                                                   \big)_{(j,\ell) \in J_{\mathbf{\LeopoldM}}} .
  \]
  It is easy to see that $\Psi$ is a bijection, and that
  \[
    \| \Psi(\xvec) \|_{\LeopoldSpace} \!
    = \! \Big\|\!
           \Big(
             \beta_j \,
             \big\|
               (x_{\iota_{j+1}(\ell)})_{\ell \in \FirstN{\LeopoldM_j}}
             \big\|_{\ell^p}
           \Big)_{\!j \in \N_0}
         \Big\|_{\ell^q} \!
    = \! \Big\|\!
           \Big(
             2^{\alpha m} \,
             \big\|
               \xvec_{m}
             \big\|_{\ell^p(\indexSet_m)}
           \Big)_{\! m \in \N}
         \Big\|_{\ell^q} \!
    =\! \| \xvec \|_{\mixSpace{q}}
  \]
  for arbitrary $p,q \in (0,\infty]$.
  Here, $\xvec_m = (x_i)_{i \in \indexSet_m}$ is as defined in Equation~\eqref{eq:weightequival}.
  In the same way, we see $\| \Psi(\xvec) \|_{\LeopoldUnweighted} = \| \xvec \|_{\mixSpace{q}{0}}$
  and also
  \(
    \| \Psi(\xvec) \|_{\LeopoldUnweighted{2}{2}}
    = \| \xvec \|_{\mixSpace{2}{0}<2>}
    = \| \xvec \|_{\ell^2(\indexSet)}
    .
  \)
  Using these identities, it is straightforward to see that
  $\mixSpace{q} \hookrightarrow \ell^2(\indexSet)$ holds if and only if
  $\LeopoldReal \hookrightarrow \LeopoldUnweightedReal{2}{2}$, and furthermore that
  \[
    e_k (\sequenceSpaceSignalClass; \ell^2(\indexSet))
    = e_k \big( \mixSpace{q} \hookrightarrow \ell^2(\indexSet) \big)
    = e_k \big( \LeopoldReal \hookrightarrow \LeopoldUnweightedReal{2}{2} \big)
    \qquad \forall \, k \in \N .
  \]

  \medskip{}

  There are now two cases.
  First, if $p \leq 2$, then Equation~\eqref{eq:LeopoldResultForRealSpaces} and the first part
  of Theorem~\ref{thm:LeopoldEntropyResults} with $p_1 = p, q_1 = q$ and $p_2 = q_2 = 2$
  show that ${\LeopoldReal \hookrightarrow \LeopoldUnweightedReal{2}{2}}$,
  and yield a constant $C_1 > 0$ such that
  \[
    e_{2 N_L} \big( \sequenceSpaceSignalClass; \ell^2 (\indexSet) \big)
    = e_{2 N_L} \big( \LeopoldReal \hookrightarrow \LeopoldUnweightedReal{2}{2} \big)
    \leq C_1 \cdot \beta_L^{-1} \cdot N_L^{-(p^{-1} - 2^{-1})}
  \]
  for all $L \in \N$.

  If otherwise $p > 2$, then $2^{-1} - p^{-1} > 0$, so that our assumptions
  concerning $\alpha$ imply that $\alpha > d \cdot (2^{-1} - p^{-1})_+ = d \cdot (2^{-1} - p^{-1})$,
  and hence $\gamma := \alpha + d \cdot (p^{-1} - 2^{-1}) > 0$.
  Therefore, the sequence
  $(K_j)_{j \in \N_0} := \big( \beta_j \cdot N_j^{p^{-1} - 2^{-1}}\big)_{j \in \N_0}$
  is almost strongly increasing; indeed, if
  \(
    k \geq j + \big\lceil
                 \gamma^{-1} \cdot \big( 1 + \log_2 [ (a / A)^{p^{-1} - 2^{-1}}] \big)
               \big\rceil ,
  \)
  then we see because of $a \, 2^{d (j+1)} \leq N_j \leq A \, 2^{d(j+1)}$ that
  \begin{align*}
    K_k
    & \geq A^{p^{-1} - 2^{-1}} \cdot 2^{\alpha (k+1)} \, 2^{d (k+1) (p^{-1} - 2^{-1})}
      =    2^\gamma A^{p^{-1} - 2^{-1}} \cdot 2^{\gamma k} \\
    & \geq 2^\gamma A^{p^{-1} - 2^{-1}} \cdot 2^{\gamma j} \cdot 2 \cdot (a/A)^{p^{-1} - 2^{-1}}
      =    2 \, a^{p^{-1} - 2^{-1}} \cdot 2^{\alpha (j+1)} \, 2^{d (j+1) (p^{-1} - 2^{-1})}
      \geq 2 \, K_j .
  \end{align*}
  Thus, Part~(ii) of Theorem~\ref{thm:LeopoldEntropyResults}
  and Equation~\eqref{eq:LeopoldResultForRealSpaces} show that
  ${\LeopoldReal \hookrightarrow \LeopoldUnweightedReal{2}{2}}$,
  and that there is a constant $C_2 > 0$ such that
  \[
    e_{2 N_L} \big( \sequenceSpaceSignalClass; \ell^2 (\indexSet) \big)
    = e_{2 N_L} \big( \LeopoldReal \hookrightarrow \LeopoldUnweightedReal{2}{2} \big)
    \leq C_2 \cdot \beta_L^{-1} \cdot N_L^{-(p^{-1} - 2^{-1})}
  \]
  for all $L \in \N$.

  \medskip{}

  Define $C_3 := \max \{ C_1, C_2 \}$ and note that the preceding estimates only yield
  bounds for the entropy numbers $e_{k} (\sequenceSpaceSignalClass; \ell^2(\indexSet))$
  in case of $k = 2 N_L$ for some $L \in \N$, not for general $k \in \N$.
  This, however, suffices to handle the general case.
  Indeed, let $R \in \N$ with $R \geq 2 \LeopoldM_1$ be arbitrary, and let
  $L \in \N$ be maximal with $2 \LeopoldM_L \leq R + 1$;
  this is possible since $\LeopoldM_L \to \infty$ as $L \to \infty$.
  Note
  \(
    R
    \leq R+1
    <    2 \, \LeopoldM_{L+1}
    =    2 \, n_{L+2}
    \leq 2A \, 2^{d (L+2)}
    =    2^{2d+1}A \, 2^{d L}
  \)
  by maximality.
  Since the sequence of entropy numbers
  $\big( e_k(\sequenceSpaceSignalClass;\ell^2(\indexSet)) \big)_{k \in \N}$
  is non-increasing, we thus see
  \begin{align*}
    e_{R+1} \big( \sequenceSpaceSignalClass; \ell^2(\indexSet) \big)
    & \leq e_{2 \LeopoldM_L} \big( \sequenceSpaceSignalClass; \ell^2(\indexSet) \big)
      \leq C_3 \cdot \beta_L^{-1} \cdot N_L^{-(p^{-1} - 2^{-1})} \\
    ({\scriptstyle{\text{since } N_L = n_{L+1} \sim 2^{d L}}})
    & \leq C_4 \cdot 2^{-\alpha L} \cdot 2^{(2^{-1} - p^{-1}) d L}
      =    C_4 \cdot \big( 2^{d L} \big)^{-(\frac{\alpha}{d} + p^{-1} - 2^{-1})} \\
    ({\scriptstyle{\text{since } \frac{\alpha}{d} + p^{-1} - 2^{-1} > 0}})
    & \leq C_5 \cdot R^{-(\frac{\alpha}{d} + p^{-1} - 2^{-1})},
  \end{align*}
  for all $R \geq 2 \LeopoldM_1$ and suitable constants $C_4,C_5 > 0$
  which are independent of $R$.

  Now, since $\sequenceSpaceSignalClass \subset \ell^2(\indexSet)$ is bounded
  (otherwise, all entropy numbers would be infinite), it is easy to see
  \(
    e_{R+1} (\sequenceSpaceSignalClass; \ell^2(\indexSet))
    \leq e_1(\sequenceSpaceSignalClass; \ell^2(\indexSet))
    \lesssim R^{-(\frac{\alpha}{d} + p^{-1} - 2^{-1})}
  \)
  for $R \in \N$ with $R < 2 \LeopoldM_1$.
  With this, the claim
  \(
    \optRateSmall<\sequenceSpaceSignalClass>{\ell^2(\indexSet)}
    \geq \frac{\alpha}{d} - \bigl( \frac{1}{2} - \frac{1}{p} \bigr)
  \)
  follows from the relation between entropy numbers and optimal
  distortion described in Lemma~\ref{lem:EntropyAndDistortion}.

  Finally, since $e_R \bigl(\sequenceSpaceSignalClass; \ell^2(\indexSet)\bigr) \to 0$ as $R \to \infty$,
  it follows that $\sequenceSpaceSignalClass \subset \ell^2(\indexSet)$ is totally bounded.
  Since $\sequenceSpaceSignalClass \subset \ell^2(\indexSet)$ is also easily seen to be closed
  (this essentially follows from Fatou's lemma),
  we see that $\sequenceSpaceSignalClass \subset \ell^2(\indexSet)$ is compact.
\end{proof}

\section{A review of Besov spaces}
\label{sec:BesovReview}

In this subsection, we review the relevant properties of Besov spaces on $\R^d$ and on domains,
including the characterization of these spaces in terms of wavelets;
see Section~\ref{sub:BesovWaveletCharacterization}.

Before we dive into the details, a word of caution is in order.
In the literature, there are two common definitions of Besov spaces:
A Fourier analytic definition and a definition using moduli of continuity.
Here, we only consider the former definition;
the reader interested in the latter is referred to \cite{DeVoreNonlinearApproximation}.
It should be mentioned, however, that the two definitions do \emph{not} agree in general;
see for instance \cite{HaroskeBesovSpacesWithPositiveSmoothness}.
Nevertheless, in the regime that we are interested in, the two definitions coincide,
as can be deduced from \mbox{\cite[Theorem in Section 2.5.12]{TriebelTheoryOfFunctionSpaces1}}.
Since we focus on the Fourier analytic definition only, we omit the details.

\subsection{The (Fourier-analytic) definition of Besov spaces}
\label{sub:BesovFourierDefinition}

Our presentation here follows \cite[Section~2.3]{TriebelTheoryOfFunctionSpaces1} and
\cite[Section~1.3]{TriebelTheoryOfFunctionSpaces3}.
In this section, all functions are taken to be complex-valued, unless indicated otherwise.
Let $\Schwartz (\R^d)$ denote the space of \emph{Schwartz functions}
(see, for instance, \cite[Section 8.1]{FollandRA}), and $\Schwartz' (\R^d)$ its topological
dual space, the space of \emph{tempered distributions} (see \cite[Section 9.2]{FollandRA}).
We use the \emph{Fourier transform} on $L^1(\R^d)$ with the same normalization as in
\cite{WojtaszczykIntroductionToWavelets,TriebelTheoryOfFunctionSpaces1}; that is,
\[
  \widehat{f}(\xi)
  := \Fourier f (\xi)
  := (2\pi)^{-d/2} \int_{\R^d}
                     f(x) e^{- i \langle x,\xi \rangle}
                   \, d x
  \quad \text{for} \quad
  f \in L^1(\R^d) \text{ and } \xi \in \R^d,
\]
where $\langle x,\xi \rangle = \sum_{j=1}^d x_j \xi_j$ denotes the standard inner product on $\R^d$.
With this normalization, the Fourier transform $\Fourier : L^1(\R^d) \to C_0(\R^d)$
extends to a unitary operator $\Fourier : L^2(\R^d) \to L^2(\R^d)$
and also to linear homeomorphisms $\Fourier : \Schwartz(\R^d) \to \Schwartz(\R^d)$ and
$\Fourier : \Schwartz'(\R^d) \to \Schwartz'(\R^d)$, with the latter defined by
\(
  \langle \Fourier f, \varphi \rangle_{\Schwartz',\Schwartz}
  := \langle f, \Fourier \varphi \rangle_{\Schwartz', \Schwartz}.
\)
Here, as in the remainder of the paper, the dual pairing for distributions
are taken to be \emph{bilinear}.
In any case, the inverse Fourier transform is given by (the extension of) the operator
$\Fourier^{-1} f (x) = \Fourier f (-x)$.
All of the facts listed here can be found in \cite[Chapter~7]{RudinFA}.

Fix $\varphi_0 \in \Schwartz (\R^d)$ satisfying $\varphi_0 (\xi) = 1$
for all $\xi \in \R^d$ such that $|\xi| \leq 1$, and $\varphi_0 (\xi) = 0$
for all $\xi \in \R^d$ satisfying $|\xi| \geq 3/2$.
Define $\varphi_k : \R^d \to \CC, \xi \mapsto \varphi_0(2^{-k} \xi) - \varphi_0 (2^{-k + 1} \xi)$
for $k \in \N$, noting that $\sum_{j=0}^\infty \varphi_j (\xi) \equiv 1$ on $\R^d$.

With this, the (inhomogeneous) \emph{Besov space} $B^{\BesovSmoothness}_{p,q}(\R^d)$
with smoothness $\BesovSmoothness \in \R$ and integrability exponents $p,q \in (0,\infty]$
is defined (see \cite[Section~1.3, Definition~1.2]{TriebelTheoryOfFunctionSpaces3}) as
\[
  B^{\BesovSmoothness}_{p,q}(\R^d)
  := \big\{
       f \in \Schwartz'(\R^d)
       \colon
       \| f \|_{B^{\BesovSmoothness}_{p,q}(\R^d)} < \infty
     \big\}
\]
where
\[
  \| f \|_{B^{\BesovSmoothness}_{p,q} (\R^d)}
  := \Big\|
       \big(
         2^{j \BesovSmoothness} \cdot
         \| \Fourier^{-1} (\varphi_j \cdot \widehat{f\,} \, ) \|_{L^p}
       \big)_{j \in \N_0}
     \Big\|_{\ell^q}
  \in [0,\infty]
  \quad \text{for} \quad
  f \in \Schwartz'(\R^d).
\]
This is well-defined, since $\varphi_j \cdot \widehat{f\,}$ is a tempered distribution with
compact support, so that the Paley-Wiener theorem (see \cite[Theorem~7.23]{RudinFA})
shows that $\Fourier^{-1} (\varphi_j \cdot \widehat{f\,})$ is a smooth function
of which one can take the $L^p$ norm (which might be infinite).
One can show that the definition of $B^{\BesovSmoothness}_{p,q} (\R^d)$ is independent
of the precise choice of the function $\varphi_0$,
with equivalent quasi-norms for different choices;
see \mbox{\cite[Proposition 1 in Section 2.3.2]{TriebelTheoryOfFunctionSpaces1}}.
Furthermore, the spaces $B^{\BesovSmoothness}_{p,q} (\R^d)$ are quasi-Banach spaces that satisfy
$B^{\BesovSmoothness}_{p,q}(\R^d) \hookrightarrow \Schwartz'(\R^d)$;
see \mbox{\cite[Theorem in Section 2.3.3]{TriebelTheoryOfFunctionSpaces1}}.

Now, let $\emptyset \neq \Omega \subset \R^d$ be a bounded open set,
and let $\BesovSmoothness \in \R$ and $p,q \in (0,\infty]$.
We will use the space $\calD' (\Omega)$ of distributions on $\Omega$;
for more details on these spaces, we refer to \cite[Chapter 6]{RudinFA}.
Following \cite[Definition 1.95]{TriebelTheoryOfFunctionSpaces3}, we then define
\[
  B^{\BesovSmoothness}_{p,q} (\Omega)
  := \big\{ f|_{\Omega} \colon f \in B^{\BesovSmoothness}_{p,q} (\R^d) \big\}
\]
and
\begin{equation}
  \| f \|_{B^{\BesovSmoothness}_{p,q} (\Omega)}
  := \inf \big\{
            \| g \|_{B^{\BesovSmoothness}_{p,q} (\R^d)}
            \colon
            g \in B^{\BesovSmoothness}_{p,q}(\R^d) \text{ and } g|_{\Omega} = f
          \big\}
  \quad \text{for} \quad f \in B^{\BesovSmoothness}_{p,q}(\Omega).
  \label{eq:BesovDomainNormDefinition}
\end{equation}
Here, given a tempered distribution $f \in \Schwartz'(\R^d)$, we write $f|_{\Omega}$
for the restriction of $f$ to $\Omega$, given by
$f|_{\Omega} : C_c^\infty (\Omega) \to \CC, \psi \mapsto f(\psi)$.
It is easy to see that $f|_{\Omega} \in \calD'(\Omega)$.
The spaces $B^{\BesovSmoothness}_{p,q}(\Omega)$ are quasi-Banach spaces
that satisfy $B^{\BesovSmoothness}_{p,q}(\Omega) \hookrightarrow \calD'(\Omega)$;
see \cite[Remark~1.96]{TriebelTheoryOfFunctionSpaces3}.

\subsection{The wavelet characterization of Besov spaces}
\label{sub:BesovWaveletCharacterization}

Wavelets are usually constructed using a so-called \emph{multiresolution analysis of $L^2(\R)$}.
A multiresolution analysis (see \cite[Definition~2.2]{WojtaszczykIntroductionToWavelets} or
\cite[Section~5.1]{DaubechiesTenLectures})
of $L^2(\R)$ is a sequence $(V_j)_{j \in \Z}$
of closed subspaces $V_j \subset L^2(\R)$ with the following properties:
\begin{enumerate}
  \item $V_j \subset V_{j+1}$ for all $j \in \Z$;

  \item $\bigcup_{j \in \Z} V_j$ is dense in $L^2(\R)$;

  \item $\bigcap_{j \in \Z} V_j = \{0\}$;

  \item for $f \in L^2(\R)$, we have $f \in V_j$ if and only if $f(2^{-j} \bullet) \in V_0$;

  \item there exists a function $\psi_F \in V_0$
        (called the \emph{scaling function} or the \emph{father wavelet})
        such that $\big( \psi_F (\bullet - m) \big)_{m \in \Z}$ is an orthonormal basis of $V_0$.
\end{enumerate}

To each multiresolution analysis, one can associate a \emph{(mother) wavelet} $\psi_M \in L^2(\R)$;
see \cite[Theorem~2.20]{WojtaszczykIntroductionToWavelets}.
More precisely, denote by $W_0 \subset L^2(\R)$ the orthogonal complement of $V_0$
as a subset of $V_1$, and define $W_j := \{ f (2^j \bullet) \colon f \in W_0 \}$ for $j \in \N$,
so that $W_j$ is the orthogonal complement of $V_j$ in $V_{j+1}$.
We then have $L^2(\R) = V_0 \oplus \bigoplus_{j=0}^\infty W_j$, where the sum is orthogonal.

One can show (see \cite[Lemma~2.19]{WojtaszczykIntroductionToWavelets}) that there exists
$\psi_M \in W_0$ such that the family $\big( \psi_M (\bullet - k) \big)_{k \in \Z}$
is an orthonormal basis of $W_0$.
In this case, we say that $\psi_M$ is a \emph{mother wavelet}
associated to the given multiresolution analysis.
For each such $\psi_M$, one can show
(see \mbox{\cite[Proposition 1.51]{TriebelTheoryOfFunctionSpaces3}}) that if we define
\[
  \psi_{j,m} :
  \R \to \CC,
  x \mapsto \begin{cases}
              \psi_F (x - m) ,                                & \text{if } j=0 \\[0.2cm]
              2^{\frac{j-1}{2}} \cdot \psi_M (2^{j-1} x - m), & \text{if } j \in \N
            \end{cases}
\]
for $j \in \N_0$ and $m \in \Z$, then the \emph{inhomogeneous wavelet system}
$(\psi_{j,m})_{j \in \N_0, m \in \Z}$ forms an orthonormal basis of $L^2(\R)$.
Furthermore, the family $\big( 2^{j/2} \, \psi_M (2^j \bullet - k) \big)_{j,k \in \Z}$
is an orthonormal basis of $L^2(\R)$.

For our purposes, we will need sufficiently regular wavelet systems,
as provided by the following theorem:

\begin{thm}\label{thm:WaveletExistence}
  For each $k \in \N$, there is a multiresolution analysis $(V_j)_{j \in \Z}$ of $L^2(\R)$
  with father/mother wavelets $\psi_F, \psi_M \in L^2(\R)$ such that the following hold:
  \begin{enumerate}
    \item $\psi_F, \psi_M$ are real-valued and have compact support;
          \vspace{-0.15cm}
    \item $\psi_F, \psi_M \in C^k (\R)$;
          \vspace{-0.15cm}
    \item $\myhat{\psi_F} (0) = (2\pi)^{-1/2}$;
          \vspace{-0.15cm}
    \item $\int_{\R} x^\ell \cdot \psi_M (x) \, d x = 0$ for all $\ell \in \{ 0,\dots,k \}$
          (\emph{vanishing moment condition}).
  \end{enumerate}
\end{thm}

\begin{proof}
  The existence of a multi-resolution analysis $(V_j)_{j \in \Z}$ with compactly supported
  father/mother wavelets $\psi_F, \psi_M \in C^k (\R)$ is shown in
  \cite[Theorem~4.7]{WojtaszczykIntroductionToWavelets}
  (while the original proof was given in \cite{DaubechiesCompactlySupportedWavelets}).
  It is not stated explicitly, however, that $\psi_F, \psi_M$ are real-valued;
  but this can be extracted from the proof:
  The function $\Phi := \psi_F$ is constructed as $\Phi = (2\pi)^{-1/2} \, \Fourier^{-1} \Theta$,
  with $\Theta(\xi) = \prod_{j=1}^\infty m(2^{-j} \xi)$
  (see \mbox{\cite[Theorem~4.1]{WojtaszczykIntroductionToWavelets}}),
  where ${m(\xi) = \sum_{k=0}^T a_k e^{i k \xi}}$ is obtained through
  \cite[Lemma~4.6]{WojtaszczykIntroductionToWavelets}, so that $a_0,\dots,a_T \in \R$
  and $m(0) = 1$.
  Therefore, \cite[Lemma~4.3]{WojtaszczykIntroductionToWavelets} shows that $\Phi$ is real-valued.
  Finally, $\Psi := \psi_M$ is obtained from $\Phi$ as
  \({
    \Psi (x) = 2\sum_{k=0}^T \overline{a_k} (-1)^k \Phi(2x + k + 1);
  }\)
  see \cite[Equation~(4.5)]{WojtaszczykIntroductionToWavelets}.
  Since $a_0, \dots, a_T \in \R$, this shows that $\psi_M = \Psi$ is real-valued as well.

  The above construction also implies $\myhat{\psi_F} (0) = (2\pi)^{-1/2} \, \Theta(0) = (2\pi)^{-1/2}$,
  because of $\Theta(0) = \prod_{j=1}^\infty m(0) = 1$.
  Finally, the vanishing moment condition is a consequence of
  \cite[Proposition 3.1]{WojtaszczykIntroductionToWavelets}.
\end{proof}

Wavelet systems in $\R^d$ can be constructed by taking suitable tensor products
of a one-dimensional wavelet system.
To describe this, let $\psi_F, \psi_M$ be father/mother wavelets, and
let $\waveletT_0 := \{F\}^d$ and $\waveletT_j := \waveletT := \{F,M\}^d \setminus \waveletT_0$ for $j \in \N$.
Now, for $t = (t_1,\dots,t_d) \in \waveletT$ and $m = (m_1,\dots,m_d) \in \Z^d$, define
$\Psi_m : \R^d \to \CC$ and $\Psi_{t,m} : \R^d \to \CC$ by
\begin{equation}
  \Psi_m (x) := \prod_{j=1}^d \psi_F (x_j - m_j)
  \quad \text{and} \quad
  \Psi_{t,m} (x) := \prod_{j=1}^d \psi_{t_j} (x_j - m_j).
  \label{eq:WaveletsMultiDimensionalPreparation}
\end{equation}
Finally, set $J := \{ (j,t,m) \colon j \in \N_0, t \in \waveletT_j , m \in \Z^d \}$, and
\begin{equation}
  \Psi_{j,t,m} :
  \R^d \to \CC,
  x \mapsto
  \begin{cases}
    \Psi_m (x)                               & \text{if } j = 0 , \,
                                                          t \in \waveletT_0, \,
                                                          \text{ and } m \in \Z^d \\[0.15cm]
    2^{(j-1) d / 2} \, \Psi_{t,m} (2^{j-1} x)& \text{if } j \in \N , \,
                                                          t \in \waveletT_j, \,
                                                          \text{ and } m \in \Z^d
  \end{cases}
  \label{eq:WaveletsMultiDimensionalDefinition}
\end{equation}
Then (see \cite[Proposition 1.53]{TriebelTheoryOfFunctionSpaces3}),
the system $(\Psi_{j,t,m})_{(j,t,m) \in J}$ is an orthonormal basis of $L^2(\R^d)$.

Finally, we have the following wavelet characterization of the Besov spaces
$B^{\BesovSmoothness}_{p,q}(\R^d)$.

\begin{thm}\label{thm:BesovWaveletCharacterization}
  (consequence of \cite[Theorem~1.64]{TriebelTheoryOfFunctionSpaces3})

  Let $d \in \N$, $p,q \in (0,\infty]$ and $\BesovSmoothness \in \R$.
  For a sequence $\cvec = (c_{j,t,m})_{(j,t,m) \in J} \in \CC^J$ define
  \[
    \| \cvec \|_{b^{\BesovSmoothness}_{p,q}}
    := \| \cvec \|_{b^{\BesovSmoothness}_{p,q}(\R^d)}
    := \Big\|
         \Big(
           2^{j (\BesovSmoothness + d (2^{-1} - p^{-1}))} \cdot
           \| (c_{j,t,m})_{m \in \Z^d} \|_{\ell^p}
         \Big)_{j \in \N_0, t \in \waveletT_j}
       \Big\|_{\ell^q}
    \in [0,\infty],
  \]
  and
  \(
    b^{\BesovSmoothness}_{p,q} (\R^d)
    := \big\{
         \cvec \in \CC^{J}
         \colon
         \| \cvec \|_{b^{\BesovSmoothness}_{p,q}(\R^d)} < \infty
       \big\}
  \).

  \medskip{}

  Let $k \in \N$, and let $\psi_F, \psi_M$ as provided by Theorem~\ref{thm:WaveletExistence}.
  Let the $d$-dimensional wavelet system $(\Psi_{j,t,m})_{(j,t,m) \in J}$
  be as defined in Equation~\eqref{eq:WaveletsMultiDimensionalDefinition}.
  If
  \(
    k > \max \big\{ \BesovSmoothness, \frac{2 d}{p} + \frac{d}{2} - \BesovSmoothness \big\},
  \)
  then the map
  \[
    \Gamma = \Gamma_k :
    b^{\BesovSmoothness}_{p,q} (\R^d) \to B^{\BesovSmoothness}_{p,q} (\R^d),
    (c_{j,t,m})_{(j,t,m) \in J}
    \mapsto \sum_{(j,t,m) \in J}
              c_{j,t,m} \, \Psi_{j,t,m}
  \]
  is well-defined (with unconditional convergence of the series in $\Schwartz'(\R^d)$),
  and an isomorphism of (quasi)-Banach spaces.
  The inverse map of $\Gamma$ will be denoted by
  \[
    \Theta = \Theta_k := \Gamma_k^{-1}
    : B^{\BesovSmoothness}_{p,q} (\R^d) \to b^{\BesovSmoothness}_{p,q}(\R^d) ,
    f \mapsto \big( \theta_{j,t,m} (f) \big)_{(j,t,m) \in J} .
  \]
\end{thm}

We will also use the \emph{real-valued Besov space}
\[
  B^{\BesovSmoothness}_{p,q} (\R^d ; \R)
  := B^{\BesovSmoothness}_{p,q} (\R^d) \cap \Schwartz' (\R^d; \R)
  \quad \text{equipped with the (quasi)-norm} \quad
  \| \cdot \|_{B^{\BesovSmoothness}_{p,q} (\R^d)},
\]
where we write
\(
  \Schwartz' (\R^d; \R)
  := \big\{
       \varphi \in \Schwartz' (\R^d)
       \colon
       \forall \, f \in \Schwartz(\R^d ; \R) :
         \langle \varphi, f \rangle_{\Schwartz',\Schwartz} \in \R
     \big\}
\)
and $\Schwartz(\R^d ; \R) := \{ f : \R^d \to \R \colon f \in \Schwartz(\R^d) \}$.
The spaces $B_{p,q}^{\BesovSmoothness}(\Omega;\R)$ are defined similarly.
We will also use the space
\(
  b^{\BesovSmoothness}_{p,q} (\R^d; \R)
  := b^{\BesovSmoothness}_{p,q}(\R^d)
     \cap \R^J
\).

\subsection{Wavelets and Besov spaces on bounded domains}
\label{sub:WaveletsForBesovSpacesOnDomains}

Note that Theorem~\ref{thm:BesovWaveletCharacterization} only pertains to the Besov spaces
$B^{\BesovSmoothness}_{p,q}(\R^d)$.
To describe Besov spaces on domains, we will use the sequence spaces
$b^{\BesovSmoothness}_{p,q}(\Omega_{\interior};\R)$
and $b^{\BesovSmoothness}_{p,q}(\Omega_{\ext};\R)$ that we now define.

\begin{defn}\label{def:DomainSequenceSpaces}
  Let $p,q \in (0,\infty]$ and $\BesovSmoothness \in \R$,
  and let $k \in \N$ with $k > \max \{ \tau, \frac{2d}{p} + \frac{d}{2} - \tau \}$.
  Let $\emptyset \neq \Omega \subset \R^d$ be a bounded open set.
  With the father/mother wavelets $\psi_F, \psi_M$ as in Theorem~\ref{thm:WaveletExistence}
  and $\Psi_{j,t,m}$ as in Equation~\eqref{eq:WaveletsMultiDimensionalDefinition}, define
  \begin{align*}
    & J^{\ext} := \bigcup_{j \in \N_0} \big( \{j\} \times J_j^{\ext} \big)
      \,\, \text{where} \,\,
      J_{j}^{\ext}
      := \{
           (t,m) \in \waveletT_j \times \Z^d
           \colon
           \Omega \cap \supp \Psi_{j,t,m} \neq \emptyset
         \}, \\
    \text{and} \quad
    & J^{\interior} := \bigcup_{j \in \N_0} \big( \{j\} \times J_j^{\interior} \big)
      \,\, \text{where} \,\,
      J_{j}^{\interior} := \{
                             (t,m) \in \waveletT_j \times \Z^d
                             \colon
                             \supp \Psi_{j,t,m} \subset \Omega
                           \}.
  \end{align*}
  Finally, set
  \[
    b^{\BesovSmoothness}_{p,q} (\Omega_{\ext} ; \R)
    := \big\{
         (c_{j,t,m})_{(j,t,m) \in J} \in b^{\BesovSmoothness}_{p,q}(\R^d ; \R)
         \colon
         c_{j,t,m} = 0
         \quad \forall \, (j,t,m) \in J \setminus J^{\ext}
       \big\},
  \]
  and define $b^{\BesovSmoothness}_{p,q}(\Omega_{\interior}; \R)$ similarly.
  Both of these spaces are considered as subspaces of $b^{\BesovSmoothness}_{p,q}(\R^d; \R)$;
  they are thus equipped with the (quasi)-norm $\| \cdot \|_{b^{\BesovSmoothness}_{p,q}}$.
\end{defn}

\begin{rem*}
  Strictly speaking, the spaces $b_{p,q}^{\BesovSmoothness}(\Omega_{\interior}; \R)$
  and $b_{p,q}^{\BesovSmoothness}(\Omega_{\ext}; \R)$ depend on the choice of $k \in \N$
  and on the precise choice of $\psi_F, \psi_M$.
  We will, however, suppress this dependence.
\end{rem*}


The next lemma describes the relation between these sequence spaces
and the Besov spaces $B_{p,q}^{\BesovSmoothness} (\Omega;\R)$.

\begin{lem}\label{lem:BesovSequenceSpaceConnection}
  Let $d \in \N$, $\emptyset \neq \Omega \subset \R^d$ open and bounded,
  $p,q \in (0,\infty]$ and $\BesovSmoothness \in \R$.
  Let $k \in \N$ with $k > \max \{ \BesovSmoothness, \frac{2d}{p} + \frac{d}{2} - \tau \}$.
  Let $b_{p,q}^{\BesovSmoothness} (\Omega_{\interior};\R)$
  and $b_{p,q}^{\BesovSmoothness} (\Omega_{\ext};\R)$
  be as in Definition~\ref{def:DomainSequenceSpaces},
  and $J$ as defined before Equation~\eqref{eq:WaveletsMultiDimensionalDefinition}.

  Then there are continuous linear maps
  \[
    T_{\interior} :
    b_{p,q}^{\BesovSmoothness} (\Omega_\interior;\R) \to B_{p,q}^{\BesovSmoothness} (\Omega;\R)
    \quad \text{and} \quad
    T_{\ext} :
    b_{p,q}^{\BesovSmoothness} (\Omega_\ext;\R) \to B_{p,q}^{\BesovSmoothness} (\Omega;\R)
  \]
  with the following properties:
  \begin{itemize}
    \item There is $\gamma > 0$ such that
          \(
            \| T_{\interior} \cvec \|_{L^2(\Omega)}
            = \gamma \cdot \| \cvec \|_{\ell^2}
          \)
          for all $\cvec \in \ell^2 (J) \cap b_{p,q}^{\BesovSmoothness}(\Omega_{\interior}; \R)$,
          and
          \(
            \| T_{\interior} \cvec \|_{B_{p,q}^{\BesovSmoothness}(\Omega)}
            \leq \| \cvec \|_{b_{p,q}^{\BesovSmoothness}}
          \)
          for all $\cvec \in b_{p,q}^{\BesovSmoothness}(\Omega_{\interior}; \R)$.

    \item There is $\varrho > 0$ such that
          $\| T_{\ext} \cvec \|_{L^2(\Omega)} \leq \varrho \cdot \| \cvec \|_{\ell^2}$
          for all $\cvec \in b_{p,q}^{\BesovSmoothness} (\Omega_\ext; \R)$, and we have
          \begin{equation}
            \ball\big( 0,1;B_{p,q}^{\BesovSmoothness} (\Omega;\R) \big)
            \subset T_{\ext} \Big(
                               \ball \big(
                                       0,1;b_{p,q}^{\BesovSmoothness}(\Omega_{\ext};\R)
                                     \big)
                             \Big) .
            \label{eq:ExteriorSynthesisSurjective}
          \end{equation}
  \end{itemize}
\end{lem}

\begin{proof}
  With the operator $\Gamma$ as in Theorem~\ref{thm:BesovWaveletCharacterization}, let
  $\gamma := \bigl(1 + \| \Gamma \|_{b_{p,q}^{\BesovSmoothness}(\R^d) \to B_{p,q}^{\BesovSmoothness}(\R^d)}\bigr)^{-1}$,
  and define
  \[
    T_{\interior} :
    b_{p,q}^{\BesovSmoothness} (\Omega_{\interior}; \R) \to B_{p,q}^\tau (\Omega;\R),
    \cvec \mapsto \gamma \cdot (\Gamma \, \cvec)|_{\Omega} .
  \]
  By definition of the Besov space $B_{p,q}^{\BesovSmoothness}(\Omega;\R)$ and its norm
  (see Equation~\eqref{eq:BesovDomainNormDefinition}), we then see that
  $T_{\interior}$ is a well-defined continuous linear map, with
  \[
    \| T_{\interior} \, \cvec \|_{B_{p,q}^{\BesovSmoothness} (\Omega)}
    \leq \gamma \cdot \| \Gamma \, \cvec \|_{B_{p,q}^{\BesovSmoothness}(\R^d)}
    \leq \| \cvec \|_{b_{p,q}^{\BesovSmoothness}}
    \qquad \forall \, \cvec \in b_{p,q}^{\BesovSmoothness} (\Omega_{\interior}; \R).
  \]
  Next, let
  \(
    \cvec
    = (c_{j,t,m})_{(j,t,m) \in J}
    \in \ell^2(J)
        \cap b_{p,q}^{\BesovSmoothness} (\Omega_{\interior} ; \R)
  \)
  be arbitrary.
  Since the family ${(\Psi_{j,t,m})_{(j,t,m) \in J} \subset L^2(\R^d)}$ is orthonormal,
  and since $c_{j,t,m} = 0$ for $(j,t,m) \in J \setminus J^{\interior}$,
  while $\supp \Psi_{j,t,m} \subset \Omega$ for $(j,t,m) \in J^{\interior}$, we see
  \begin{align*}
    \| T_{\interior} \, \cvec \|_{L^2(\Omega)}
    & = \gamma \cdot \Big\|
                       \sum_{(j,t,m) \in J^{\interior}}
                         c_{j,t,m} \, \Psi_{j,t,m}
                     \Big\|_{L^2(\Omega)}
      = \gamma \cdot \Big\|
                       \sum_{(j,t,m) \in J^{\interior}}
                         c_{j,t,m} \, \Psi_{j,t,m}
                     \Big\|_{L^2(\R^d)} \\
    & = \gamma \cdot \| (c_{j,t,m})_{(j,t,m) \in J^{\interior}} \|_{\ell^2}
      = \gamma \cdot \| \cvec \|_{\ell^2} .
  \end{align*}

  To construct $T_{\ext}$, let $\Theta$ be as in Theorem~\ref{thm:BesovWaveletCharacterization}, set
  $\varrho := 2 \cdot \bigl(1 + \| \Theta \|_{B_{p,q}^{\BesovSmoothness}(\R^d) \to b_{p,q}^\tau (\R^d)}\bigr)$,
  and define
  \[
    T_{\ext} :
    b_{p,q}^{\BesovSmoothness} (\Omega_{\ext}; \R) \to B_{p,q}^{\BesovSmoothness}(\Omega; \R),
    \cvec \mapsto \varrho \cdot (\Gamma \, \cvec)|_{\Omega} .
  \]
  Exactly as for $T_{\interior}$, we see that $T_{\ext}$ is a well-defined continuous linear map.
  Furthermore, using again that the family $(\Psi_{j,t,m})_{(j,t,m) \in J} \subset L^2(\R^d)$
  is an orthonormal system, we see
  \[
    \| T_{\ext} \, \cvec \|_{L^2(\Omega)}
    \leq \varrho \cdot \| \Gamma \, \cvec \|_{L^2(\R^d)}
    \leq \varrho \cdot \| \cvec \|_{\ell^2}
    \quad \forall \, \cvec \in b_{p,q}^{\BesovSmoothness} (\Omega_\ext; \R).
  \]
  It remains to prove the inclusion~\eqref{eq:ExteriorSynthesisSurjective}.
  To this end, let $f \in \ball \big( 0,1; B_{p,q}^{\BesovSmoothness} (\Omega;\R) \big)$
  be arbitrary.
  By definition, this implies $f = g|_{\Omega}$ for some $g \in B_{p,q}^{\BesovSmoothness} (\R^d)$
  with $\| g \|_{B_{p,q}^{\BesovSmoothness}} \leq 2$.
  Let ${\evec := \Theta g \in b_{p,q}^{\BesovSmoothness} (\R^d)}$, and
  \(
    \cvec
    = (c_{j,t,m})_{(j,t,m) \in J}
  \)
  where
  \(
    c_{j,t,m}
    := \varrho^{-1}
       \cdot \indicator_{J^{\ext}} ( (j,t,m) )
       \cdot \operatorname{Re} (e_{j,t,m}).
  \)
  Clearly,
  \(
    \| \cvec \|_{b_{p,q}^{\BesovSmoothness}}
    \leq \varrho^{-1} \| \evec \|_{b_{p,q}^{\BesovSmoothness}}
    \leq 2 \| \Theta \| / \varrho
    \leq 1,
  \)
  which means $\cvec \in \ball \big( 0,1; b_{p,q}^{\BesovSmoothness} (\Omega_{\ext}; \R) \big)$.

  Finally, for an arbitrary \emph{real-valued} test function $\varphi \in C_c^\infty (\Omega)$,
  we have $\langle \Psi_{j,t,m}, \varphi \rangle \in \R$ for all $(j,t,m) \in J$
  and $\langle \Psi_{j,t,m} , \varphi \rangle = 0$ if $(j,t,m) \notin J^{\ext}$.
  Therefore,
  \begin{align*}
    \langle T_{\ext} \, \cvec, \varphi \rangle
    & = \operatorname{Re}
        \sum_{(j,t,m) \in J^{\ext}}
        \big(
          e_{j,t,m} \, \langle \Psi_{j,t,m} , \varphi \rangle
        \big)
      = \operatorname{Re} \,
        \Big\langle
          \sum_{(j,t,m) \in J}
            e_{j,t,m} \, \Psi_{j,t,m} , \quad
          \varphi
        \Big\rangle \\
    & = \operatorname{Re} \langle \Gamma \, \evec, \varphi \rangle
      = \operatorname{Re} \langle \Gamma (\Theta g), \varphi \rangle
      = \operatorname{Re} \langle g, \varphi \rangle
      = \operatorname{Re} \langle f, \varphi \rangle
      = \langle f, \varphi \rangle,
  \end{align*}
  since $\Theta = \Gamma^{-1}$ and $f = g|_{\Omega}$ and $\varphi \in C_c^\infty (\Omega)$,
  and since $f$ is a real-valued distribution.
  Therefore, $f = T_{\ext} \, \cvec$, proving~\eqref{eq:ExteriorSynthesisSurjective}.
\end{proof}

Finally, we show that the sequence spaces $b_{p,q}^{\BesovSmoothness}(\Omega_{\interior} ; \R)$
and $b_{p,q}^{\BesovSmoothness}(\Omega_{\ext} ; \R)$ are quite similar to the sequence spaces
$\mixSpace{q}$ introduced in Definition~\ref{def:SequenceSpaces}.
In fact, the following (seemingly) weak property will be enough for our purposes.

\begin{lem}\label{lem:WaveletSequenceSpacesAreNice}
  Let $\emptyset \neq \Omega \subset \R^d$ be open and bounded.
  Let $p,q \in (0,\infty]$ and $\BesovSmoothness \in \R$, and define
  $\alpha := \BesovSmoothness + d \cdot (2^{-1} - p^{-1})$.
  Let $k \in \N$ with
  $k > \max \{\BesovSmoothness, \frac{2d}{p} + \frac{d}{2} - \BesovSmoothness \}$,
  and let $b_{p,q}^{\BesovSmoothness}(\Omega_{\interior} ; \R)$ and
  $b_{p,q}^{\BesovSmoothness}(\Omega_{\ext} ; \R)$ be
  as in Definition~\ref{def:DomainSequenceSpaces}.

  Assume that $\alpha > d \cdot (2^{-1} - p^{-1})_+$.
  Then the embeddings $b_{p,q}^\tau (\Omega_{\interior};\R) \hookrightarrow \ell^2(J^{\interior})$
  and $b_{p,q}^\tau (\Omega_{\ext};\R) \hookrightarrow \ell^2(J^{\ext})$ hold.
  Furthermore,
  \begin{enumerate}[label=(\roman*)]
    \item There is a $d$-regular partition $\partition^{\interior}$ of $J^{\interior}$
          and some $\gamma > 0$ such that if we define
          \[
            \iota_{\interior} :
            \mixSpace[\partition^{\interior}]{q} \to b_{p,q}^{\BesovSmoothness}(\Omega_{\interior}; \R),
            \cvec \mapsto \gamma \cdot \cvec^{\natural}
          \]
          where $\cvec^{\natural} \in \R^J$ is obtained by extending $\cvec \in \R^{J^{\interior}}$
          by zero, then $\| \iota_{\interior} \| \leq 1$ and
          $\| \iota_{\interior} \, \cvec \|_{\ell^2} = \gamma \, \| \cvec \|_{\ell^2}$
          for all $\cvec \in \mixSpace[\partition^{\interior}]{q}$.

    \item There is a $d$-regular partition $\partition^{\ext}$ of $J^{\ext}$ and some $\varrho > 0$
          such that if we define
          \[
            \iota_{\ext} :
            \mixSpace[\partition^{\ext}]{q} \to b_{p,q}^{\BesovSmoothness}(\Omega_{\ext}; \R),
            \cvec \mapsto \varrho \cdot \cvec^{\natural}
          \]
          where $\cvec^{\natural} \in \R^J$ is obtained by extending $\cvec \in \R^{J^{\ext}}$
          by zero, then $\| \iota_{\ext} \, \cvec \|_{\ell^2} = \varrho \, \| \cvec \|_{\ell^2}$
          for all $\cvec \in \mixSpace[\partition^{\ext}]{q}$, and
          \[
            \ball \big( 0,1; b_{p,q}^{\BesovSmoothness} (\Omega_{\ext}; \R) \big)
            \subset \iota_{\ext} \big( \ball(0, 1; \mixSpace[\partition^{\ext}]{q}) \big) .
          \]
  \end{enumerate}
\end{lem}

\begin{proof}
  The proof is divided into three steps.
  \medskip{}

  \textbf{Step 1} \emph{(Estimating $|J_{j}^{\interior}|$ and $|J_j^{\ext}|$):}
  We show that there are $j_0 \in \N$ and $a, A > 0$ satisfying
  \[
    |J_j^{\interior}| \leq |J_{j}^{\ext}| \leq A \cdot 2^{d j} \quad \forall \, j \in \N
    \qquad \text{and} \qquad
    |J_j^{\ext}| \geq |J_{j}^{\interior}| \geq a \cdot 2^{d j}
    \quad \forall \, j \in \N_{\geq j_0}.
  \]

  First of all, we clearly have $J_j^{\interior} \subset J_j^{\ext}$
  and thus $|J_j^{\interior}| \leq |J_j^{\ext}|$.
  Next, since $\Omega \subset \R^d$ is bounded and $\psi_F, \psi_M$ have compact support,
  there is $R \in \N$ such that $\Omega \subset [-R, R]^d$ and
  $\supp \psi_F \cup \supp \psi_M \subset [-R, R]$.
  Define $A := (8 R)^d$.
  In view of Equations~\eqref{eq:WaveletsMultiDimensionalPreparation}
  and \eqref{eq:WaveletsMultiDimensionalDefinition}, this implies
  $\supp \Psi_m \cup \supp \Psi_{t,m} \subset m + [-R, R]^d$, and hence
  $\supp \Psi_{0,t,m} \subset m + [-R, R]^d$ for $t \in \waveletT_0$ and $m \in \Z^d$,
  and finally $\supp \Psi_{j,t,m} \subset 2^{1-j} (m + [-R, R]^d)$ for $j \in \N$,
  $t \in \waveletT_j$, and $m \in \Z^d$.

  Now, it is not hard to see that if
  \(
    \emptyset
    \neq \Omega \cap \supp \Psi_{0,t,m}
    \subset [-R,R]^d \cap (m + [-R,R]^d) ,
  \)
  then $m \in [-2R,2R]^d \cap \Z^d = \{-2R, \dots, 2R\}^d$, and thus
  $|J_0^{\ext}| \leq (1 + 4R)^d \leq A \cdot 2^{d \cdot 0}$.

  Furthermore, if $j \in \N$ and
  \({
    \emptyset
    \neq \Omega \cap \supp \Psi_{j,t,m}
    \subset [-R,R]^d \cap 2^{1-j} (m + [-R,R]^d)
  }\),
  then $2^{1-j} (m + x) = y$ for certain $x,y \in [-R,R]^d$, and hence
  \[
    m = 2^{j-1}y - x
    \in \Z^d \cap [-(R + 2^{j-1}R), R + 2^{j-1} R]^d
    \subset \{ - 2^j R, \dots, 2^j R \}^d.
  \]
  Because of $|\waveletT_j| \leq 2^d$, this implies
  \(
    |J_j^{\ext}|
    \leq |\waveletT_j| \cdot (1 + 2^{j+1} R)^d
    \leq (8R)^d 2^{j d}
    \leq A \, 2^{d j} .
  \)

  Regarding the lower bound, recall that $\Omega \neq \emptyset$ is open,
  so that there are $x_0 \in \R^d$ and $n \in \N$ satisfying
  ${x_0 + [-r, r]^d \subset \Omega}$, where $r := 2^{-n}$.
  Choose $j_0 \in \N_{\geq n + 3}$ such that $2^{j_0 - 1} r \geq 2 R$,
  and note ${2^{j_0 - 3} r = 2^{j_0 - 3 - n} \in \N}$.
  Let $j \geq j_0$.
  Choose $m_0 := \lfloor 2^{j-1} x_0 \rfloor \in \Z^d$,
  with the ``floor'' operation applied componentwise.
  We have $\| 2^{j-1} x_0 - m_0 \|_\infty \leq 1$, and hence
  \[
    \| 2^{j-1} x_0 - (m + m_0) \|_\infty
    \leq 1 + 2^{j-3} r
    \leq 2^{j-2} r
    \quad \text{for} \quad
    m \in \{ - 2^{j-3} r, \dots, 2^{j-3} r \}^d.
  \]
  Here, one should observe $2^{j-3} r = 2^{j - j_0} 2^{j_0 - 3 - n} \in \N$.
  Because of $R \leq 2^{j_0 - 2} r \leq 2^{j-2} r$, the above estimate implies that
  \begin{align*}
    2^{1-j} \cdot \big( m+m_0 + [-R,R]^d \big)
    & \subset 2^{1-j} \cdot \big( 2^{j-1}x_0 + [-(R+2^{j-2}r), (R + 2^{j-2}r)]^d \big) \\
    & \subset 2^{1-j} \cdot \big( 2^{j-1}x_0 + [-2^{j-1} r, 2^{j-1} r]^d \big)
      =       x_0 + [-r,r]^d
      \subset \Omega
  \end{align*}
  for all $m \in \{ - 2^{j-3} r, \dots, 2^{j-3} r \}^d$.
  Because of $\supp \Psi_{j,t,m+m_0} \subset 2^{1-j}(m + m_0 + [-R, R]^d)$, this implies
  $|J_j^{\interior}| \geq (2^{j-2} r)^d = (r/4)^d \cdot 2^{dj}$,
  so that we can choose $a = (r/4)^d$.

  \medskip

  \textbf{Step 2} \emph{(Constructing the partitions $\partition^{\interior}, \partition^{\ext}$
  and showing $\| c^{\natural} \|_{b_{p,q}^{\BesovSmoothness}} \asymp \| c \|_{\mixSpace[\partition]{q}}$):}
  Define $\indexSet_1^{\interior} := \bigcup_{j=0}^{j_0} (\{ j \} \times J_j^{\interior} )$
  and $\indexSet_1^{\ext} := \bigcup_{j=0}^{j_0} (\{ j \} \times J_j^{\ext} )$,
  as well as
  \[
    \indexSet_m^{\interior} := \{ j_0 + m - 1 \} \times J_{j_0 + m - 1}^{\interior}
    \qquad \text{and} \qquad
    \indexSet_m^{\ext} := \{ j_0 + m - 1 \} \times J_{j_0 + m - 1}^{\ext}
  \]
  for $m \in \N_{\geq 2}$.
  As shown in Step 1, we have for $m \in \N_{\geq 2}$ that
  \[
    a \cdot 2^{d m}
    \leq a \cdot 2^{d (j_0 + m - 1)}
    \leq |\indexSet_m^{\interior}|
    \leq |\indexSet_m^{\ext}|
    \leq A \cdot 2^{d (j_0 + m - 1)}
    =:   A' \cdot 2^{d m}
  \]
  and also
  \(
    |\indexSet_1^{\ext}|
    \geq |\indexSet_1^{\interior}|
    \geq |J_{j_0}^{\interior}|
    \geq a \cdot 2^{d j_0}
    \geq a \cdot 2^d .
  \)
  Thus,
  \({
    a \cdot 2^{d m}
    \leq |\indexSet_m^{\interior}|
    \leq |\indexSet_m^{\ext}|
    \leq A'' \cdot 2^{d m}
  }\)
  for all $m \in \N$, where $A'' := \max \{ A', |\indexSet_1^{\ext}| \}$.
  Furthermore, we have ${J^{\interior} = \biguplus_{m \in \N} \indexSet_m^{\interior}}$
  and ${J^{\ext} = \biguplus_{m \in \N} \indexSet_m^{\ext}}$, so that
  $\partition^{\interior} := \big( \indexSet_m^{\interior} \big)_{m \in \N}$
  and $\partition^{\ext} := \big( \indexSet_m^{\ext} \big)_{m \in \N}$
  are $d$-regular partitions of $J^{\interior}$ and $J^{\ext}$, respectively.

  Now, for $J_0 \subset J$ and $\cvec \in \R^{J_0}$, let $\cvec^{\natural} \in \R^{J}$
  be the sequence $\cvec$, extended by zero.
  We claim that there are $C_1, C_2 > 0$ such that
  \begin{equation}
    C_1 \cdot \| \cvec \|_{\mixSpace[\partition^{\interior}]{q}}
    \leq \| \cvec^{\natural} \|_{b_{p,q}^{\BesovSmoothness}}
    \leq C_2 \cdot \| \cvec \|_{\mixSpace[\partition^{\interior}]{q}}
    \qquad \forall \, \cvec \in \R^{J^{\interior}} ,
    \label{eq:BetaBijectionPreservesBesovNorm}
  \end{equation}
  and similarly for $\partition^{\ext}$ and $J^{\ext}$ instead of $\partition^{\interior}$ and $J^{\interior}$.
  For brevity, we only prove the claim for $\partition^{\interior}$.

  To prove \eqref{eq:BetaBijectionPreservesBesovNorm}, let $\cvec \in \R^{J^{\interior}}$.
  For $m \in \N$ and $j \in \N_0$, define
  \(
    \zeta_m := 2^{\alpha m} \, \| (c_\kappa)_{\kappa \in \indexSet_m^{\interior}} \|_{\ell^p}
  \)
  and
  \(
    \omega_j
    := 2^{\alpha j}
       \big\|
         \big(
           \|
             (c_{j,t,k}^{\natural})_{k \in \Z^d}
           \|_{\ell^p}
         \big)_{t \in \waveletT_j}
       \big\|_{\ell^q},
  \)
  noting that
  $\| \cvec \|_{\mixSpace[\partition^{\interior}]{q}} = \| (\zeta_m)_{m \in \N} \|_{\ell^q}$ as well as
  ${\| \cvec^{\natural} \|_{b_{p,q}^{\BesovSmoothness}} = \| (\omega_j)_{j \in \N_0} \|_{\ell^q}}$.
  Since $|\waveletT_j| \leq 2^d$ for all $j \in \N_0$, we have
  $\| \cdot \|_{\ell^p (\waveletT_j)} \asymp \| \cdot \|_{\ell^q (\waveletT_j)}$ for all
  $j \in \N_0$, with implied constant only depending on $d,p,q$.

  Now, define $J_{j,t}^{\interior} := \{ k \in \Z^d \colon (t, k) \in J_j^{\interior} \}$
  for $j \in \N_0$ and $t \in \waveletT_j$, and note for $m \geq 2$ that
  \(
    \indexSet_m^{\interior}
    = \biguplus_{t \in \waveletT_{j_0 + m - 1}}
        \big( \{ j_0 + m - 1 \} \times \{ t \} \times J_{j_0 + m - 1, t}^{\interior} \big) ,
  \)
  which implies
  \begin{align*}
    \zeta_m
    & = 2^{\alpha m} \,
        \Big\|
          \Big(
            \big\|
              \big( c_{j_0 + m - 1, t, k} \big)_{k \in J_{j_0 + m - 1, t}^{\interior}}
            \big\|_{\ell^p}
          \Big)_{t \in \waveletT_{j_0 + m - 1}}
        \Big\|_{\ell^p} \\
    & \asymp 2^{\alpha (j_0 + m - 1)}
             \Big\|
               \Big(
                 \big\|
                   \big( c_{j_0 + m - 1, t, k}^{\natural} \big)_{k \in \Z^d}
                 \big\|_{\ell^p}
               \Big)_{t \in \waveletT_{j_0 + m - 1}}
             \Big\|_{\ell^q}
      =      \omega_{j_0 + m - 1} ,
  \end{align*}
  with implied constants only depending on $d,p,q,j_0,\alpha$.
  With similar arguments, we see that
  \begin{align*}
    \zeta_1
    & = 2^{\alpha}
        \Big\|
          \Big(
            \| (c_{j,t,k})_{k \in J_{j,t}^{\interior}} \|_{\ell^p}
          \Big)_{j \in \{ 0,\dots,j_0 \}, t \in \waveletT_j}
        \Big\|_{\ell^p} \\
    & \asymp \Big\|
               \Big(
                 \big\|
                   \big(
                     2^{\alpha j} \,
                     \| (c^{\natural}_{j,t,k})_{k \in \Z^d} \|_{\ell^p}
                   \big)_{t \in \waveletT_{j}}
                 \big\|_{\ell^q}
               \Big)_{j \in \{ 0, \dots, j_0 \}}
             \Big\|_{\ell^q}
      = \big\| (\omega_j)_{j \in \{ 0, \dots, j_0 \}} \big\|_{\ell^q}.
  \end{align*}
  Overall, we obtain that
  \begin{align*}
    \| \cvec \|_{\mixSpace[\partition^{\interior}]{q}}
    & = \big\| \big( \zeta_m \big)_{m \in \N} \big\|_{\ell^q}
      \asymp \zeta_1 + \big\| \big( \zeta_m \big)_{m \in \N_{\geq 2}} \big\|_{\ell^q} \\
    & \asymp \big\| \big( \omega_j \big)_{j \in \{ 0, \dots, j_0 \}} \big\|_{\ell^q}
             + \big\|  \big( \omega_{m + j_0 - 1} \big)_{m \in \N_{\geq 2}} \big\|_{\ell^q}
      \asymp \big\| \big( \omega_j \big)_{j \in \N_0} \big\|_{\ell^q}
      =      \| \cvec^{\natural} \|_{b_{p,q}^{\BesovSmoothness}} ,
  \end{align*}
  which proves Equation~\eqref{eq:BetaBijectionPreservesBesovNorm}.

  \medskip{}

  \textbf{Step 3} \emph{(Completing the proof):}
  Step~2 guarantees the existence of $\gamma > 0$ satisfying
  \({
    \| \iota_{\interior} \, \cvec \|_{b_{p,q}^{\BesovSmoothness}}
    = \gamma \cdot \| \cvec^{\natural} \|_{b_{p,q}^{\BesovSmoothness}}
    \leq \| \cvec \|_{\mixSpace[\partition^{\interior}]{q}}
  }\)
  for all $\cvec \in \mixSpace[\partition^{\interior}]{q}$.
  Furthermore, we clearly have
  $\| \iota_{\interior} \, \cvec \|_{\ell^2} = \gamma \, \| \cvec \|_{\ell^2}$.

  Similarly, Step~2 shows that there is $\varrho > 0$ satisfying
  \(
    \| \cvec \|_{\mixSpace[\partition^{\ext}]{q}}
    \leq \varrho \, \| \cvec^{\natural} \|_{b_{p,q}^{\BesovSmoothness}}
  \)
  for all $\cvec \in \R^{J^{\ext}}$.
  Now, given $\bvec \in \ball\bigl(0, 1; b_{p,q}^{\BesovSmoothness}(\Omega_{\ext}; \R)\bigr)$, note that
  $\bvec = (\bvec|_{J^{\ext}})^{\natural}$ and furthermore
  \(
    \| \bvec|_{J^{\ext}} \|_{\mixSpace[\partition^{\ext}]{q}}
    \leq \varrho \, \| (\bvec|_{J^{\ext}})^{\natural} \|_{b_{p,q}^{\BesovSmoothness}}
    \leq \varrho ,
  \)
  so that $\cvec := \varrho^{-1} \cdot \bvec|_{J^{\ext}} \in \ball\bigl(0, 1; \mixSpace[\partition^{\ext}]{q}\bigr)$
  satisfies $\bvec = \iota_{\ext} \cvec$.
  It is clear that $\| \iota_{\ext} \cvec \|_{\ell^2} = \varrho \| \cvec \|_{\ell^2}$
  for all $\cvec \in \mixSpace[\partition^{\ext}]{q}$.
\end{proof}

\section{The phase transition for Sobolev spaces with
\texorpdfstring{$p \in \{1,\infty\}$}{p ∈ \{1, ∞\}}}
\label{sec:SobolevExceptionalCaseProof}

In this subsection we provide the missing proof of Theorem~\ref{thm:SobolevPhaseTransition}
for the cases $p = 1$ and $p = \infty$.
We begin with the case $p = 1$.

\subsection{The case \texorpdfstring{$p = 1$}{p = 1}}%
\label{sub:SobolevPhasTransitionPEqual1}

The proof is crucially based on the following embedding.

\begin{lem}\label{lem:IntegrableSobolevEmbedsIntoBesov}
  For arbitrary $k,d \in \N$ and $1 \leq p < \infty$,
  we have $W^{k,p}(\R^d) \hookrightarrow B^{k}_{p,\infty}(\R^d)$.
\end{lem}

\begin{proof}
  This follows from \cite[Section~7.33]{AdamsSobolevSpaces}.
  Here, the definition of Besov spaces used in \cite{AdamsSobolevSpaces} coincides with our
  definition, as can be seen by combining \cite[Theorem in Section~2.5.12]{TriebelTheoryOfFunctionSpaces1}
  with \cite[Proposition~17.21 and Theorem~17.24]{LeoniFirstCourseInSobolevSpaces}.
\end{proof}

Using this embedding, we can now prove Theorem~\ref{thm:SobolevPhaseTransition}
for the case $p = 1$

\begin{proof}[Proof of Theorem~\ref{thm:SobolevPhaseTransition} for $p = 1$]
  Let $k \in \N$ with $k > d \cdot (1^{-1} - 2^{-1})_{+} = \frac{d}{2}$, and define
  $\signalClass := \ball( 0, 1; W^{k,1}(\Omega) )$.
  Our goal is to apply Theorem~\ref{thm:LipschitzTransferResult} for
  $\banachOne := \banachTwo := \banachThree := L^2(\Omega)$,
  $\signalClass_{\banachOne} := \ball (0,1;B^k_{1,\infty}(\Omega))$, and
  $\signalClass_{\banachTwo} := \ball (0,1;W^{k,2}(\Omega))$,
  with suitable choices of $\Phi,\Psi,\P$.

  To this end, first note that since $\Omega \subset \R^d$ is bounded, there is $\kappa > 0$
  satisfying $\kappa \, \| f \|_{W^{k,1}(\Omega)} \leq \| f \|_{W^{k,2}(\Omega)}$
  for all $f \in W^{k,2}(\Omega)$.
  Next, Theorems~\ref{thm:mainbesovresult} and \ref{thm:SobolevPhaseTransition}
  (the latter for $p = 2 \in (1,\infty)$) show that
  $\signalClass_{\banachOne}, \signalClass_{\banachTwo} \subset L^2(\Omega)$ are bounded with
  \[
    \optRateSmall<\signalClass_{\banachOne}>{L^2(\Omega)}
    = \optRateSmall<\signalClass_{\banachTwo}>{L^2(\Omega)}
    = \frac{k}{d}
  \]
  and that there exists a Borel measure $\P_0$ on $\signalClass_{\banachTwo}$ that is critical
  for $\signalClass_{\banachTwo}$ with respect to $L^2(\Omega)$.

  Next, we claim that there is $C > 0$ satisfying
  $\| f \|_{B^k_{1,\infty}(\Omega)} \leq C \, \| f \|_{W^{k,1}(\Omega)}$
  for all ${f \in W^{k,1}(\Omega)}$.
  Indeed, \cite[Theorem~5 in Chapter~VI]{SteinSingularIntegrals} shows that there is a bounded linear
  \emph{extension operator} $\extension : W^{k,1}(\Omega) \to W^{k,1}(\R^d)$ satisfying
  $(\extension f)|_{\Omega} = f$ for all $f \in W^{k,1}(\Omega)$.
  Then, Lemma~\ref{lem:IntegrableSobolevEmbedsIntoBesov} yields $C_1 > 0$ satisfying
  \[
    \| f \|_{B^k_{1,\infty}(\Omega)}
    = \| (\extension f)|_{\Omega} \|_{B^k_{1,\infty}(\Omega)}
    \leq \| \extension f \|_{B^k_{1,\infty}(\R^d)}
    \leq C_1 \cdot \| \extension f \|_{W^{k,1}(\R^d)}
    \leq C_1 \| \extension \| \, \| f \|_{W^{k,1}(\Omega)},
  \]
  so that we can choose $C = C_1 \cdot \| \extension \|$.
  In particular, this implies that $\signalClass \subset C \cdot \signalClass_{\banachOne} \subset L^2(\Omega)$
  is bounded, so that Lemma~\ref{lem:SobolevSpaceMeasurableOnDomain} shows
  that $\signalClass = \signalClass \cap L^2(\Omega) \subset L^2(\Omega)$ is measurable.

  Overall, we see that if we choose
  \[
    \Phi : \signalClass_{\banachOne} \to L^2(\Omega), f \mapsto C \cdot f
    \quad \text{and} \quad
    \Psi : \signalClass_{\banachTwo} \to \signalClass, f \mapsto \kappa \cdot f,
  \]
  then $\Phi,\Psi$ are well-defined and satisfy all assumptions of Theorem~\ref{thm:LipschitzTransferResult}.
  This theorem then shows that $\optRateSmall{L^2(\Omega)} = \frac{k}{d}$ and that
  $\P := \P_0 \circ \Psi^{-1}$ is a Borel probability measure on $\signalClass$ that is critical
  for $\signalClass$ with respect to $L^2(\Omega)$.

  Finally, Part~\ref{enu:BesovCriticalCodecExists} of Theorem~\ref{thm:mainbesovresult} yields a codec
  \(
    \code^\ast
    = \big( (E_R^\ast, D_R^\ast) \big)_{R \in \N}
    \in \codecs<\signalClass_{\banachOne}>{L^2(\Omega)}
  \)
  satisfying $\distortion<\signalClass_{\banachOne}>{L^2(\Omega)} (E_R^\ast, D_R^\ast) \lesssim R^{-k/d}$.
  Since $\Phi$ is Lipschitz with $\Phi(\signalClass_{\banachOne}) \supset \signalClass$, the remark after
  Lemma~\ref{lem:ApproximationRateTransference} shows that there is a codec
  $\code = \big( (E_R, D_R) \big)_{R \in \N} \in \codecs{L^2(\Omega)}$ satisfying
  $\distortion{L^2(\Omega)} (E_R, D_R) \lesssim R^{-k/d}$ as well,
  as claimed in Part~\ref{enu:SobolevCriticalCodecExists} of Theorem~\ref{thm:SobolevPhaseTransition}.
\end{proof}

\subsection{The case \texorpdfstring{$p = \infty$}{p = ∞}}%
\label{sub:SobolevPhasTransitionPEqualInfty}

Let $k \in \N$ with $k > d \cdot (\frac{1}{\infty} - \frac{1}{2})_+ = 0$
and define $\signalClass := \ball(0, 1; W^{k,\infty}(\Omega))$.
Note that trivially $\signalClass \subset L^\infty(\Omega) \subset L^2(\Omega)$ is bounded,
so that Lemma~\ref{lem:SobolevSpaceMeasurableOnDomain} implies that $\signalClass \subset L^2(\Omega)$
is Borel measurable.
Our goal is to apply Theorem~\ref{thm:LipschitzTransferResult} for
$\banachOne := \banachTwo := \banachThree := L^2(\Omega)$,
$\signalClass_{\banachOne} := \ball(0,1; W^{k,2}(\Omega))$, and
$\signalClass_{\banachTwo} := \ball(0,1;B^k_{\infty,1}(\Omega))$,
for suitable choices of $\Phi,\Psi$ and $\P$.

To this end, first note that since $\Omega \subset \R^d$ is bounded, there is $C > 0$ satisfying
$\| f \|_{W^{k,2}(\Omega)} \leq C \, \| f \|_{W^{k,\infty}(\Omega)}$ for all $f \in W^{k,\infty}(\Omega)$.

Next, it is well-known (see for instance \cite[Example~7.2]{DecompositionIntoSobolev}) that there is
$\kappa > 0$ such that $\kappa \, \| f \|_{W^{k,\infty}(\R^d)} \leq \| f \|_{B^k_{\infty,1}(\R^d)}$
for all $f \in B^k_{\infty,1}(\R^d)$.
Now, for $f \in B^k_{\infty,1}(\Omega)$ and $\eps > 0$,
by definition of the norm on $B^k_{\infty,1}(\Omega)$ there is some $g \in B^k_{\infty,1}(\R^d)$
with ${\| g \|_{B^k_{\infty,1}(\R^d)} \leq (1 + \eps) \| f \|_{B^k_{\infty,1}(\Omega)}}$
and $f = g|_{\Omega}$.
Since $g \in B^k_{\infty,1}(\R^d) \subset W^{k,\infty}(\R^d)$,
we see $f \in W^{k,\infty}(\Omega)$ and
\(
  \kappa \| f \|_{W^{k,\infty}(\Omega)}
  \leq \kappa \| g \|_{W^{k,\infty}(\R^d)}
  \leq \| g \|_{B^k_{\infty,1}(\R^d)}
  \leq (1 + \eps) \, \| f \|_{B^k_{\infty,1}(\Omega)} .
\)
We have thus shown
\[
  \kappa \| f \|_{W^{k,\infty}(\Omega)} \leq \| f \|_{B^k_{\infty,1}(\Omega)}
  \qquad \forall \, f \in B^k_{\infty,1}(\Omega).
\]

Finally, Theorems~\ref{thm:mainbesovresult} and \ref{thm:SobolevPhaseTransition}
(the latter applied with $p = 2 \in (1,\infty)$) show that
\(
  \optRateSmall<\signalClass_{\banachOne}>{L^2(\Omega)}
  = \optRateSmall<\signalClass_{\banachTwo}>{L^2(\Omega)}
  = \frac{k}{d}
\)
and that there exists a Borel probability measure $\P_0$ on $\signalClass_{\banachTwo}$
that is critical for $\signalClass_{\banachTwo}$ with respect to $L^2(\Omega)$.

Combining these observations, it is not hard to see that all assumptions
of Theorem~\ref{thm:LipschitzTransferResult} are satisfied for
\[
  \Phi : \signalClass_{\banachOne} \to L^2(\Omega), f \mapsto C \cdot f
  \quad \text{and} \quad
  \Psi : \signalClass_{\banachTwo} \to \signalClass, f \mapsto \kappa \cdot f.
\]
This theorem thus shows that $\optRate{L^2(\Omega)} = \frac{k}{d}$ and that
$\P := \P_0 \circ \Psi^{-1}$ is a Borel probability measure on $\signalClass$
that is critical for $\signalClass$ with respect to $L^2(\Omega)$.

Finally, Theorem~\ref{thm:SobolevPhaseTransition} shows that there exists
$\code^\ast = \big( (E_R^\ast, D_R^\ast) \big)_{R \in \N} \in \codecs<\signalClass_{\banachOne}>{L^2(\Omega)}$
satisfying $\distortion<\signalClass_{\banachOne}>{L^2(\Omega)} (E_R^\ast, D_R^\ast) \lesssim R^{-k/d}$.
Since $\Phi$ is Lipschitz with $\Phi(\signalClass_{\banachOne}) \supset \signalClass$, the remark after
Lemma~\ref{lem:ApproximationRateTransference} shows that there exists a codec
$\code = \big( (E_R, D_R) \big)_{R \in \N} \in \codecs{L^2(\Omega)}$
satisfying $\distortion{L^2(\Omega)} (E_R, D_R) \lesssim R^{-k/d}$ as well,
as claimed in Part~\ref{enu:SobolevCriticalCodecExists} of Theorem~\ref{thm:SobolevPhaseTransition}.
\hfill$\square$

\section{Measurability of Besov and Sobolev balls}
\label{sec:SignalClassesMeasurable}

In this subsection, we show for the range of parameters considered in
Theorems~\ref{thm:mainbesovresult} and \ref{thm:SobolevPhaseTransition}
that the balls $\ball \big( 0, R; B^\BesovSmoothness_{p,q}(\Omega) \big)$
and $\ball \big( 0, R; W^{k,p}(\Omega) \big)$ are measurable subsets of $L^2(\Omega)$.
We remark that for the case where $p,q \in (1,\infty)$, easier proofs than the ones given here
are possible.
Yet, since the proofs for the cases where $p \in \{ 1,\infty \}$ or $q \in \{ 1,\infty \}$
apply verbatim for a whole range of exponents, we prefer to state and prove the
more general results.

We begin with the case of Besov spaces, for which the balls are in fact closed.

\begin{lem}\label{lem:BesovBallsMeasurable}
  Let $\emptyset \neq \Omega \subset \R^d$ be open and bounded and let $p,q \in (0,\infty]$
  and $\BesovSmoothness \in \R$ with $\BesovSmoothness > d \cdot (p^{-1} - 2^{-1})_{+}$.
  Then $B^{\BesovSmoothness}_{p,q}(\Omega) \hookrightarrow L^2(\Omega)$, and
  the balls $\ball(0, R; B^{\BesovSmoothness}_{p,q}(\Omega)) \subset L^2(\Omega)$
  are closed for all $R > 0$.
\end{lem}

\begin{proof}
  Let $p_0 := \max \{ p, 2 \}$.
  Then \cite[Example~7.2]{DecompositionIntoSobolev} shows that
  $B^{\BesovSmoothness}_{p,q} (\R^d) \hookrightarrow L^{p_0}(\R^d)$,
  since $p \leq p_0$ and since $\BesovSmoothness > d \cdot (p^{-1} - p_0^{-1})$
  by our assumptions on $\BesovSmoothness$.
  This implies $B^{\BesovSmoothness}_{p,q}(\Omega) \hookrightarrow L^2(\Omega)$,
  since if $f \in B^{\BesovSmoothness}_{p,q}(\Omega)$, then by definition of this space
  there exists some ${g \in B^{\BesovSmoothness}_{p,q}(\R^d)}$ satisfying $f = g|_{\Omega}$ and
  $\| g \|_{B_{p,q}^{\BesovSmoothness}(\R^d)} \leq 2 \| f \|_{B_{p,q}^{\BesovSmoothness}(\Omega)}$,
  and hence
  \[
    \| f \|_{L^2(\Omega)}
    \lesssim \| f \|_{L^{p_0}(\Omega)}
    = \big\| g|_{\Omega} \big\|_{L^{p_0}(\Omega)}
    \leq \| g \|_{L^{p_0}(\R^d)}
    \lesssim \| g \|_{B^{\BesovSmoothness}_{p,q}(\R^d)}
    \leq 2 \, \| f \|_{B^{\BesovSmoothness}_{p,q}(\Omega)} .
  \]

  It remains to show that $\ball(0, R; B^{\BesovSmoothness}_{p,q}(\Omega)) \subset L^2(\Omega)$
  is closed.
  To see this, first note that if $(g_n)_{n \in \N} \subset B_{p,q}^{\BesovSmoothness}(\R^d)$
  satisfies $g_n \to g \in \Schwartz'(\R^d)$ with convergence in $\Schwartz'(\R^d)$, then
  \(
    \| g \|_{B^{\BesovSmoothness}_{p,q} (\R^d)}
    \leq \liminf_{n \to \infty}
           \| g_n \|_{B^{\BesovSmoothness}_{p,q}(\R^d)} .
  \)
  Indeed, with the family $(\varphi_j)_{j \in \N_0} \subset \Schwartz(\R^d)$
  used in the definition of Besov spaces (see Section~\ref{sub:BesovFourierDefinition}),
  we have for $f \in \Schwartz'(\R^d)$ and $x \in \R^d$ that
  \(
    \Fourier^{-1}(\varphi_j \cdot \widehat{f} \,) (x)
    = (2\pi)^{-d/2}
      \cdot \big\langle
              \widehat{f}, \quad
              e^{i \langle x, \bullet \rangle} \varphi_j
            \big\rangle_{\Schwartz', \Schwartz} ;
  \)
 see for instance \cite[Theorem~7.23]{RudinFA}.
 From this, we easily see that
 $\Fourier^{-1}(\varphi_j \cdot \widehat{g_n}) \to \Fourier^{-1}(\varphi_j \cdot \widehat{g})$,
 with pointwise convergence as $n \to \infty$.
 Therefore, Fatou's lemma shows that
 \(
   \| \Fourier^{-1}(\varphi_j \cdot \widehat{g}) \|_{L^p}
   \leq \liminf_{n \to \infty}
          \| \Fourier^{-1}(\varphi_j \cdot \widehat{g_n}) \|_{L^p} .
 \)
 By another application of Fatou's lemma, we therefore see
 \begin{align*}
   \| g \|_{B^{\BesovSmoothness}_{p,q} (\R^d)}
   & = \Big\|
        \Big(
          2^{\BesovSmoothness j}
          \| \Fourier^{-1} (\varphi_j \cdot \widehat{g}) \|_{L^p}
        \Big)_{j \in \N_0}
      \Big\|_{\ell^q}
    \leq \liminf_{n \to \infty}
           \Big\|
             \Big(
               2^{\BesovSmoothness j}
               \| \Fourier^{-1} (\varphi_j \cdot \widehat{g_n}) \|_{L^p}
             \Big)_{j \in \N_0}
           \Big\|_{\ell^q} \\
   & =   \liminf_{n \to \infty}
           \| g_n \|_{B^{\BesovSmoothness}_{p,q} (\R^d)} ,
 \end{align*}
 as claimed.

 Now we prove the claimed closedness.
 Let $(f_n)_{n \in \N} \subset \ball(0, R; B_{p,q}^{\BesovSmoothness}(\Omega)) \subset L^2(\Omega)$
 such that $f_n \to f \in L^2(\Omega)$ with convergence in $L^2(\Omega)$.
 By definition of $B_{p,q}^\tau (\Omega)$, for each $n \in \N$ there is
 $g_n \in B_{p,q}^{\BesovSmoothness}(\R^d)$ with
 \(
   \| g_n \|_{B_{p,q}^{\BesovSmoothness}(\R^d)}
   \leq (1 + \frac{1}{n}) \| f_n \|_{B_{p,q}^{\BesovSmoothness}(\Omega)}
   \leq (1 + \frac{1}{n}) R
   \leq 2 R
 \)
 and $f_n = g_n |_\Omega$.

 As seen above, $B_{p,q}^{\BesovSmoothness}(\R^d) \hookrightarrow L^{p_0}(\R^d)$,
 so that $(g_n)_{n \in \N} \subset L^{p_0}(\R^d) = (L^{p_0'}(\R^d))'$ is bounded,
 where $p_0' \leq 2 < \infty$, so that $L^{p_0'}(\R^d)$ is separable.
 Thus, \cite[Theorem~8.5]{AltFA} shows that there is a subsequence $(g_{n_k})_{k \in \N}$
 and some $g \in L^{p_0}(\R^d)$ such that $g_{n_k} \to g$ in the weak-$\ast$-sense in
 $L^{p_0}(\R^d) = (L^{p_0'}(\R^d) )'$.
 In particular, $g_{n_k} \to g$ in $\Schwartz'(\R^d)$.
 By what we showed above, this implies
 \(
   \| g \|_{B_{p,q}^{\BesovSmoothness}(\R^d)}
   \leq \liminf_{k \to \infty} \| g_{n_k} \|_{B_{p,q}^{\BesovSmoothness}(\R^d)}
   \leq R .
 \)
 Finally, we have for any $\varphi \in C_c^\infty (\Omega)$ that
 \(
   \langle g, \varphi \rangle
   = \lim_{k \to \infty} \langle g_{n_k}, \varphi \rangle
   = \lim_{k \to \infty} \langle f_{n_k}, \varphi \rangle
   = \langle f, \varphi \rangle ,
 \)
 since $f_{n_k} = g_{n_k}|_{\Omega}$ and $f_{n_k} \to f$ in $L^2(\Omega)$.
 Overall, we thus see that $f = g|_{\Omega} \in B_{p,q}^{\BesovSmoothness}(\Omega)$ and
 $\| f \|_{B_{p,q}^{\BesovSmoothness}(\Omega)} \leq \| g \|_{B_{p,q}^{\BesovSmoothness}(\R^d)} \leq R$.
\end{proof}

For the Sobolev spaces $W^{k,p}(\Omega)$ with $p = 1$,
the set $\ball \big( 0, R; W^{k,1}(\Omega) \big)$ is not closed in $L^2(\Omega)$.
In order to show that this ball is nonetheless Borel measurable,
we begin with the following result on $\R^d$.

\begin{lem}\label{lem:SobolevSpaceMeasurableOnWholeSpace}
  Let $d,k \in \N$ and $p \in [1,2]$.
  Then $L^2(\R^d) \cap W^{k,p}(\R^d)$ is a Borel-measurable subset of $L^2(\R^d)$.
\end{lem}

\begin{proof}
  Let $\varphi \in C_c^\infty (\R^d)$ with $\varphi \geq 0$ and $\int_{\R^d} \varphi(x) \, d x = 1$,
  and define $\varphi_n (x) := n^{d} \cdot \varphi(n x)$.
  It follows from \cite[Section~4.13]{AltFA} that if $f \in L^2(\R^d)$, then
  $\varphi_n \ast f \in L^2(\R^d) \cap C^\infty(\R^d)$
  with $\partial^\alpha (\varphi_n \ast f) = (\partial^\alpha \varphi_n) \ast f$.

  \smallskip{}

  \textbf{Step 1:} Define $\signalClass := L^2(\R^d) \cap W^{k,p}(\R^d)$.
  In this step, we show that
  \[
    \signalClass = \big\{
                     f \in L^2(\R^d)
                     \quad\colon\quad
                     \forall \, |\alpha| \leq k :
                       \big( (\partial^\alpha \varphi_n) \ast f \big)_{n \in \N}
                       \text{ is Cauchy in } L^p(\R^d)
                   \big\} .
  \]
  For ``$\subset$'', note that if $f \in \signalClass$,
  then from the definition of the weak derivative we see
  \begin{align*}
    [ (\partial^\alpha \varphi_n) \ast f ](x)
    & = \int_{\R^d}
          f(y) \cdot (\partial^\alpha \varphi_n)(x - y)
        \, d y
      = (-1)^{|\alpha|}
        \int_{\R^d}
          f(y) \cdot \partial^\alpha_y [\varphi_n (x-y)]
        \, d y \\
    & = \int_{\R^d}
          \partial^\alpha f (y) \cdot \varphi_n (x-y)
        \, d y
      = [\varphi_n \ast (\partial^\alpha f)](x),
  \end{align*}
  so that \cite[Theorem~4.15]{AltFA} shows that
  $(\partial^\alpha \varphi_n) \ast f \xrightarrow[n\to\infty]{} \partial^\alpha f$,
  with convergence in $L^p (\R^d)$.
  This proves ``$\subset$''.

  For ``$\supset$'', let $f \in L^2(\R^d)$ such that
  $\big( (\partial^\alpha \varphi_n) \ast f \big)_{n \in \N}$
  is Cauchy in $L^p (\R^d)$ for each $\alpha \in \N_0^d$ with $|\alpha| \leq k$.
  Define $g_\alpha := \lim_{n\to\infty} [(\partial^\alpha \varphi_n) \ast f] \in L^p(\R^d)$
  for $|\alpha| \leq k$.
  Since \cite[Theorem~4.15]{AltFA} shows that $\varphi_n \ast f \to f$ with convergence in $L^2$,
  we get $f = g_0 \in L^p(\R^d)$.
  Furthermore, as seen above, we have $\varphi_n \ast f \in C^\infty (\R^d)$ with
  $\partial^\alpha (\varphi_n \ast f) = (\partial^\alpha \varphi_n) \ast f$.
  Therefore, we see for arbitrary $\psi \in C_c^\infty (\R^d)$ and $\alpha \in \N_0^d$
  with $|\alpha| \leq k$ that
  \begin{align*}
    \int_{\R^d}
      f \cdot \partial^\alpha \psi
    \, d x
    & = \lim_{n \to \infty}
        \int_{\R^d}
          (\varphi_n \ast f) \cdot \partial^\alpha \psi
        \, d x
      = \lim_{n \to \infty}
        (-1)^{|\alpha|}
        \int_{\R^d}
          [(\partial^\alpha \varphi_n) \ast f] \cdot \psi
        \, d x \\
    ({\scriptstyle{\text{since } \psi \in C_c^\infty \subset L^{p'}}})
    & = (-1)^{|\alpha|}
        \int_{\R^d}
          g_\alpha \cdot \psi
        \, d x ,
  \end{align*}
  which shows that $g_\alpha$ is the $\alpha$-th weak derivative of $f$; that is,
  $\partial^\alpha f = g_\alpha \in L^p(\R^d)$.
  Since this holds for all $|\alpha| \leq k$, we see that $f \in W^{k,p}(\R^d)$
  and thus $f \in \signalClass$.

  \medskip{}

  \textbf{Step 2:}
  For $n,m,M \in \N$, define
  \[
    \Gamma_{n,m,M} :
    L^2(\R^d) \to [0,\infty),
    f \mapsto \big\|
                [\partial^\alpha (\varphi_n - \varphi_m) \ast f] \cdot \indicator_{[-M,M]^d}
              \big\|_{L^p} .
  \]
  Since $p \leq 2$, it is easy to see that $\Gamma_{n,m,M}$ is well-defined and continuous.
  Furthermore,
  \(
    \| [(\partial^\alpha \varphi_n) \ast f] - [(\partial^\alpha \varphi_m) \ast f] \|_{L^p}
    = \sup_{M \in \N}
        \Gamma_{n,m,M}(f),
  \)
  which---together with the result from Step~1---implies that
  \[
    \signalClass
    = \bigcap_{\ell = 1}^\infty
        \bigcup_{N=1}^\infty
          \bigcap_{n,m=N}^\infty
            \bigcap_{M=1}^\infty
              \big\{ f \in L^2(\R^d) \colon \Gamma_{n,m,M}(f) \leq 1/\ell \big\}
  \]
  is a Borel-measurable subset of $L^2(\R^d)$.
\end{proof}

We can now prove a similar result on bounded domains.
For the convenience of the reader, we recall that
$\| f \|_{W^{k,p}} = \max_{|\alpha| \leq k} \| \partial^\alpha f \|_{L^p}$;
see Equation~\eqref{eq:SobolevNormDefinition}.

\begin{lem}\label{lem:SobolevSpaceMeasurableOnDomain}
  Let $p \in [1,\infty]$, $k \in \N$, $R \in (0,\infty)$,
  and let $\Omega \subset \R^d$ be open and bounded.
  In case of $p = 1$, assume additionally that $\Omega$ is a Lipschitz domain.

  Then ${L^2(\Omega) \cap \ball \big( 0, R; W^{k,p}(\Omega) \big)}$
  is a Borel-measurable subset of $L^2(\Omega)$.
\end{lem}

\begin{proof}
  \textbf{Step 1:}
  The space $C^k (\overline{\Omega})$ (with the norm
  ${\| f \|_{C^k(\overline{\Omega})} = \max_{|\alpha| \leq k} \| \partial^\alpha f \|_{\sup}}$)
  is separable; see \cite[Section~4.18]{AltFA}.
  Since subsets of separable spaces are separable, there exists a sequence
  $(\varphi_n)_{n \in \N} \subset C_c^\infty(\Omega) \setminus \{ 0 \}$
  that is dense in $C_c^\infty(\Omega) \setminus \{ 0 \}$ with respect to
  $\| \bullet \|_{C^k(\overline{\Omega})}$.
  For $n \in \N$, define
  \[
    \gamma_n : L^2(\Omega) \to [0,\infty),
    f \mapsto \max_{|\alpha| \leq k}
                \left|
                  \int_{\Omega}
                    f \cdot \partial^\alpha \varphi_n
                  \, d x
                \right|
              \Big/
              \| \varphi_n \|_{L^{p'}},
  \]
  where $p' \in [1,\infty]$ is the conjugate exponent to $p$.
  Since $\partial^\alpha \varphi_n \in C_c^\infty(\Omega) \subset L^2(\Omega)$,
  we see that $\gamma_n$ is continuous, so that
  \(
    \gamma : L^2(\Omega) \to [0,\infty],
             f \mapsto \sup_{n \in \N} \gamma_n (f)
  \)
  is Borel measurable.

  \medskip{}

  \textbf{Step 2:}
  We claim that
  $|\int_\Omega f \cdot \partial^\alpha \varphi \, d x| \leq \gamma(f) \cdot \| \varphi \|_{L^{p'}}$
  for all $f \in L^2(\Omega)$, $\varphi \in C_c^\infty(\Omega)$, and $|\alpha| \leq k$.
  Clearly, we can assume without loss of generality that $\gamma(f) < \infty$ and $\varphi \neq 0$.
  Thus, there is a subsequence $(n_\ell)_{\ell \in \N}$ such that
  $\| \varphi - \varphi_{n_\ell} \|_{C^k(\overline{\Omega})} \to 0$,
  which easily implies $\| \varphi_{n_\ell} \|_{L^{p'}} \to \| \varphi \|_{L^{p'}}$ and
  $\partial^\alpha \varphi_{n_\ell} \to \partial^\alpha \varphi$
  with convergence in $L^2(\Omega)$ for all $|\alpha| \leq k$.
  Hence,
  \(
    |\int_\Omega f \cdot \partial^\alpha \varphi \, d x|
    = \lim_{\ell \to \infty}
        |\int_\Omega f \cdot \partial^\alpha \varphi_{n_\ell} \, d x|
    \leq \lim_{\ell \to \infty}
           \gamma(f) \cdot \| \varphi_{n_\ell} \|_{L^{p'}}
    =   \gamma(f) \cdot \| \varphi \|_{L^{p'}},
  \)
  as claimed.

  \medskip{}

  \textbf{Step 3:} In this step, we prove for $p > 1$ that
  $\signalClass := L^2(\Omega) \cap \ball (0, R ; W^{k,p}(\Omega))$
  satisfies $\signalClass = \{ f \in L^2(\Omega) \colon \gamma(f) \leq R \}$,
  which then implies that $\signalClass$ is a Borel measurable subset of $L^2(\Omega)$.

  First, if $f \in \signalClass$, then
  \(
    |\int_\Omega f \, \partial^\alpha \varphi_n \, d x|
    = |\int_\Omega \varphi_n \, \partial^\alpha f \, d x|
    \leq \| \partial^\alpha f \|_{L^p} \cdot \| \varphi_n \|_{L^{p'}}
    \leq R \cdot \| \varphi_n \|_{L^{p'}}
  \)
  for all $|\alpha| \leq k$ and $n \in \N$, so that $\gamma(f) \leq R$.

  Conversely, if $\gamma(f) \leq R$, then
  Step~2 shows for arbitrary $|\alpha| \leq k$ and $\varphi \in C_c^\infty(\Omega)$ that
  \(
    \strut
    |\int_\Omega f \cdot \partial^\alpha \varphi \, d x|
    \leq \gamma(f) \cdot \| \varphi \|_{L^{p'}}
    \leq R \cdot \| \varphi \|_{L^{p'}},
  \)
  so that \cite[Section~E6.7]{AltFA} implies that $f \in W^{k,p}(\Omega) \strut$;
  this uses our assumption $p > 1$.
  Finally, for $\varphi \in C_c^\infty (\Omega)$ and $|\alpha| \leq k$, we have
  \(
    |\int_\Omega \varphi \cdot \partial^\alpha f \, d x|
    = |\int_\Omega f \cdot \partial^\alpha \varphi \, d x|
    \leq R \cdot \| \varphi \|_{L^{p'}} .
  \)
  Therefore, \cite[Corollary~6.13]{AltFA} shows
  $\| \partial^\alpha f \|_{L^p} \leq R$ for all $|\alpha| \leq k$.
  By our definition of $\| \bullet \|_{W^{k,p}(\Omega)}$
  (see Equation~\eqref{eq:SobolevNormDefinition}), this implies $f \in \signalClass$.

  \medskip{}

  \textbf{Step 4:} We prove the claim for the case $p = 1$.
  Since $\Omega$ is a Lipschitz domain, \mbox{\cite[Theorem~5 in Chapter~VI]{SteinSingularIntegrals}}
  yields a linear extension operator ${E : L^1(\Omega) \to L^1(\R^d)}$
  satisfying $(E f)|_\Omega = f$ for all $f \in L^1(\Omega)$,
  and such that for arbitrary $\ell \in \N_0$ and $q \in [1,\infty]$
  the restriction $E : W^{\ell,q}(\Omega) \to W^{\ell,q}(\R^d)$ is well-defined and bounded.
  In particular, $E : L^2(\Omega) \to L^2(\R^d)$ is continuous and hence measurable.
  By Lemma~\ref{lem:SobolevSpaceMeasurableOnWholeSpace}, this means that
  $\Theta := \{ f \in L^2(\Omega) \colon E f \in W^{k,1}(\R^d) \} \subset L^2(\Omega)$
  is measurable.
  We claim that
  \({
    \signalClass
    := L^2(\Omega) \cap \ball (0,R; W^{k,1}(\Omega))
    = \Theta \cap \{ f \in L^2(\Omega) \colon \gamma(f) \leq R \}
    ,
  }\)
  which then implies that $\signalClass \subset L^2(\Omega)$ is measurable.

  For ``$\subset$'', we see as in Step~3 that $\gamma(f) \leq R$ if $f \in \signalClass$.
  Furthermore, by the properties of the extension operator $E$, we also have $f \in \Theta$
  if $f \in \signalClass$.
  For ``$\supset$'', let $f \in \Theta$ satisfy $\gamma(f) \leq R$.
  Since $f \in \Theta$, we have $f = (E f)|_\Omega \in W^{k,1}(\Omega)$.
  One can then argue as at the end of Step~3 (using \cite[Corollary~6.13]{AltFA})
  to see that $f \in \ball(0, R; W^{k,1}(\Omega))$ and thus $f \in \signalClass$.
\end{proof}





\section{Proof of the lower bounds for neural network approximation}%
\label{sec:NNLowerBoundProofs}

We begin by explaining the connection between rate distortion theory and approximation
by neural networks.
This is based on the observation from \cite{boelcskeiNeural,PetersenVoigtlaenderOptimalReLUApproximation}
that one can use the existence of approximating networks to construct a codec for a function class.
This in turn relies on the fact that neural networks ca be encoded as bit strings,
as described in the following result.

\begin{lem}\label{lem:NNEncoding}
  Let $\varrho : \R \to \R$ with $\varrho(0) = 0$.
  For $d,\sigma,W \in \N$, let $\mathcal{NN}_{d,W}^{\sigma,\varrho}$
  be as defined in Theorem~\ref{thm:IntroductionNNResult}.
  Then there exists an injective map
  \[
    \Gamma_{d,W}^{\sigma,\varrho} :
    \mathcal{NN}_{d,W}^{\sigma,\varrho}
    \to \{0,1\}^{C_0 \, \sigma \cdot W \cdot \lceil \log_2 (1+W) \rceil^2},
  \]
  where $C_0 = C_0(d) \in \N$ is a universal constant.
\end{lem}

\begin{proof}
  This follows from \cite[Lemma~B.4]{PetersenVoigtlaenderOptimalReLUApproximation},
  once we note that $|\Z \cap [a,b]| \leq 1 + (b-a)$ for $a \leq b$, and hence
  \begin{align*}
    & \Big|
        \big[ - W^{\sigma \lceil \log_2 W \rceil}, W^{\sigma \lceil \log_2 W \rceil} \big]
        \cap 2^{-\sigma \lceil \log_2 W \rceil^2} \Z
      \Big| \\
    & = \Big|
          \Z \cap
          2^{\sigma \lceil \log_2 W \rceil^2}
          \big[ - W^{\sigma \lceil \log_2 W \rceil}, W^{\sigma \lceil \log_2 W \rceil} \big]
        \Big| \\
    & \leq 1 + 2 \cdot W^{\sigma \lceil \log_2 W \rceil} \, 2^{\sigma \lceil \log_2 W \rceil^2} \\
    & \leq 1 + 2 \cdot 2^{2 \sigma \lceil \log_2 W \rceil^2}
      =:   \theta ,
  \end{align*}
  so that there is a surjection
  \(
    B : \{ 0, 1 \}^{M}
        \to \big[ \!
              - W^{\sigma \lceil \log_2 W \rceil}, W^{\sigma \lceil \log_2 W \rceil}
            \big]
            \cap 2^{-\sigma \lceil \log_2 W \rceil^2} \Z
  \)
  as soon as ${M \geq \lceil \log_2 \theta \rceil}$.
  Since $\sigma \in \N$, this holds for
  ${M := 5\sigma \lceil \log_2(1 + W) \rceil^2}$.
\end{proof}

The precise connection to rate distortion theory is established by the following lemma.

\begin{lem}\label{lem:NNRateDistortion}
  Let $\Omega \subset \R^d$ be measurable, let $\varrho : \R \to \R$ be measurable with $\varrho(0) = 0$,
  and let $\sigma \in \N$.
  For $f \in L^2(\Omega)$ and $\eps \in (0, 1)$, let $W^{\sigma,\varrho}_\eps (f) \in \N \cup \{ \infty \}$
  be as defined in Theorem~\ref{thm:IntroductionNNResult}.
  For $\tau > 0$ define
  \[
    \mathcal{A}_{\mathcal{NN},\sigma,\varrho}^\tau
    := \big\{
         f \in L^2(\Omega)
         \quad\colon\quad
         \exists \, C > 0 \quad
           \forall \, \eps \in (0, 1) : \quad
             W_{\eps}^{\sigma,\varrho} (f) \leq C \cdot \eps^{-\tau}
       \big\} .
  \]
  Then there is a codec $\code = \code(\sigma, \varrho, \Omega) \in \codecs<L^2(\Omega)>{L^2(\Omega)}$
  such that
  \[
    \mathcal{A}_{\mathcal{NN},\sigma,\varrho}^\tau
    \subset \approxClass{\tau^{-1} - \delta}<L^2(\Omega)>{L^2(\Omega)} (\code)
    \qquad \forall \,\,\, \delta \in (0, \tau^{-1}) .
  \]
\end{lem}

\begin{proof}
  \textbf{Step 1:} (\emph{Constructing the codec $\code$}):
  Let $C_0 \in \N$ as in Lemma~\ref{lem:NNEncoding}.
  For $R \geq C_0 \, \sigma$, let $W_R \in \N$ be maximal with
  $C_0 \sigma \cdot W_R \cdot \lceil \log_2 (1+W_R) \rceil^2 \leq R$.
  By Lemma~\ref{lem:NNEncoding}, there exists a surjection
  $D_R : \{ 0,1 \}^R \to \mathcal{NN}_{d,W_R}^{\sigma,\varrho}$.
  For each $f \in L^2(\Omega)$, choose $\cvec(R,f) \in \{ 0,1 \}^R$ such that
  \[
    \| f - D_R(\cvec (R, f)) \|_{L^2(\Omega)}
    = \min_{\cvec \in \{ 0,1 \}^R}
        \| f - D_R(\cvec) \|_{L^2(\Omega)}
    = \min_{g \in \mathcal{NN}_{d,W_R}^{\sigma,\varrho}}
        \| f - g \|_{L^2(\Omega)}
    ,
  \]
  and define ${E_R : L^2(\Omega) \to \{ 0,1 \}^R, f \mapsto \cvec(R,f)}$.

  Finally, for $R < C_0 \, \sigma$, define
  \[
    E_R : L^2(\Omega) \to \{0,1\}^R, f \mapsto (0,\dots,0)
    \qquad \text{and} \qquad
    D_R : \{ 0,1 \}^R \to L^2(\Omega), \cvec \mapsto 0 .
  \]

  \medskip{}

  \textbf{Step 2:} (\emph{Completing the proof}):
  Let $\delta \in (0,\tau^{-1})$ and $f \in \mathcal{A}_{\mathcal{NN},\sigma,\varrho}^\tau$,
  so that there is $C = C(f) > 0$ satisfying $W^{\sigma,\varrho}_\eps (f) \leq C \cdot \eps^{-\tau}$
  for all $\eps \in (0,1)$.

  Since the logarithm grows slower than any positive power and since the maximality of $W_R$
  implies that $R \leq C_0 \, \sigma \cdot (W_R + 1) \cdot \lceil \log_2 (W_R + 2) \rceil^2$,
  it is easy to see that there is $C_1 = C_1(\tau,\delta,d,\sigma) > 0$
  such that $R \leq C_1 \cdot W_R^{1/(1 - \tau \delta)}$ for all $R \in \N_{\geq C_0 \sigma}$.
  Note that if $R$ is large enough, then
  $\eps := C^{1/\tau} \cdot C_1^{(1 - \delta \tau)/\tau} \cdot R^{-(\tau^{-1} - \delta)}$
  satisfies $\eps \in (0,1)$.
  For these $R$, we thus get
  \[
    W^{\sigma,\varrho}_\eps(f)
    \leq C \cdot \eps^{-\tau}
    \leq C_1^{-(1 - \delta \tau)} \cdot R^{1 - \delta \tau}
    \leq W_R .
  \]
  By definition of $W^{\sigma,\varrho}_\eps(f)$ and by choice of $D_R, E_R$,
  we therefore see for all sufficiently large $R \in \N$ that
  \begin{align*}
    \bigl\| f - D_R(E_R(f)) \bigr\|_{L^2(\Omega)}
    & = \min_{g \in \mathcal{NN}_{d,W_R}^{\sigma,\varrho}}
          \| f - g \|_{L^2(\Omega)} \\
    & \leq \min_{g \in \mathcal{NN}_{d,W^{\sigma,\varrho}_\eps (f)}^{\sigma,\varrho}}
             \| f - g \|_{L^2(\Omega)}
      \leq \eps
      =    C^{1/\tau} \cdot C_1^{(1 - \delta \tau) / \tau} \cdot R^{-(\tau^{-1} - \delta)} ,
  \end{align*}
  which easily implies that $f \in \approxClass{\tau^{-1} - \delta}<L^2(\Omega)>{L^2(\Omega)} (\code)$.
  Since $f \in \mathcal{A}_{\mathcal{NN},\sigma,\varrho}^\tau$ and $\delta \in (0,\tau^{-1})$
  were arbitrary, we are done.
\end{proof}

\begin{proof}[Proof of Theorem~\ref{thm:IntroductionNNResult}]
  \textbf{Part 1:} Let $s > s^\ast$.
  The measure $\P$ from Theorem~\ref{thm:IntroductionSobolevBesovTransition} is critical for
  $\signalClass$ with respect to $L^2(\Omega)$, so that we have
  $\P \big( \signalClass \cap \ball(f,\eps; L^2(\Omega)) \big) \leq 2^{-c \cdot \eps^{-1/s}}$
  for all $\eps \in (0,\eps_0)$ and suitable $c, \eps_0 > 0$; see Equation~\eqref{eq:CriticalMeasureDefinition}.
  Lemma~\ref{lem:NNEncoding} easily implies that
  $|\mathcal{NN}_{d,W}^{\sigma,\varrho}| \leq 2^{C_0 \sigma W \lceil \log_2(1+W) \rceil^2}$
  for all $W \in \N$ and a suitable $C_0 = C_0(d)$.
  Thus, setting $C := C_0 \, \sigma$, we get
  \[
    \mathrm{Pr}
    \Big(
      \min_{g \in \mathcal{NN}_{d,W}^{\sigma,\varrho}}
        \| f - g \|_{L^2(\Omega)}
      \leq \eps
    \Big)
    = \P \Big(
           \bigcup_{g \in \mathcal{NN}_{d,W}^{\sigma,\varrho}}
             \ball(g,\eps; L^2(\Omega))
         \Big)
    \leq 2^{C W \lceil \log_2(1+W) \rceil^2}
         2^{-c \cdot \eps^{-1/s}}
  \]
  for all $\eps \in (0,\eps_0)$.
  This is precisely what is claimed in Part~1 of Theorem~\ref{thm:IntroductionNNResult}.

  \medskip{}

  \noindent
  \textbf{Part 2:}
  For $\ell,\sigma \in \N$, define
  \[
    \mathcal{A}_{\ell,\sigma}
    :=
    \bigl\{
      f \in \signalClass
      \quad \colon \quad
      \exists \, C > 0 \quad
        \forall \, \eps \in (0,1): \quad
          W_\eps^{\sigma,\varrho}(f) \leq C \cdot \eps^{-(1 - \frac{1}{2\ell}) / s^\ast}
    \bigr\} .
  \]
  It is not hard to see that
  \(
    \mathcal{A}_{\mathcal{NN},\varrho}^\ast \!
    \subset \bigcup_{\sigma \in \N}
              \bigcup_{\ell \in \N}
                \mathcal{A}_{\ell,\sigma},
  \)
  so that it suffices to show ${\P^\ast (\mathcal{A}_{\ell,\sigma}) = 0}$ for all $\sigma,\ell \in \N$.
  To see this, let $\code \in \codecs<L^2(\Omega)>{L^2(\Omega)}$ as in Lemma~\ref{lem:NNRateDistortion},
  and note with the notation of that lemma and with $\delta := \frac{s^\ast / 2}{2\ell - 1}$ that
  \[
    \mathcal{A}_{\ell,\sigma}
    = \signalClass \cap \mathcal{A}_{\mathcal{NN},\sigma,\varrho}^{(1 - \frac{1}{2\ell}) / s^\ast}
    \subset \signalClass \cap \mathcal{A}_{L^2(\Omega),L^2(\Omega)}^{\frac{s^\ast}{1- 1/(2\ell)} - \delta} (\code)
    =       \mathcal{A}_{\signalClass,L^2(\Omega)}^{s^\ast \cdot \frac{2 \ell - 1/2}{2\ell - 1}} (\code),
  \]
  where Theorem~\ref{thm:IntroductionSobolevBesovTransition} shows that
  \(
    \P^\ast
    \Bigl(
      \mathcal{A}_{\signalClass,L^2(\Omega)}^{s^\ast \cdot \frac{2 \ell - 1/2}{2\ell - 1}} (\code)
    \Bigr)
    = 0.
  \)
\end{proof}

We close this section by proving Remark~\ref{rem:NNResultSharpness}.

\begin{proof}[Proof of Remark~\ref{rem:NNResultSharpness}]
  \textbf{Case~1 (Besov spaces):}
  Here, we have $\signalClass = \ball(0,1; B_{p,q}^{\BesovSmoothness} (\Omega; \R))$
  and $s^\ast = \BesovSmoothness / d$, where $\BesovSmoothness > d \cdot (\frac{1}{p} - \frac{1}{2})_+$.
  By definition of the space $B_{p,q}^{\BesovSmoothness} (\Omega; \R)$, each $f \in \signalClass$
  extends to a function $\tilde{f} \in B_{p,q}^{\BesovSmoothness} (\R^d; \R)$ satisfying
  $\| \tilde{f} \|_{B_{p,q}^{\BesovSmoothness} (\R^d; \R)} \leq 2$.
  Thanks to \mbox{\cite[Theorem in Section~2.5.12]{TriebelTheoryOfFunctionSpaces1}},
  this implies for a suitable $C_1 = C_1(d,p,q,\BesovSmoothness) > 0$
  that $\| f \|_{B_{p,q}^{\BesovSmoothness,\ast} (\Omega)} \leq C_1$,
  where $\| \cdot \|_{B_{p,q}^{\BesovSmoothness,\ast}(\Omega)}$ is the norm on the Besov space
  used in \cite{SuzukiBesovNNApproximation}.

  \smallskip{}

  Now, \cite[Proposition~1]{SuzukiBesovNNApproximation} yields $C_2, C_3, N_0, \theta \in \N_{\geq 2}$
  (all depending only on $d,p,q,\BesovSmoothness$) such that for every $f \in \signalClass$
  and $N \geq N_0$, there exists a network ${\Phi = \Phi(f,N)}$ satisfying
  $\| f - R_\varrho (\Phi) \|_{L^2(\Omega)} \leq \frac{C_2}{2} \cdot N^{- s^\ast}$
  as well as $L(\Phi) \leq C_3 \cdot \log_2 N$ and ${W(\Phi) \leq C_3 \cdot N \cdot \log_2 N}$,
  and such that all weights of $\Phi$ have absolute value at most $C_3 \cdot N^{\theta}$.
  This almost implies the desired estimate; the main issue is that the weights are merely
  bounded, but not necessarily quantized.
  To fix this, we will use \cite[Lemma~VI.8]{GrohsNNApproximationTheory}.

  To make this formal, let us assume in what follows that $\eps \in (0,\frac{1}{2}) \cap (0, C_2 \cdot N_0^{-s^\ast})$;
  it is easy to see that this implies the claim of Remark~\ref{rem:NNResultSharpness}
  for general $\eps \in (0, \frac{1}{2})$.
  Let ${N \in \N}$ be minimal with $C_2 \cdot N^{-s^\ast} \leq \eps$, noting that this entails
  $N \geq N_0 \geq 2$ as well as
  ${\eps \leq C_2 \cdot (N-1)^{-s^\ast} \leq 2^{s^\ast} C_2 \cdot N^{-s^\ast}}$,
  and therefore $N \leq 2 \, C_2^{1/s^\ast} \eps^{-1/s^\ast}$.
  Now, given $f \in \signalClass$, choose $\Phi = \Phi (f, N)$ as above and note that
  $\| f - R_\varrho (\Phi) \|_{L^2(\Omega)} \leq \frac{\eps}{2}$.

  Define $W := \lceil C_3 \cdot N \cdot \log_2 N \rceil \geq N$ and choose
  $k = k(p,q,d,\BesovSmoothness) \in \N$ with ${k \geq \frac{\theta}{s^\ast}}$ so large that
  \(
    C_3 \cdot \bigl(2 C_2^{1/s^\ast}\bigr)^2
    \leq C_3 \cdot \bigl(2 C_2^{1/s^\ast}\bigr)^\theta
    \leq 2^k .
  \)
  Since $\log_2 N \leq N$ 
  we then see that
  \({
    W
    \leq C_3 N^2
    \leq C_3 \cdot\! \bigl(2 \, C_2^{1/s^\ast} \eps^{-1/s^\ast}\bigr)^2
    \!\leq\! (\frac{\eps}{2})^{-k}
  }\)
  and
  \({
    C_3 \!\cdot\! N^\theta
    \leq C_3 \cdot \bigl(2 \, C_2^{1/s^\ast} \eps^{-1/s^\ast}\bigr)^\theta
    \leq (\frac{\eps}{2})^{-k} .
  }\)
  Therefore, \cite[Lemma~VI.8]{GrohsNNApproximationTheory} produces a network $\Phi '$ satisfying
  ${W(\Phi') \leq W(\Phi) \leq W}$ and
  ${\strut \| R_\varrho (\Phi) - R_\varrho (\Phi') \|_{L^\infty ([0,1]^d)} \leq \frac{\eps}{2}}$
  and such that
  \[
    \text{all weights of $\Phi'$ belong to }
    [-(\tfrac{\eps}{2})^{-\sigma_0}, (\tfrac{\eps}{2})^{-\sigma_0}]
    \cap 2^{- \sigma_0 \lceil \log_2(2/\eps) \rceil} \Z,
    \text{ where }
    \sigma_0 := 3 k L(\Phi) .
  \]
  To see that this implies the claim, first note that
  $\strut \| f - R_\varrho (\Phi') \|_{L^2(\Omega)} \leq \eps$ and that
  \[
    W(\Phi')
    \leq W
    \leq 2 C_3 \, N \, \log_2 N
    \leq 4 C_2^{\frac{1}{s^\ast}} C_3
         \cdot \eps^{-\frac{1}{s^\ast}}
         \cdot \log_2 \bigl(2 C_2^{\frac{1}{s^\ast}} \eps^{-\frac{1}{s^\ast}}\bigr)
    \leq C_4 \cdot \eps^{-\frac{1}{s^\ast}} \log_2 (1/\eps)
  \]
  for a suitable constant $C_4 = C_4 (d,p,q,\BesovSmoothness)$.
  Regarding the quantization, first define $\sigma_1 := 3 k C_3$, so that
  $\sigma_0 = 3 k L(\Phi) \leq 3 k C_3 \cdot \log_2 N \leq \sigma_1 \lceil \log_2 W \rceil$.
  Next, observe that ${2/\eps \leq C_2 / \eps \leq N^{s^\ast} \leq W^{s^\ast}}$,
  meaning
  $(\eps/2)^{-\sigma_0} \leq (W^{s^\ast})^{\sigma_0} \leq W^{\sigma \lceil \log_2 W \rceil}$
  with ${\sigma := \sigma_1 \, \lceil s^\ast \rceil}$.
  Furthermore, we have
  \(
    \strut
    \log_2 (2/\eps)
    \leq \log_2 (W^{s^\ast})
    \leq \lceil s^\ast \rceil \, \lceil \log_2 W \rceil
  \),
  which easily implies that
  \({
    \strut
    2^{-\sigma_0 \lceil \log_2 (2/\eps) \rceil} \Z
    \subset 2^{-\sigma_0 \lceil s^\ast \rceil \lceil \log_2 W \rceil} \Z
    \subset 2^{-\sigma \lceil \log_2 W \rceil^2} \Z .
  }\)
  Overall, this shows that $\Phi'$ is $(\sigma,W)$-quantized, so that
  $\Phi' \in \mathcal{NN}_{d,W}^{\sigma,\varrho}$.
  Because of $\| f - R_\varrho (\Phi') \|_{L^2 (\Omega)} \leq \eps$, this implies that
  $W_\eps^{\sigma,\varrho}(f) \leq W \leq C_4 \cdot \eps^{-1/s^\ast} \log_2(1/\eps)$,
  which is what we wanted to show.

  \medskip{}

  \textbf{Case~2 (Sobolev spaces):}
  Set $p_1 := \min \{ p, 2 \}$ and note
  $\signalClass \subset \signalClass' := \ball(0,1; W^{k,p_1}(\Omega))$.
  Since $\Omega = [0,1]^d$ is a Lipschitz domain,
  \cite[Chapter VI, Theorem~5]{SteinSingularIntegrals} shows that each
  ${f \in \signalClass'}$ extends to a function
  $\tilde{f} \in W^{k,p_1}(\R^d)$ with $\| \tilde{f} \|_{W^{k,p_1}(\R^d)} \leq C_1$,
  where ${C_1 = C_1(d,p,k)}$.
  Now, Lemma~\ref{lem:IntegrableSobolevEmbedsIntoBesov} shows that
  $\tilde{f} \in B^k_{p_1,\infty}(\R^d)$ with $\| \tilde{f} \|_{B^{k}_{p_1, \infty}(\R^d)} \leq C_2$
  where ${C_2 = C_2 (d,p,k)}$.
  Overall, this easily implies $\signalClass \subset \ball(0,C_2; B^k_{p_1, \infty}(\Omega; \R))$,
  so that the claim follows from that for the Besov spaces.
  Here, we implicitly used that the condition $k > d \cdot (\frac{1}{p} - \frac{1}{2})_+$
  holds if and only if $k > d \cdot (\frac{1}{p_1} - \frac{1}{2})_+$.
\end{proof}

\section{Technical results concerning the optimal compression rate}
\label{sec:SingleElementProofs}

In the introduction, it was claimed that if the signal class $\signalClass \subset \banach$ is
closed and convex, then---for each codec $\code = \big( (E_R, D_R) \big)_{R \in \N} \in \codecs{\banach}$
and each $s > \optRateSmall{\banach}$---one can find a \emph{single} signal $x \in \signalClass$
on which the codec $\code$ does not attain the rate $s$.
The following proposition shows that even a slightly stronger statement holds:
Given $\code$, one can find a ``badly encodable'' $\xvec = \xvec(\code) \in \signalClass$
such that $\xvec$ is not encoded at any rate $s > \optRateSmall{\banach}$ by the codec $\code$.

\begin{prop}\label{prop:SingleHardToEncodeElement}
  Let $\banach$ be a Banach space and let $\emptyset \neq \signalClass \subset \banach$.
  Assume that either
  \begin{enumerate}
    \item $\signalClass$ is closed, bounded, and convex; or

    \item $\signalClass = \{ \xvec \in \banach \colon \| \xvec \|_{\ast} \leq r \}$
          for some $r \in (0,\infty)$ and a map $\| \cdot \|_\ast : \banach \to [0,\infty]$
          with the following properties:
          \begin{enumerate}
            \item $\| \cdot \|_\ast$ is a \emph{quasi-norm}; that is, there exists $\kappa \geq 1$
                  such that ${\| \alpha \, \xvec \|_{\ast} = |\alpha| \cdot \| \xvec \|_\ast}$ and
                  $\| \xvec + \yvec \|_{\ast} \leq \kappa \cdot (\| \xvec \|_{\ast} + \| \yvec \|_{\ast})$
                  for all $\alpha \in \mathbb{R}$ and $\xvec,\yvec \in \banach$;

            \item there is $C \geq 1$ satisfying $\| \xvec \|_{\banach} \leq C \cdot \| \xvec \|_{\ast}$
                  for all $\xvec \in \banach$;

            \item $\signalClass \subset \banach$ is closed;

            \item $\| \cdot \|_\ast$ is continuous ``with respect to itself'', meaning that
                  $\| \xvec \|_\ast \to \| \xvec_0 \|_\ast$ if $\| \xvec - \xvec_0 \|_\ast \to 0$.
          \end{enumerate}
  \end{enumerate}
  Set $s^\ast := \optRateSmall{\banach}$.
  Then, for each codec $\code = \big( (E_R, D_R) \big)_{R \in \N} \in \codecs{\banach}$
  there is some $\xvec = \xvec(\code) \in \signalClass$ such that for each $\ell \in \N$,
  we have
  \[
    \big\| \xvec - D_R (E_R (\xvec) ) \big\|_{\banach} \geq R^{-(s^\ast + \ell^{-1})}
    \qquad \text{for infinitely many} \quad R \in \N .
  \]
\end{prop}

\begin{rem*}
  1) In particular, we see for each $s > s^\ast$ (by choosing $\ell \in \N$ such that $s^\ast + \ell^{-1} < s$)
  that $R^s \cdot \| \xvec - D_R(E_R(\xvec)) \|_{\banach} \geq R^{s - s^\ast - \ell^{-1}} \to \infty$
  for some sequence $R = R_N \to \infty$.
  Therefore,
  \[
    \xvec \in \signalClass \setminus \bigcup_{s > s^\ast} \approxClass{s}{\banach} (\code) .
  \]

  \medskip{}

  2) The assumptions on the quasi-norm $\| \cdot \|_{\ast}$ might appear quite technical,
  but they are usually satisfied.
  Indeed, the condition $\| \xvec \|_{\banach} \leq C \cdot \| \xvec \|_{\ast}$ is equivalent to
  $\signalClass \subset \banach$ being bounded, which is necessary for having
  $\optRateSmall{\banach} > 0$.
  Next, most naturally appearing quasi-norms are \emph{$q$-norms} for some $q \in (0,1]$,
  meaning that $\| \xvec + \yvec \|_{\ast}^q \leq \| \xvec \|_\ast^q + \| \yvec \|_\ast^q$.
  In this case, it is not hard to see
  $\big| \| \xvec \|_\ast^q - \| \yvec \|_\ast^q \big| \leq \| \xvec - \yvec \|_\ast^q$,
  which implies that $\| \cdot \|_\ast$ is ``continuous with respect to itself''.
  Finally, most natural quasi-norms will satisfy a \emph{Fatou property},
  in the sense that if $\xvec_n \to \xvec$ in $\banach$,
  then $\| \xvec \|_\ast \leq \liminf_{n\to\infty} \| \xvec_n \|_{\ast}$.
  If this is the case, then $\signalClass \subset \banach$ is closed.
\end{rem*}

\begin{proof}
  \textbf{Step 1:} (\emph{Setup for applying the Baire category theorem}).
  Let us assume towards a contradiction that the claim does \emph{not} hold.
  Define $M_R := \mathrm{range}(D_R) \subset \banach$.
  Then for each $\xvec \in \signalClass$ there exist $n_\xvec, \ell_\xvec \in \N$ satisfying
  ${\| \xvec - D_R(E_R (\xvec)) \|_{\banach} < R^{-(s^\ast + \ell_\xvec^{-1})}}$ for all $R \geq n_\xvec$.
  Thus, it is not hard to see that
  \[
    \dist_{\banach} (\xvec, M_R)
    \leq \| \xvec - D_R (E_R (\xvec)) \|_{\banach}
    \leq N_\xvec \cdot R^{-(s^\ast + \ell_{\xvec}^{-1})}
    \qquad \forall \, R \in \N ,
  \]
  where we defined
  \(
    N_\xvec
    := 1 + \max \{
                  k^{s^\ast + \ell_\xvec^{-1}} \cdot \| \xvec - D_k (E_k (\xvec)) \|_{\banach}
                  \colon
                  1 \leq k \leq n_\xvec
                \}
  \).

  Thus, if we define
  \[
    \mathcal{G}_{N,\ell} := \big\{
                              \xvec \in \signalClass
                              \,\, \colon \,\,
                              \forall \, R \in \N :
                                \dist_{\banach} (\xvec, M_R) \leq N \cdot R^{-(s^\ast + \ell^{-1})}
                            \big\}
    \qquad \text{for } N,\ell \in \N,
  \]
  then $\signalClass = \bigcup_{N,\ell \in \N} \mathcal{G}_{N,\ell}$.
  Furthermore, since $\dist_{\banach} (\cdot, M_R)$ is continuous, it is not hard to see
  that each set $\mathcal{G}_{N,\ell} \subset \signalClass$ is closed.
  Finally, $\signalClass \subset \banach$ is a closed set, and hence a complete metric space
  (equipped with the metric induced by $\|\cdot\|_{\banach}$).
  Thus, the Baire category theorem (\cite[Theorem~5.9]{FollandRA})
  shows that there are certain $N,\ell \in \N$ such that the relative interior $\mathcal{G}_{N,\ell}^\circ$
  of $\mathcal{G}_{N,\ell}$ in $\signalClass$ satisfies $\mathcal{G}_{N,\ell}^\circ \neq \emptyset$.

  \medskip{}

  \textbf{Step 2:} (\emph{Proving that there are $\xvec_0' \in \banach$ and $t > 0$
  satisfying $\xvec_0' + t \signalClass \subset \mathcal{G}_{N,\ell}$}).
  We distinguish the two cases regarding the assumptions on $\signalClass$.

  \emph{Case 1: $\signalClass$ is convex.}
  Choose $\xvec_0 \in \mathcal{G}_{N,\ell}^\circ \subset \signalClass$ and note that there is
  some $\eps \in (0, 1)$ such that $\signalClass \cap \ball(\xvec_0, \eps; \banach) \subset \mathcal{G}_{N,\ell}$.
  Next, since $\signalClass$ is bounded, there is some $C \geq 1$ satisfying
  $\signalClass \subset \ball(\mathbf{0}, C; \banach)$.
  Let us define $\xvec_0' := (1 - t) \xvec_0 \in \banach$ where $t := \frac{\eps}{2 C}$,
  noting that $t \in (0,1)$.
  With these choices, we see for arbitrary $\xvec \in \signalClass$
  that $\xvec_0' + t \xvec \in \signalClass$ by convexity, and furthermore
  \[
    \big\| \xvec_0 - \big( \xvec_0' + t \, \xvec \big) \big\|_{\banach}
    = t \cdot \| \xvec_0 - \xvec \|_{\banach}
    \leq 2 t C \leq \eps .
  \]
  Thus,
  \(
    \xvec_0' + t \signalClass
    \subset \signalClass \cap \ball(\xvec_0, \eps; \banach)
    \subset \mathcal{G}_{N,\ell}
  \).

  \smallskip{}

  \emph{Case 2: $\signalClass = \{ \xvec \in \banach \colon \| \xvec \|_\ast \leq r \}$.}
  Choose $\xvec_0 \in \mathcal{G}_{N,\ell}^\circ \subset \signalClass$ and note that there is
  some $\eps \in (0,1)$ such that $\signalClass \cap \ball(\xvec_0, \eps; \banach) \subset \mathcal{G}_{N,\ell}$.
  With $C \geq 1$ as in Part~2b) of the assumptions of the proposition, let
  $0 < \sigma < \frac{\eps}{2 \kappa C (1 + r)}$ and define $\xvec_0' := (1 - \sigma) \, \xvec_0$,
  noting that $\| \xvec_0' \|_\ast < r$.
  By continuity of $\| \cdot \|_\ast$ with respect to itself,
  there is some $0 < \delta < \frac{\eps \min\{ 1, r\} }{2 \kappa C}$ such that
  $\| \xvec_0' + \yvec \|_\ast < r$ for all $\yvec \in \banach$ satisfying $\| \yvec \|_\ast \leq \delta$.
  Define $t := \frac{\delta}{r}$.
  For arbitrary $\yvec \in \signalClass$, we then have $\| t \, \yvec \|_\ast \leq t r = \delta$,
  and hence $\xvec_0' + t \, \yvec \in \signalClass$.
  Furthermore,
  \(
    \| (\xvec_0 ' + t \, \yvec) - \xvec_0 \|_{\banach}
    \leq C \cdot \| - \sigma \xvec_0 + t \, \yvec \|_{\ast}
    \leq \kappa \sigma C r + \kappa C \delta
    \leq \eps .
  \)
  Overall, we have shown that
  \(
    \xvec_0' + t \signalClass
    \subset \signalClass \cap \ball(\xvec_0, \eps; \banach)
    \subset \mathcal{G}_{N,\ell}
    ,
  \)
  as desired.

  \medskip{}

  \textbf{Step 3:} (\emph{Completing the proof})
  For each $\xvec \in \signalClass$, we have $\xvec_0' + t \, \xvec \in \mathcal{G}_{N,\ell}$,
  and therefore $\dist_{\banach} \big( \xvec_0' + t \, \xvec , M_R \big) \leq N \cdot R^{-(s^\ast + \ell^{-1})}$
  for all $R \in \N$.
  Because of $M_R = \mathrm{range} (D_R)$, this implies that there is $c_{\xvec,R} \in \{0,1\}^R$
  satisfying $\| (\xvec_0' + t \xvec) - D_R(c_{\xvec,R}) \|_{\banach} \leq N \cdot R^{-(s^\ast + \ell^{-1})}$.
  Now, we define a new codec
  $\widetilde{\code} = \big( (\widetilde{E}_R, \widetilde{D}_R) \big)_{R \in \N}$ by
  \[
    \widetilde{E}_R : \signalClass \to \{0,1\}^R ,
                      \xvec \mapsto c_{\xvec,R}
    \qquad \text{and} \qquad
    \widetilde{D}_R : \{0,1\}^R \to \banach ,
                      c \mapsto t^{-1} \cdot \big( D_R (c) - \xvec_0' \big) .
  \]
  For arbitrary $\xvec \in \signalClass$, we then see
  \begin{align*}
    \big\| \xvec - \widetilde{D}_R \big( \widetilde{E}_R (\xvec) \big) \big\|_{\banach}
    & =    t^{-1} \cdot \big\| t \, \xvec - \big( D_R(c_{\xvec,R}) - \xvec_0' \big) \big\|_{\banach} \\
    & =    t^{-1} \cdot \big\| (\xvec_0' + t \, \xvec) - D_R (c_{\xvec,R}) \big\|_{\banach}
      \leq \frac{N}{t} \cdot R^{-(s^\ast + \ell^{-1})}
  \end{align*}
  for all $R \in \N$.
  By definition of the optimal exponent, this implies
  $s^\ast = \optRateSmall{\banach} \geq s^\ast + \ell^{-1}$, which is the desired contradiction.
\end{proof}

As the second result in this appendix, we show that the preceding property does \emph{not} hold for general
compact sets $\signalClass \subset \banach$, even if $\banach = \hilbert$ is a Hilbert space.
In other words, some additional regularity assumption---like convexity---is necessary to
ensure the property stated in Proposition~\ref{prop:SingleHardToEncodeElement}.

\begin{ex}\label{exa:SingleElementExistenceCounterExample}
  We consider the Hilbert space $\hilbert := \ell^2(\N)$, where we denote the standard orthonormal
  basis of this space by $(\evec_n)_{n \in \N}$.
  Fix $s > 0$, define $\xvec_0 := \mathbf{0} \in \ell^2 (\N)$
  and $\xvec_n := (\log_2(n+1))^{-s} \cdot \evec_n \in \ell^2(\N)$ for $n \in \N$, and finally set
  \[
    \signalClass := \{ \xvec_n \colon n \in \N_0 \} .
  \]
  We claim that $\optRateSmall = s$, but that there is a codec
  $\code = \big( (E_R, D_R) \big)_{R \in \N} \in \codecs$ such that
  $\approxClass{\sigma}(\code) = \signalClass$ for every $\sigma > 0$;
  that is, \emph{every} element $\xvec \in \signalClass$ is compressed by $\code$
  with \emph{arbitrary} rate $\sigma > 0$.

  To prove $\optRateSmall \leq s$, let $R \in \N$ and $(E,D) \in \endec$.
  By the pigeonhole-principle, there are $n,m \in \{1, \dots, 2^R + 1\}$ satisfying $n \neq m$
  but $E (\xvec_n) = E(\xvec_m)$.
  By symmetry, we can assume that $n < m$, so that $n+1 \leq 2^R + 1 \leq 2^{R+1}$.
  Therefore,
  \begin{align*}
    2^{-s} \cdot R^{-s}
    & \leq (R+1)^{-s}
      \leq (\log_2 (n+1))^{-s}
      \leq \| \xvec_n - \xvec_m \|_{\ell^2} \\
    & \leq \| \xvec_n - D(E(\xvec_n)) \|_{\ell^2} + \| D(E(\xvec_m)) - \xvec_m \|_{\ell^2}
      \leq 2 \, \distortion (E, D) .
  \end{align*}
  Since this holds for any encoder/decoder pair $(E,D) \in \endec$ and arbitrary $R \in \N$,
  we see $\optRateSmall \leq s$.

  Next, we construct the codec $\code$ mentioned above.
  To do so, for each $n \in \N$, fix a bijection $\kappa_n : \{0, \dots, 2^n - 1\} \to \{0,1\}^n$,
  and define
  \begin{alignat*}{5}
    E_n : \quad & \signalClass \to \{0,1\}^n, \quad
                && \xvec_m \mapsto \begin{cases}
                                     \kappa_n (m), & \text{if } m \leq 2^n - 1 , \\
                                     \kappa_n (0), & \text{otherwise},
                                   \end{cases} \\
    D_n : \quad & \{0,1\}^n \to \signalClass , \quad
                && \theta \mapsto \xvec_{\kappa_n^{-1} (\theta)} .
  \end{alignat*}
  For $m \in \N_0$ with $m \leq 2^n - 1$, we then have
  $D_n (E_n (\xvec_m)) = \xvec_{\kappa_n^{-1} (\kappa_n (m))} = \xvec_m$,
  while if $m \geq 2^n$, then
  $D_n (E_n (\xvec_m)) = \xvec_{\kappa_n^{-1} (\kappa_n(0))} = \xvec_0 = \mathbf{0}$, and hence
  \[
    \| \xvec_m - D_n (E_n (\xvec_m) ) \|_{\ell^2}
    = \| \xvec_m \|_{\ell^2}
    = \big( \log_2 (m+1) \big)^{-s}
    \leq n^{-s} .
  \]
  Therefore, $\| \xvec - D_n (E_n (\xvec)) \|_{\ell^2} \leq n^{-s}$ for all $\xvec \in \signalClass$,
  and thus $\distortion (E_n , D_n) \leq n^{-s}$, so that $\optRateSmall \geq s$.

  We have now proved that $\optRateSmall = s$.
  Finally, it is easy to see that given \emph{arbitrary} $\sigma > 0$,
  the codec $\code = \big( (E_R, D_R) \big)_{R \in \N}$ constructed above approximates
  each \emph{fixed} $\xvec \in \signalClass$ with rate $\sigma$.
  Indeed, for $m \in \N$ and $\xvec = \xvec_m$, we have
  \begin{align*}
    \| \xvec - D_n (E_n (\xvec)) \|_{\ell^2}
    & = \begin{cases}
          \big( \log_2 (m+1) \big)^{-s}, & \text{if } n < \log_2 (m+1), \\
          0           , & \text{if } n \geq \log_2(m+1)
        \end{cases} \\
    & \leq \big( \log_2 (m+1) \big)^\sigma \cdot n^{-\sigma}
      =:   C_{\xvec,\sigma} \cdot n^{-\sigma}
  \end{align*}
  for all $n \in \N$, while for $\xvec = \xvec_0$ we have
  $\| \xvec - D_n (E_n (\xvec)) \|_{\ell^2} = 0$ for all $n \in \N$.
  \hfill$\blacktriangleleft$
\end{ex}

\section{Technical results concerning sequence spaces}
\label{sec:TechnicalSequenceSpaceProofs}

\begin{proof}[Proof of Lemma~\ref{lem:ElementaryEmbeddings}]
  Let $\xvec \in \R^{\indexSet}$.
  Set $u_m := m^{\theta + \vartheta} \cdot 2^{\alpha m} \cdot \| \xvec_m \|_{\ell^p(\indexSet_m)}$
  and $v_m := m^{-\vartheta}$, and observe that
  $\| (u_m)_{m \in \N} \|_{\ell^q} = \| \xvec \|_{\generalSpace{\theta+\vartheta}{q}}$
  and $\| (v_m \cdot u_m)_{m \in \N} \|_{\ell^r} = \| \xvec \|_{\generalSpace{\theta}{r}}$.

  Let us first consider the case $q < \infty$.
  In this case, $\frac{q}{r} \in (1,\infty)$ and $\frac{q}{q-r} \in (1,\infty)$ are conjugate
  exponents, so that Hölder's inequality shows
  \begin{align*}
    \big\| (v_m \cdot u_m)_{m \in \N} \big\|_{\ell^r}
    & = \big\| (v_m^r \cdot u_m^r)_{m \in \N} \big\|_{\ell^1}^{1/r}
      \leq \Big(
             \big\| (v_m^r)_{m \in \N} \big\|_{\ell^{q/(q-r)}}
             \cdot \big\| (u_m^r)_{m \in \N} \big\|_{\ell^{q/r}}
           \Big)^{1/r} \\
    & = \big\| (v_m)_{m \in \N} \big\|_{\ell^{rq / (q-r)}}
        \cdot \big\| (u_m)_{m \in \N} \big\|_{\ell^q} .
  \end{align*}
  Here, we note that $\vartheta \cdot \frac{r q}{q - r} = \frac{\vartheta}{r^{-1} - q^{-1}} > 1$
  and $v_m = m^{-\vartheta}$, so that ${\kappa := \| (v_m)_{m \in \N} \|_{\ell^{rq / (q-r)}}}$
  is finite.

  Finally, in case of $q = \infty$, simply note that
  \[
    \big\| (v_m \, u_m)_{m \in \N} \big\|_{\ell^r}
    \leq \big\| (m^{-\vartheta})_{m \in \N} \big\|_{\ell^r}
         \cdot \big\| (u_m)_{m \in \N} \big\|_{\ell^q} ,
  \]
  where now $\kappa := \| (m^{-\vartheta})_{m \in \N} \|_{\ell^r}$ is finite,
  since $\vartheta > \frac{1}{r} - \frac{1}{q} = \frac{1}{r}$.
\end{proof}

\begin{proof}[Proof of Lemma~\ref{lem:ProductSigmaAlgebraLarge}]
  By definition of the product $\sigma$-algebra, each of the \emph{fi\-nite-di\-men\-sion\-al} projections
  $\pi_m : \R^{\indexSet} \to \R^{\indexSet_m}, \xvec \mapsto \xvec_m$ is measurable.
  Since $\| \cdot \|_{\ell^p (\indexSet_m)}$ is continuous on $ \R^{\indexSet_m}$
  and hence Borel measurable,
  \({
     q_m : \R^{\indexSet} \to [0,\infty),
           \xvec \mapsto 2^{\alpha m} \, m^{\theta} \, \| \xvec_m \|_{\ell^p(\indexSet_m)}
  }\)
  is $\calB_{\indexSet}$-measurable for each $m \in \N$.

  In case of $q < \infty$, this implies that the map
  \[
    \R^{\indexSet} \to [0,\infty],
    \xvec \mapsto \| \xvec \|_{\generalSpace{\theta}{q}}^q = \sum_{m = 1}^\infty [q_m (\xvec)]^q
  \]
  is $\calB_{\indexSet}$-measurable as a countable series of measurable, non-negative functions,
  and hence so is $\xvec \mapsto \| \xvec \|_{\generalSpace{\theta}{q}}$.
  If $q = \infty$, the (quasi) norm
  $ \| \cdot \|_{\generalSpace{\theta}{\infty}} = \sup_{m \in \N} q_m$
  is $\calB_{\indexSet}$-measurable as a countable supremum of $\calB_{\indexSet}$-measurable,
  non-negative functions.

  \medskip{}

  For proving the final claim, let us write $\calT := \ell^2(\indexSet) \Cap \calB_{\indexSet}$
  for brevity.
  By the first part of the lemma,
  ${\| \cdot \|_{\ell^2} = \| \cdot \|_{\generalSpace{0}{2}{0}<2>}} : \R^{\indexSet} \to [0,\infty]$
  is $\calB_{\indexSet}$-measurable.
  Furthermore, for arbitrary $\xvec \in \R^{\indexSet}$ the translation
  $\R^{\indexSet} \to \R^{\indexSet}, \yvec \mapsto \yvec + \xvec$ is $\calB_{\indexSet}$-measurable.
  These two observations imply that the norm $\| \cdot \|_{\ell^2} : \ell^2(\indexSet) \to [0,\infty)$
  and the translation operator $\ell^2(\indexSet) \to \ell^2(\indexSet), \yvec \mapsto \yvec + \xvec$
  are $\mathcal{T}$-measurable for any $\xvec \in \ell^2(\indexSet)$.
  This implies that
  $B_r (\xvec) = \{ \yvec \in \ell^2(\indexSet) \colon \| \yvec + (-\xvec) \|_{\ell^2} < r \}$
  is $\mathcal{T}$-measurable.
  But $\ell^2(\indexSet)$ is separable, so that every open set is a countable union of open balls;
  therefore, it follows that $\calB_{\ell^2} \subset \mathcal{T}$.
  Conversely, $\mathcal{T}$ is generated by sets of the form
  $\{ \xvec \in \ell^2(\indexSet) \colon p_i(\xvec) \in M \}$,
  where $M \subset \R$ is a Borel set and
  $p_i : \R^{\indexSet} \to \R, (x_j)_{j \in \indexSet} \mapsto x_i$.
  Since $p_i |_{\ell^2(\indexSet)} : \ell^2(\indexSet) \to \R$ is continuous with respect to
  $\| \cdot \|_{\ell^2(\indexSet)}$, we see that each generating set of $\mathcal{T}$
  also belongs to $\calB_{\ell^2}$, which completes the proof.
\end{proof}

\bibliographystyle{amsplain}

{\scriptsize
\bibliography{critbib}
}

\end{document}